\documentclass[11pt]{article}

\usepackage[style=alphabetic, 
maxbibnames=99, 
backend=biber, 
doi=false, isbn=false, url=false]{biblatex}
\usepackage[oldstyle]{libertine}
 
\usepackage[english]{babel}
\usepackage{amsmath}
\usepackage{amsthm}
\usepackage{amssymb}

\usepackage[all,cmtip]{xy}
\usepackage{tikz}
\usepackage{tikz-cd}
\usepackage{import}

\usepackage[utf8]{inputenc}

\usepackage{csquotes}

\usepackage{geometry}
\usepackage{hyperref}

\usepackage{dsfont}
\usepackage{wasysym}

\usepackage[all,cmtip]{xy}
\usepackage{tikz}
\usepackage{tikz-cd}

\usepackage{mathtools}

\usepackage{braket}
\usepackage{stmaryrd}% for \rrbracket

\usepackage{color}

\usepackage[utf8]{inputenc}

\usepackage{lastpage}

\usepackage{hyperref}

\newtheorem{thm}{Theorem}[section]
\newtheorem*{thm*}{Theorem}
\newtheorem{theorem}{Theorem}

\newtheorem*{question*}{Question}

\newtheorem*{obs*}{Observation}
\newtheorem{proposition}[thm]{Proposition}
\newtheorem{fact}[thm]{Fact}
\newtheorem{lem}[thm]{Lemma}
\newtheorem{cor}[thm]{Corollary}

\theoremstyle{definition}

\newtheorem{defn}[thm]{Definition}
\newtheorem*{defn*}{Definition}

\theoremstyle{remark}
\newtheorem{rem}[thm]{\textsc{Remark}}

\newtheorem{example}[thm]{\textsc{Example}}

\newtheorem{warning}[thm]{Warning}

\newcommand{\IN}{\mathbb{N}}
\newcommand{\IZ}{\mathbb{Z}}
\newcommand{\IQ}{\mathbb{Q}}
\newcommand{\IR}{\mathbb{R}}

\newcommand{\IS}{\mathbb{S}}
\newcommand{\IT}{\mathbb{T}}

\newcommand{\CP}{\mathbb{CP}}

\newcommand{\Top}{{\mit{T\!op}}}

\newcommand{\Diff}{{\op{Diff}}}
\newcommand{\Emb}{\mi{Emb}}

\newcommand{\TC}{\mi{TC}}

\newcommand{\Sets}{{\mathit{Set}}}

\newcommand{\op}[1]{\operatorname{#1}}
\newcommand{\id}{\op{id}}

\newcommand{\pr}{\op{pr}}

\newcommand{\Map}{\op{Map}}
\newcommand{\Hom}{\op{Hom}}

\newcommand{\colim}{\op{colim}}

\newcommand{\GL}{\op{GL}}
\newcommand{\sign}{\op{sign}}
\newcommand{\Bun}{{\op{Bun}}}

\newcommand{\Fin}{\mi{Fin}}

\newcommand{\eps}{\varepsilon}
\newcommand{\ga}{\alpha}
\newcommand{\gb}{\beta}
\newcommand{\gd}{\delta}
\newcommand{\gD}{\Delta}
\newcommand{\gc}{\gamma}
\newcommand{\gC}{\Gamma}
\newcommand{\gi}{\iota}
\newcommand{\gk}{\kappa}
\newcommand{\gl}{\lambda}
\newcommand{\gL}{\Lambda}
\newcommand{\gs}{\sigma}
\newcommand{\gS}{\Sigma}
\newcommand{\gt}{\theta}

\newcommand{\gO}{\Omega}
\newcommand{\gp}{\varphi}

\newcommand{\mi}[1]{\mathit{#1}}
\newcommand{\mrm}[1]{\mathrm{#1}}
\newcommand{\mc}[1]{\mathcal{#1}}
\newcommand{\mcA}{\mc{A}}
\newcommand{\mcB}{\mc{B}}
\newcommand{\mcC}{\mc{C}}
\newcommand{\mcD}{\mc{D}}
\newcommand{\mcE}{\mc{E}}
\newcommand{\mcF}{\mc{F}}

\newcommand{\mcO}{\mc{O}}

\newcommand{\mcZ}{\mc{Z}}
\newcommand{\cl}[1]{\overline{#1}}
\newcommand{\ul}[1]{\underline{#1}}
\newcommand{\wt}[1]{\widetilde{#1}}
\newcommand{\wh}[1]{\widehat{#1}}
\newcommand{\doublebs}{\!\mathbin{/\mkern-6mu/}}

\newcommand{\inj}{\hookrightarrow}
\newcommand{\surj}{\twoheadrightarrow}

\newcommand{\ot}{\otimes}
\newcommand{\cd}{\bullet}

\newcommand{\qand}{\quad \text{and} \quad}
\newcommand{\blank}{\underline{\ \ }}
\newcommand{\gle}[1]{\langle #1 \rangle}

\newcommand{\Conf}{\mrm{Conf}}

\newcommand{\Cob}{\mathcal{C}}
\newcommand{\Cobcl}{\Cob^{\mathrm{cl}}}
\newcommand{\Cobred}{\Cob^{\mathrm{red}}}

\newcommand{\Ccl}{C^{\mathrm{cl}}}

\newcommand{\Cut}{\mrm{Cut}}

\newcommand{\Mor}{\mrm{Mor}}
\newcommand{\Obj}{\mrm{Obj}}

\renewcommand{\Emb}{\mrm{Emb}}

\newcommand{\cube}{(-1,1)^\infty}

\newcommand{\dlgk}{\cl{\gk}}

\newcommand{\oh}{\tfrac{1}{2}}

\title{
The classifying space of the one-dimensional bordism category and 
a cobordism model for TC of spaces
}

\author{Jan Steinebrunner}

\begin{document}

\maketitle

\begin{abstract}
    The homotopy category of the bordism category
    $h\Cob_d$ has as objects closed oriented $(d-1)$-manifolds and as morphisms
    diffeomorphism classes of $d$-dimensional bordisms.
    Using a new fiber sequence for bordism categories,
    we compute the classifying space of $h\Cob_d$ for $d=1$,
    exhibiting it as a circle bundle over $\gO^{\infty-2}\CP^\infty_{-1}$.
    
    As part of our proof we construct a quotient $\Cobred_1$
    of the cobordism category where circles are deleted.
    We show that this category has classifying space 
    $\gO^{\infty-2}\CP^\infty_{-1}$ 
    and moreover that, if one equips these bordisms with a map to a 
    simply connected space $X$, the resulting 
    $\Cobred_1(X)$ can be thought of as a cobordism model 
    for the topological cyclic homology $TC(\IS[\gO X])$.
    
    In the second part of the paper we construct
    an infinite loop space map $B(h\Cobred_1) \to Q(\gS^2 \CP^\infty_+)$
    in this model 
    and use it to derive combinatorial formulas for 
    rational cocycles on $h\Cobred_1$ representing 
    the Miller-Morita-Mumford classes $\gk_i \in H(B(h\Cob_1);\IQ)$.
\end{abstract}

%--------------------------------------------------------------------
%--------------------------------------------------------------------
%--------------------------------------------------------------------

\section{Introduction}

Bordism categories play an important role in organising 
geometry, mathematical physics, and algebra.
Their most prominent use is probably
as the main ingredient to Atiyah and Segal's definition
of topological and conformal field theory.
More recently, cobordism categories have been a key tool 
in the study moduli spaces of manifolds, 
see for instance \cite{GMTW06} and \cite{GRW14}.
This approach uses a more sophisticated version of the cobordism category: 
namely, a topological category $\Cob_d$ where for a diffeomorphism class 
$[W]$ there is not just one, but ``$B \Diff^\partial(W)$-many" morphisms.
In other words, one needs to consider the $(\infty,1)$-category
of cobordisms $\mrm{Bord}_{\gle{d-1,d}}^{\mrm{or}}$,
for which $\Cob_d$ is a specific model, 
rather than just its homotopy category $\mrm{Cob}_d^{\mrm{or}} \simeq h\Cob_d$.
From the perspective of topological field theories
$\Cob_d$ has also proven to be the more natural object of study.
For example, Lurie's proof-sketch of the Baez-Dolan cobordism hypothesis
\cite{Lur09}, crucially relies on the presence of higher structure.

In each of the above settings the classifying space $B\Cob_d$ 
plays an important role.
When studying moduli spaces, $B\Cob_d$ controls universal characteristic
classes for manifolds (\cite{MW07}, \cite{GRW17}), 
and when studying field theories $B\Cob_d$ classifies the invertible ones
(e.g.\ \cite{FH20}).
Generally, the classifying space construction also allows one to build
categorical models for interesting homotopy types.
The homotopy type of $B\Cob_d$ was completely determined 
by Galatius, Madsen, Tillmann, and Weiss, who in \cite{GMTW06} showed
that it is the infinite loop space of a certain Thom-spectrum $MTSO_d$.

This 
raises the question
for a similar computation of $B(h\Cob_d)$. 
One might expect this to be a simpler problem 
since the homotopy category $h\Cob_d$ 
is only an ordinary, rather than a topological category.
However, very little is known about $B(h\Cob_d)$.
Except for the simple observation $B(h\Cob_0) \simeq S^1 \times S^1$,
the only result in this direction is Tillmann's splitting \cite{Til96}
of the surface case: 
$B(h\Cob_2) \simeq S^1 \times X$ 
for some mysterious simply connected infinite loop space $X$.

Our main theorem is a complete computation of $B(h\Cob_1)$,
which is already far more complex than $B(\Cob_1) \simeq QS^0$.
We also prove a similar result for the unoriented case 
(Theorem \ref{thm:BhCob1-unor}).

\begin{theorem}
\label{theorem:1}
    There is an equivalence of infinite loop spaces
    \[
        B(h\Cob_1) \simeq \gO^{\infty-2}\left(\mrm{hofib}(MTSO_2 \to H\IZ) \right)
    \]
    where $MTSO_2 \to H\IZ$ is the map of spectra classifying
    the generator of $H^0(MTSO_2) \cong \IZ$.
    Consequently, the rational cohomology of the 
    identity component $B(h\Cob_1)_0 \subset B(h\Cob_1)$ is given by
    \[
        H^*(B(h\Cob_1)_0; \IQ) \cong 
        \IQ[\dlgk_1, \dlgk_2, \dlgk_3, \dots]
    \]
    where the generators are of degree $|\dlgk_i| = 2i+2$
    and we have $\pi_0 B(h\Cob_1) = \IZ$.
\end{theorem}

\subsection*{The reduction fiber sequence}

In order to prove \autoref{theorem:1} we construct a fiber sequence,
which we believe to be of independent interest.
Just like Genauer's sequence \cite{Gen12} for bordisms with horizontal boundary,
this fiber sequence relates the classifying spaces 
of three bordism categories, which we now describe.

Every $d$-dimensional cobordism $W:M \to N$ decomposes canonically 
as a disjoint union $W = c(W) \amalg r(W)$ where $c(W)$ is a closed
manifold and $r(W)$ is \emph{reduced}, i.e.\ 
it has no closed components.
We define two new topologically enriched bordism categories:
the closed bordism category is the full subcategory 
$\Cobcl_d \subset \Cob_d$ on the object $\emptyset$, 
and the reduced bordism category is the quotient category $\Cob_d \to \Cobred_d$
where two bordisms $W$ and $V$ are identified if $r(W) = r(V)$.
Every morphism in $\Cobred_d$ can be uniquely represented by 
a reduced bordism.
We refer the reader to definition \ref{defn:Cobcl/red} for a complete definition,
and to Figure \ref{fig:cobred-composition} for an illustration of the case $d=1$.

\begin{theorem}\label{theorem:2}
    For any $d \ge 0$ there are two 
    compatible homotopy fiber sequences of infinite loop spaces: 
    \[
        \begin{tikzcd}
            B \Cobcl_{d} \ar[r] \ar[d] 
            & B \Cob_{d} \ar[r, "R"] \ar[d] 
            & B \Cobred_{d} \ar[d] \\
            B (h\Cobcl_{d}) \ar[r] 
            & B (h\Cob_{d}) \ar[r, "R"] 
            & B (h\Cobred_{d}).
        \end{tikzcd}
    \]
    Here $R$ is induced by the functor $\Cob_d \to \Cobred_d$
    that deletes closed components.
\end{theorem}
In theorem \ref{thm:reduction-fiber-sequence} we also prove
\autoref{theorem:2} in the presence of tangential structures.
Most of the technical work towards a proof of \autoref{theorem:2} 
has been completed in our previous paper \cite{Stb19}, 
where we generalise Steimle's
\emph{additivity theorem for bordism categories} \cite{Stm18}
to a broader class of functors. 

The crucial observation needed to now deduce \autoref{theorem:1}
from \autoref{theorem:2} is that in dimension $d=1$ we have
that $\Cobred_1 \simeq h\Cobred_1$ because the diffeomorphism
group of an interval is contractible.

\subsection*{Topological cyclic homology of simply connected spaces}
One of the more mysterious aspects of the cobordism category 
is its relation to Waldhausen $K$-theory.
Recall that the Waldhausen $K$-theory $A(X)$ of a space $X$ 
is the infinite loop space of the algebraic $K$-theory      
of the ring spectrum $\IS[\gO X] := \gS^\infty (\gO X)_+$.
An effective way of computing algebraic $K$-theory
is via the cyclotomic trace to topological cyclic homology (TC).
In the relevant case we have a map
$\mi{trc}_p:A(X) \to \gO^\infty\TC(\IS[\gO X]; p)$
for any prime $p$.

For $X=*$ there is, up to $p$ completion,
an equivalence $\TC(\IS;p) \simeq \IS \vee \gS\CP^\infty_{-1}$
by \cite{BHM93}.
Here $\CP^\infty_{-1}$ is just another name for 
the spectrum $MTSO_2$, whose infinite loop space is $\gO B\Cob_2$.
In light of this apparent coincidence Ib Madsen concludes his 
2006 ICM plenary address \cite[407]{Mad07} with:
\begin{displayquote}
\emph{%
    Finally, and maybe most important, there are reasons to believe 
    that the moduli space of Riemann surfaces is related to 
    $\TC(\IS; p)$, possibly via field theories. 
    The spectrum $\CP^\infty_{-1}$ occurs in both theories. 
    It is a challenge to understand why.
}
\end{displayquote}
This curious coincidence has been known to experts for a while and,
even though no concrete explanations have been suggested,
the idea itself has inspired some intriguing research 
such as B\"okstedt and Madsen's map from $\gO B\Cob_d$ to $A$-theory,
see \cite{BM14} and \cite{RS14}.

In order to understand Madsen's question, it makes sense to think
of $\TC(\IS;p)$ and $B\Cob_2$ as the values of the functors 
$\TC(\IS[\gO X];p)$ and $B\Cob_2(X)$ for $X=*$.
Here $\Cob_2(X)$ is a variant of the surface category
where every object and bordism is equipped with a map to $X$.
If there is a fundamental relation between surfaces and $\TC$,
one could hope for a relation between these functors.
However, the analogy fails in this case:
the main theorem of \cite{GMTW06} implies
\( 
    B\Cob_2(X) \simeq \gO^\infty( MTSO_2 \wedge X_+)
\)
and hence the functor $B\Cob_2(X)$ is excisive in $X$, 
whereas $\TC(\IS[\gO X])$ is not.

We will try to argue that instead, topological cyclic homology 
is more naturally related to the reduced $1$-dimensional bordism category.
As a consequence of \autoref{theorem:2} we have that
$ B\Cobred_1 $ is equivalent to $\gO^{\infty-2} MTSO_2$,
a delooping of $B\Cob_2$.
For a space $X$ we define $\Cobred_1(X)$ to be the category
of reduced bordisms with maps to $X$. 
A generalisation of \autoref{theorem:2} together with the main theorem of \cite{BCCGHM96} implies:
\begin{theorem}\label{corollary:TC}
    For any simply connected space $X$ of finite type 
    there is an equivalence of infinite loop spaces
    \[
        \gO^\infty \mi{TC}(\IS[\gO X]; p) \simeq \left( Q(X_+) \times \gO B\Cobred_{1}(X) \right)_p^\wedge.
    \]
\end{theorem}
One way of interpreting this theorem is to say that $\Cobred_1(X)$
is a \emph{bordism model} for the topological cyclic homology 
of simply connected spaces.
This is similar in spirit to Raptis and Steimle's cobordism model
for Waldhausen $K$-theory.
In \cite{RS19} they construct, for every space $X$,
$\Cob^{\mrm{sym}}(X)$ such that $\gO B \Cob^{\mrm{sym}}(X) \simeq A(X)$.
One could therefore try to understand the cyclotomic trace 
in terms of these models.
Already for $X=*$ this is related to the difficult question 
of whether the cyclotomic traces
$\mi{trc}_p: A(*) \to (\gO^{\infty-1}\CP^\infty_{-1})_p^\wedge$
admit a common integral lift for all $p$.

\subsection*{Reduced bordisms as a combinatorial model for $\CP^\infty_{-1}=MTSO_2$}

As the previous section illustrates $\CP^\infty_{-1}$ 
is an interesting stable homotopy type arising in many important situations,
and it is helpful to have different models at hand.
\autoref{theorem:2} implies that $\gO^{\infty-2}\CP^\infty_{-1}$ is,
as an infinite loop space, equivalent to the classifying space of 
the symmetric monoidal category $h\Cobred_1 \simeq \Cobred_1$.
By the Madsen-Weiss theorem $\Cob_2$ also is a cobordism model
for this infinite loop space, though there is a dimension shift:
\[
    \gO^2 B(h\Cobred_1) \simeq \gO^\infty MTSO_2
    \simeq \gO B\Cob_2. 
\]
Even though $h\Cobred_1$ and $\Cob_2$ both are cobordism models 
for $MTSO_2$, they are of very different flavours.
The surface category $\Cob_2$ is of geometric nature 
as it is built from the diffeomorphism groups of surfaces.
The reduced one-dimensional bordism category $h\Cobred_1$,
on the other hand, admits a completely combinatorial description.
See figure \ref{fig:cobred-composition} for an example
of how morphisms are composed in $h\Cobred_1$.
\begin{figure}
    \centering
    \def\svgwidth{.8\linewidth}
    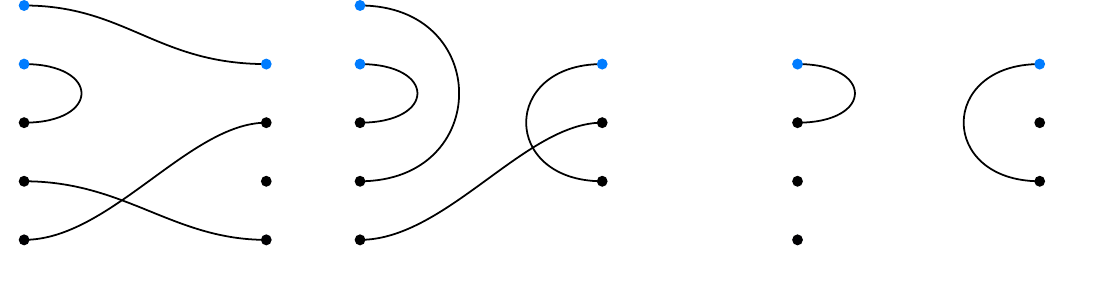
    \caption{Two morphisms in $h\Cobred_1$ and their composite.
    }
    \label{fig:cobred-composition}
\end{figure}

While \autoref{theorem:2} implies an equivalence 
$B(h\Cobred_1) \simeq \gO^{\infty-2}MTSO_2$,
it does not explicitly construct a map.
The second part of the paper aims to resolve this issue.
 
For this we construct an auxiliary simplicial space $\Cut_d$,
which we believe may be of independent interest.
$\Cut_d$ is a quotient of the nerve $N(\Cobred_d)$ 
and we show that the geometric realisation $\|\Cut_d\|$ 
is the free infinite loop space on $\gS^2 B\Diff^+(W)_+$ for all 
closed connected $d$-manifolds $W$.
We prove:
\begin{theorem}\label{theorem:continue-fib-seq}
    The quotient map $N(\Cobred_d) \to \Cut_d$ 
    continues the reduction fiber sequence of \autoref{theorem:2}
    in the sense that 
    \[
        B(\Cob_d) \longrightarrow B(\Cobred_d) 
        \longrightarrow \|\Cut_d\|
    \]
    is a homotopy fiber sequence of infinite loop space maps.
\end{theorem}
In dimension $1$ we have
\(
    \|\Cut_1\| 
    \simeq Q(\gS^2 (\CP^\infty)_+)
\)
and the quotient map $B\Cobred_1 \to \|\Cut_1\|$ 
is a rational equivalence on the basepoint component, 
see corollary \ref{cor:rational-equiv}.
In the presence of a background space $X$ this map
has an interpretation in terms of $\TC$, see corollary \ref{cor:map-to-THH}.

\subsection*{Application: cocycles for Miller-Morita-Mumford classes}
The Miller-Morita-Mumford classes are characteristic classes 
for surface bundles, which, under the equivalence 
$B\Cob_2 \simeq \gO^{\infty-1}MTSO_2$, give rise to polynomial generators:
\[
    \gk_i \in 
    H^*(\gO_0^\infty MTSO_2; \IQ) \cong \IQ[\gk_1,\gk_2,\dots].
\]
From the equivalence $B(\Cobred_1) \simeq \gO^{\infty-2}MTSO_2$
we can compute the rational cohomology of $\Cobred_1$ as:
\[
    H^*(B(\Cobred_1)_0; \IQ) \cong \IQ[\dlgk_0, \dlgk_1, \dots]. 
\]
Here $\dlgk_i$ is a double de-looping of $\gk_i$ and hence of degree $2i+2$.
Also note $\pi_0B(\Cobred_1) \cong \pi_0 B(\Cob_1) \cong \IZ$.

Generalising the notion of group cocycles, one can represent
every cohomology class on the classifying space $B\mcD$ 
of some discrete category $\mcD$ by a cocycle 
$\ga: N_i\mcD \to \IZ$ on the nerve.
While this does not work for the topologically enriched category 
$\Cob_2$, $h\Cobred_1$ is discrete and it is hence possible 
to describe the deloopings $\dlgk_i$ of the $\gk$-classes
in terms of cocycles on $h\Cobred_1$. 

To actually compute cocycles we crucially use that
every $\dlgk_i$ can be obtained as a pullback along
the map $B\Cobred_1 \to \|\Cut_1\|$ of \autoref{theorem:continue-fib-seq}.
We identify $\|\Cut_1\|$ with $Q(\gS^2 (B\gL)_+)$,
where $\gL$ is Connes' category of cyclic sets.
Then we use Igusa's description \cite{Igu04} of cocycles on $\gL$ representing
the powers of the first Chern class $(c_1)^i \in H^{2i}(B\gL; \IQ)$.
Explicitly, we have:
\begin{defn*}
For each $i\ge 0$ we construct a $(2i+2)$-cocycle $\gc_i$ 
on $h\Cobred_1$ and $h\Cob_1$ in three steps.
\medskip
\\
\textbf{(1)}
    For a $(2i+1)$-tuple of disjoint points $a_0, \dots, a_{2i} \in S^1$
    the sign $\sign(a_0,\dots,a_{2i}) \in \{\pm1\}$ is defined 
    to be equal to the sign of any permutation 
    $\gs$ of $\{0,\dots,2i\}$ such that the sequence
    $(a_{\gs(0)}, \dots, a_{\gs(2i)})$ is in cyclic order.
\medskip
\\
\textbf{(2)}
    For a $(2i+1)$-tuple of disjoint finite subsets 
    $A_0, \dots, A_{2i} \subset S^1$
    the reduced sign 
    $\cl{\sign}_{2i}(A_0,\dots,A_{2i})\in \IQ$
    is defined by an averaging procedure.
    Say that $(a_0,\dots,a_{2i}) \in \prod_i A_i$
    \emph{contains no neighbours} if for all $k<l$ there is a $j$ 
    such that the positively oriented arc $[a_k,a_l] \subset S^1$ 
    from $a_k$ to $a_l$ intersects $A_j$ in more than one point.
    Now define:
    \[
        \cl{\op{sign}}_{2i}(A_0, \dots, A_{2i}) := 
        \frac{1}{\prod_{j=0}^{2i} |A_j|} 
        \sum_{{(a_0, \dots, a_{2i}) \in \prod_j A_j}
        \atop {\text{contains no neighbours}}} 
        \op{sign}(a_0, \dots, a_{2i})
        \in \IQ.
    \]
\noindent
\textbf{(3)}
    For all $i \ge 1$ define a $(2i+2)$-cochain on the 
    simplicial set $N(h\Cobred_1)$ by the formula
    \[
        \gc_i(M_0 \xrightarrow{W_1}
        M_1 \xrightarrow{W_2} \dots
        \xrightarrow{W_{2i+2}} M_{2i+2})
        := 
        \frac{(-1)^i i!}{(2i)!}
        \sum_{[\gi:S^1 \inj W]}
        \cl{\op{sign}}_{2i}(\gi^{-1}(M_1^+), \dots, 
        \gi^{-1}(M_{2i+1}^+)).
    \]
    Here we write $W$ for the composition 
    $W_1 \cup_{M_1} \dots \cup_{M_{2i+1}} W_{2i+2}$
    and $M_i^+ \subset M_i$ for the set of positively oriented points.
    The sum $ \sum_{[\gi:S^1 \inj W]} $
    runs over isotopy classes of oriented
    embeddings $\gi:S^1 \inj W$ such that $\gi(S^1)$
    intersects $M_j$ for all $1 \le j \le 2i+1$.
\end{defn*}

\begin{theorem}
    For $i \ge 1$ the cochain $\gc_{i}$ on $h\Cobred_1$ defined above
    is a $(2i+2)$ cocycle and the cohomology class $[-\gc_i]$ corresponds to,
    possibly up to a sign $(-1)^i$, the generator $\dlgk_i$ under the isomorphism
    \[
        H^*(B(h\Cobred_1)_0; \IQ) \cong \IQ[\dlgk_0, \dlgk_1, \dlgk_2, \dots].
    \]
    The same formula describes a cocycle for $\pm\dlgk_i$ on $h\Cob_1$.
\end{theorem}

\subsection*{Outline}
This paper is divided into two parts, the first one of which is centered 
around establishing the reduction fiber sequence.
In section \ref{sec:recall} we recall various standard facts and constructions 
for bordism categories and topological categories in general.
We also prove a generalisation of the base-change theorem \cite[Theorem 5.2]{ERW19},
which allows one to change the space of objects of a topological category
without changing the homotopy type of its classifying space.
Section \ref{sec:closed-reduced} introduces the closed and reduced bordisms
categories and studies some of their basic properties.
In section \ref{sec:reduction-fibseq} we apply the main theorem 
of \cite{Stb19} to prove the reduction fiber sequence \autoref{theorem:2}.
Specialising to $d=1$, in section \ref{sec:Computations-d=1} 
we observe that $\Cobred_1 \simeq h\Cobred_1$ and use this
to deduce \autoref{theorem:1} and \autoref{corollary:TC}.
Here we also consider the unoriented case $B(h\Cob_{1,\mrm{unor}})$.

The second part is focused on gaining a more concrete understanding 
of $\Cobred_1$.
We begin in section \ref{sec:2-cocycle} where we, by elementary means,
compute the $2$-cocycle on $h\Cobred_d$ that classifies 
the central extension $h\Cob_d \to h\Cobred_d$.
Section \ref{sec:continued-and-cuts} introduces the simplicial 
space of cuts $\Cut_d$ and proves \autoref{theorem:continue-fib-seq},
identifying the connecting homomorphism of the reduction fiber sequence.
This is then used in section \ref{sec:identifying-cocycles} 
to construct the cocycles for $\dlgk$-classes on $h\Cobred_1$.

\subsection*{Acknowledgements}
I would like to thank my PhD advisor Ulrike Tillmann 
for her support throughout all stages of this project.
I also thank Oscar Randal-Williams and 
George Raptis for several useful conversations, 
Thibault D\'ecoppet for comments on an earlier version,
and the referee for a very detailed report 
that helped improving the exposition of the paper.
I am very grateful for the support by
St. John's College, Oxford through the
``Ioan and Rosemary James Scholarship'',
and the EPSRC grant no.\ 1941474.
This paper was completed while the author was in residence at 
the Mathematical Sciences Research Institute in Berkeley, California,
in spring of 2020.

\setcounter{tocdepth}{1}
\tableofcontents

\part{The reduction fiber sequence}

\section{Recollections on moduli spaces and cobordism categories}
\label{sec:recall}

While our main object of interest is the homotopy category of the bordism category $h\Cob_1$,
the proof of \autoref{theorem:1} crucially relies on comparing it to the topological 
category $\Cob_1$ and some of its variants.
It is hence essential for us to have a good understanding 
of the embedded models for the cobordism category.
In this section we recall various definitions and facts.

Although we introduce all the necessary tools,
we can only do so concisely. The interested reader is referred 
to \cite{GRW18} for a discussion of moduli spaces for manifolds,
and to \cite{ERW19} for an introduction to the world of non-unital 
topological categories and semisimplicial spaces.

\begin{rem}
    The specific model for the cobordism category $\Cob_d$ 
    that we are using is essentially that of \cite{GMTW06}.
    This is a concrete model for the $(\infty,1)$-category of bordisms,
    which is often denoted $\mrm{Bord}_{\gle{d-1,d}}^{\mrm{or}}$.
    From it we will also derive the homotopy category $h\Cob_d$.
    This is the ordinary $1$-category of bordisms,
    which is often denoted by $\mrm{Cob}_d^{\mrm{or}}$. 
    (E.g.\ \cite{Lur09}).
\end{rem}

\subsection{Tangential structures}
Most of the manifolds we consider will be oriented.
Recall that one way of defining an orientation on a manifold $M$ is by
giving an equivariant continuous map $l:\mrm{Fr}(TM) \to \{-1, +1\}$ 
from the total space of the frame bundle of $M$ to the two-element set.
Here equivariance is with respect to the group action of $\GL_d$ on 
the left by base-change and on the right by multiplication with 
the sign of the determinant.
Tangential structures generalise this notion of orientation.
\begin{defn}
    A $d$-dimensional \emph{tangential structure} $\gt$ is a space $\gt$ 
    with $\GL_d$-action.
    Given such a $\gt$ a $\gt$-structure on a $d$-dimensional manifold $W$ 
    is a $\GL_d$-equivariant
    map $l: \mrm{Fr}(TW) \to \gt$ from the frame bundle of $W$ to $\gt$.
    The space of $\gt$-structures on $M$ will be denoted by
    \[
        \mrm{Bun}^\gt(M) := \Map_{\GL_d}(\mrm{Fr}(TW), \gt).
    \]
\end{defn}

In the case of the tangential structure for orientation 
$\gt^{\mrm{or}} := \{\pm 1\}$,
we are often interested in the group of \emph{orientation preserving} 
diffeomorphisms $\Diff^+(W)$. For more general tangential structures $\gt$
the group of diffeomorphism that fix a specific $\gt$-structure `on the nose'
is generally not well-behaved.
Instead, we define a moduli space that acts as the classifying space 
of this hypothetical group.
\begin{defn}\label{defn:BDiff-gt}
    For $\gt$ and $W$ as above we define $B\Diff^\gt(W)$ as the homotopy orbit space
    \[
        B\Diff^\gt(W) := \Bun^\gt(W) \doublebs \Diff(W)
        \stackrel{\text{def}}{=} (\Bun^\gt(W) \times E \Diff(W))\big/\Diff(W).
    \]
\end{defn}

\subsection{Spaces of manifolds and cobordisms}
We recall the space of submanifolds of $\IR^N$ 
and how to use it to define a topological space of (embedded) cobordisms.

\begin{defn}
    For $U \subset \IR^N$ open let $\Psi_{d,\gt}(U)$ denote the set of
    pairs $(M,l)$ where $M \subset U$ is a
    $d$-dimensional submanifold of $\IR^N$ that is closed as a subset of $\IR^N$
    and $l:\mrm{Fr}(TM) \to \gt$ is a $\gt$-structure on $M$.
    We let $\Psi_{d,\gt}$ denote the colimit of 
    $\Psi_{d,\gt}( (-1,1)^N ) $ as $N \to \infty$.
    For finite $N$ we topologise this according to 
    \cite[Definition 2.1]{GRW10} and for $N=\infty$
    as the colimit over all finite $N$.
\end{defn}
    
One can think of $\Psi_{d,\gt}$ as a concrete topological model for the 
``moduli space of $\gt$-structured closed $d$-dimensional manifolds''.
In \cite[Section 2.2]{GRW18} this space is denoted $\mathcal{M}^\gt$.
The moduli space decomposes as a disjoint union over diffeomorphism types:
\begin{fact}\label{fact:Psi=BDiffs}
    There is a weak equivalence
    \[
        \Psi_{d,\gt} \simeq \coprod_{[W]} B\Diff^\gt(W)
    \]
    where $W$ runs over a set of representatives of diffeomorphism classes 
    of closed $d$-dimensional manifolds.
\end{fact}

\begin{defn}
    We say that $(W,l) \in \Psi_{d,\gt}(\IR \times U)$ is \emph{cylindrical}
    over some interval $(a,b)$ if 
    $W_{|(a,b)} := W \cap \left( (a,b) \times U\right)$ 
    is equal to the product $(a,b) \times M$ for some $M \in \Psi_{d}(U)$ 
    and the $\gt$-structure $l_{|(a,b)}:\mi{Fr}(TW_{|(a,b)}) \to \gt$
    can be factored as
    \[
        \mrm{Fr}(TW_{|(a,b)}) \cong \mrm{Fr}(T(a,b) \times TM) 
        \xrightarrow{ \pr } \mrm{Fr}(\IR \oplus TM) 
        \cong (\GL_d \times \mrm{Fr}(TM))/_{\GL_{d-1}} 
        \to \gt
    \]
    where the last map is induced by some $\GL_{d-1}$-equivariant map
    $l':\mrm{Fr}(TM) \to \gt$.
\end{defn}
    
\begin{defn}
    For $d$ and $\gt$ as above and $\eps>0$ we define 
    $\Phi_{d,\gt}^\eps \subset \Psi_{d,\gt}(\IR \times \cube) \times \IR_{>0}$
    as the space of those $((W,l),t)$ that are cylindrical over
    $(-\infty,\eps)$ and $(t-\eps,\infty)$.
    We let the \emph{space of $\gt$-structured $d$-dimensional cobordisms}
    $\Phi_{d,\gt}$ be the colimit as $\eps \to 0$.
\end{defn}

\begin{figure}[ht]
    \centering
    \def\svgwidth{.3\linewidth}
%% Accompanies image file 'infinite-cobordism.pdf' (pdf, eps, ps)
\begingroup%
  \makeatletter%
  \providecommand\color[2][]{%
    \errmessage{(Inkscape) Color is used for the text in Inkscape, but the package 'color.sty' is not loaded}%
    \renewcommand\color[2][]{}%
  }%
  \providecommand\transparent[1]{%
    \errmessage{(Inkscape) Transparency is used (non-zero) for the text in Inkscape, but the package 'transparent.sty' is not loaded}%
    \renewcommand\transparent[1]{}%
  }%
  \providecommand\rotatebox[2]{#2}%
  \newcommand*\fsize{\dimexpr\f@size pt\relax}%
  \newcommand*\lineheight[1]{\fontsize{\fsize}{#1\fsize}\selectfont}%
  \ifx\svgwidth\undefined%
    \setlength{\unitlength}{217.31979971bp}%
    \ifx\svgscale\undefined%
      \relax%
    \else%
      \setlength{\unitlength}{\unitlength * \real{\svgscale}}%
    \fi%
  \else%
    \setlength{\unitlength}{\svgwidth}%
  \fi%
  \global\let\svgwidth\undefined%
  \global\let\svgscale\undefined%
  \makeatother%
  \begin{picture}(1,0.62349855)%
    \lineheight{1}%
    \setlength\tabcolsep{0pt}%
    \put(0,0){\includegraphics[width=\unitlength,page=1]{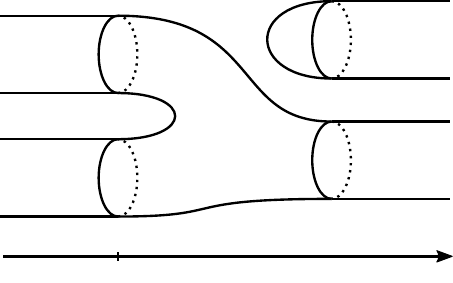}}%
    \put(0.243339,-0.02471458){\color[rgb]{0,0,0}\makebox(0,0)[lt]{\lineheight{1.25}\smash{\begin{tabular}[t]{l}$0$\end{tabular}}}}%
    \put(0.72033523,-0.02471458){\color[rgb]{0,0,0}\makebox(0,0)[lt]{\lineheight{1.25}\smash{\begin{tabular}[t]{l}$t$\end{tabular}}}}%
    \put(0,0){\includegraphics[width=\unitlength,page=2]{infinite-cobordism.pdf}}%
  \end{picture}%
\endgroup%

    \caption{A point in the space of two-dimensional cobordisms $\Phi_{2,\gt}$.}
    \label{fig:infinite-cobordism}
\end{figure}
We think of $t>0$ as the length of the cobordism $(W,l)$.
Indeed, the cylindricality condition implies that the part of the bordisms
that does not lie over $[0,t]$ is superfluous.
We still keep track of the cylinders over $(-\infty,0)$ and $(t,\infty)$
because it simplifies the definition of the topology.
For example, we can use them to define the map that sends a cobordism to its boundary:

\begin{defn}\label{defn:partial}
    We define a map
    \[
        (\partial_0, \partial_1): \Phi_{d,\gt} \to \Psi_{d-1,\gt} \times \Psi_{d-1,\gt}
    \]
    by sending $((W,l),t)$ to the unique tuple $((M_0,l_0), (M_1,l_1))$
    such that $(-\infty,0)\times M_0 \cup (t,\infty) \times M_1$ 
    is a codimension $0$ submanifold of $W$ and $l_0$ and $l_1$
    are the induced $\gt$-structures.
\end{defn}

\subsection{The cobordism category}

All cobordism categories we consider will be weakly unital%
\footnote{
    We will be using the definition of weakly unital given in \cite{Stb19}
    as it is the one relevant for the local additivity theorem \ref{thm:local-add}.
    Since the cobordism category is also fibrant this will imply that it is
    weakly left and right unital in the sense of \cite{ERW19}.
    (See \cite[Remark 3.12]{ERW19}.)
}
topological categories. We refer the reader to \cite{ERW19} for an excellent 
introduction to semisimplicial spaces, non-unital topological categories, 
and fibrancy conditions. 
By convention all our categories $\mcD$ will be non-unital 
and $B\mcD$ will denote the ``fat geometric realisation'' 
of the semisimplicial space $N\mcD$, even if $\mcD$ happens to have units.

\begin{defn}
    A \emph{non-unital topological category} $\mcD$ consists of the following data:
    a space of objects $O$, a space of morphisms $M$, source and target maps 
    $s,t:M \to O$, and a composition map
    \[
        c: M\ {{}^t\times_O^s} M = \{ (f,g) \in M^2 \;|\; t(f) = s(g) \}
        \longrightarrow M.
    \]
    This data is subject to the axioms 
    \begin{align*}
        s(c(f,g)) &= s(f) &
        t(c(f,g)) &= t(g) &
        c(c(f,g),h)) &= c(f, c(g,h))
    \end{align*}
    for any three morphisms $f,g,h \in M$ with $t(f)=s(g)$ and $t(g) = s(h)$.
\end{defn}

\begin{defn}
    For any two objects $x,y \in O$ the space of morphism 
    from $x$ to $y$ is defined as the pullback
    \[
        \hom_\mcD(x,y) := \{x\} \times_{O}^s M \ {{}^t\!\times}_{O}  \{y\}
        = \{ f \in M \;|\; s(f) = x \text{ and } t(f) = y\}.
    \]
    We usually write $f:x \to y$ to say $f \in \hom_\mcD(x, y)$,
    and we also write $g \circ f$ for the composite $c(f,g)$.
\end{defn}

For us the main example of a topological category is the cobordism category,
which we now define using the space of cobordisms $\Psi_{d,\gt}$ 
constructed in the previous section as a space of morphisms.

\begin{defn}\label{defn:Cob}
    Fix a dimension $d$ and a tangential structure $\gt$.
    The non-unital topological category $\Cob_{d,\gt}$ has 
    $\Psi_{d-1,\gt}$ as space of objects and $\Phi_{d,\gt}$ as space of morphisms.
    The source and target maps are the maps 
    $\partial_0,\partial_1:\Phi_{d,\gt} \to \Psi_{d-1,\gt}$ 
    from definition \ref{defn:partial}. 
    Composition is defined by
    \[
        ((W',l'),t') \circ ((W,l),t) := ((W'',l''), t+t')
    \]
    where
    \[
        W'' = (W \cap (-\infty,t]\times \cube) 
        \cup ( (W' + t\cdot e_0) \cap ([t,\infty) \times \cube))
    \]
    and $l''$ is induced by $l$ and $l'$. 
    Here $(\blank + t \cdot e_0)$ denotes translation by $t$ 
    in the direction of $e_0 = (1, \ul{0}) \in \IR \times \cube$.
\end{defn}

\begin{rem}
    Note that $\Cob_{d,\gt}$ is well-behaved: it is fibrant in the sense of
    \cite[Definition 3.5]{ERW19} by \cite[Proposition 3.2.4(ii)]{ERW19b}.
    Moreover, the cylinder bordisms $(M \times \IR, 1): M \to M$ 
    define weak units in the sense of \cite{Stb19}.
\end{rem}

In the case of the bordism category there is a geometric interpretation
of the homotopy type of this hom space.

\begin{fact}\label{fact:hom-in-Cob}
    For any two objects $(M,l)$ and $(N,l')$ in $\Cob_{d,\gt}$
    there is a weak equivalence
    \[
        \hom_{\Cob_{d,\gt}}((M,l), (N,l'))
        \simeq 
        \coprod_{[W,\gp]} B\Diff^\gt (W \text{ rel } M \amalg N).
    \]
    Here the coproduct runs over a set of representatives
    for compact $d$-manifolds $W$ with boundary $M \amalg N$,
    under the equivalence relation defined by diffeomorphisms that fix the boundary.
    The space $B\Diff^\gt(W \text{ rel } M \amalg N)$ is defined
    as in definition \ref{defn:BDiff-gt} with the modification that diffeomorphism
    are trivial near the boundary that the tangential structure 
    agrees with $l$ and $l'$ near the boundary.
\end{fact}

\subsection{Homotopy categories}

The homotopy category $h\mcD$ of a topological category $\mcD$ is an
ordinary category obtained from $\mcD$ in two steps: 
first we need to pass to an enriched category $\gd\mcD$ 
and then we apply $\pi_0$ to the $\hom$-spaces.

\begin{defn}\label{defn:discrete-cat}
    For any non-unital topological category $\mcD$ we define 
    its \emph{discretefication} $\gd(\mcD)$ as the non-unital topological 
    category with object space $\mrm{Obj}(\gd(\mcD))$ 
    the set $\mrm{Obj}(\mcD)$, equipped with the discrete topology, 
    and morphism space the disjoint union
    \[
        \mrm{Mor}(\gd(\mcD)) := \coprod_{x,y \in \mcD} \hom_\mcD(x,y).
    \]
    The homotopy category $h(\mcD)$ is the ordinary category with object
    set $\mrm{Obj}(\gd(\mcD))$ and hom-sets 
    \[
        \hom_{h(\mcD)}(x, y) := \pi_0 \hom_{\mcD}(x, y).
    \]
    There are canonical functors
    $
        \mcD \xleftarrow{\gi} \gd(\mcD) \xrightarrow{\pi} h(\mcD)
    $
    and they are natural in $\mcD$.
\end{defn}

\begin{rem}
    Recall that a (non-unital) \emph{topologically enriched category}
    is the same datum as a (non-unital) topological category
    $\mcD = (O, M, s, t, c)$
    where the space of objects $O$ is discrete.
    
    In this sense $\gd(\mcD)$ is always a topologically enriched category.
    The canonical functor $\gi: \gd(\mcD) \to \mcD$ is a continuous bijection
    on the object the morphism space, but it is not a homeomorphism
    unless $\mrm{Obj}(\mcD)$ was already discrete.
    From the perspective of classifying spaces, however, 
    $\gd(\mcD)$ is equivalent to $\mcD$ as long as $\mcD$ is sufficiently well-behaved.
\end{rem}

\begin{lem}\label{lem:gd-doesnt-change-B}
    If $\mcD$ is a fibrant non-unital topological category with weak 
    left (or right) units, then $\gi$ induces a weak equivalence
    $B(\gd(\mcD)) \to B(\mcD)$.
\end{lem}
\begin{proof}
    This is the basechange theorem \cite[Theorem 5.2]{ERW19} in the 
    case where $X$ is the set $\mrm{Obj}(\mcD)$ equipped with the 
    discrete topology.
    We will give an independent proof of (a generalisation of)
    the basechange theorem in lemma \ref{lem:base-change}.
\end{proof}

We cannot generally expect the homotopy category $h\mcD$ to have a
classifying space equivalent to that of $\mcD$. Nevertheless,
the canonical map is always $2$-connected:
\begin{lem}\label{lem:gd-h-2con}
    The map $B(\gd\mcD) \to B(h\mcD)$ is $2$-connected for any 
    non-unital topological category $\mcD$.
\end{lem}
\begin{proof}
    This is well-known and for example follows from 
    \cite[Lemma 2.4]{ERW19} together with the observation that 
    $\mrm{Mor}(\gd(\mcD)) \to \pi_0 \mrm{Mor}(\gd(\mcD)) = \mrm{Mor}(h(\mcD))$
    is $1$-connected.
\end{proof}

Of course, our main category of interest is the cobordism category.
The homotopy category of the bordism category also admits 
a more conceptual description that does not rely on the embeddings.
For simplicity, we only spell this out in the oriented case.
\begin{fact}
    The homotopy category $h(\Cob_{d,\gt^{\mrm{or}}})$ is equivalent 
    to the following category:
    \begin{itemize}
        \item objects are closed oriented $(d-1)$-dimensional manifolds, 
        \item morphisms $W:M \to N$ are equivalence classes of compact oriented
    $d$-dimensional manifolds with $\partial W = M^{-} \amalg N$,
    where two such manifolds are equivalent if there is an orientation preserving
    diffeomorphism between them that fixes the boundary, and
        \item composition is defined by gluing cobordisms.
    \end{itemize}
\end{fact}

\subsection{Infinite loop space structures}
\label{subsec:infinite-loop}

All homotopy categories of cobordism categories have symmetric monoidal
structures defined by the disjoint union operation for manifolds.
This induces an infinite loop space structure on their classifying spaces.
Many of our computations will rely on these infinite loops space structures
and their compatibility  with certain constructions.
We will be keeping track of the infinite loop space structures using 
Segal's $\gC$-spaces \cite{Seg74}, which we recall here.

\begin{defn}
    Segal's category $\gC^{op}$ has as objects natural numbers $n \ge 0$
    and as morphisms from $n$ to $m$ maps of sets
    $
        \gl:\{*,1,\dots,n\} \to \{*,1,\dots,m\}
    $
    satisfying $\gl(*) = *$. We let $\rho_n^i:n \to 1$ denote the morphism
    with $\rho_n^i(j) = *$ for $i \neq j$ and $\rho_n^i(i) = 1$.
    
    A $\gC$-space is a functor $X:\gC^{op} \to \Top$.
    We call $X$ \emph{special} if the Segal map 
    \[
        X(n) \xrightarrow{(\rho_n^1,\dots,\rho_n^n)} 
        (X(1))^n
    \]
    is a weak equivalence for all $n\ge 0$.
    A special $\gC$-space is \emph{very special} if the map 
    \[
        \pi_0 X(2) \xrightarrow{(\rho_2^1, \mu)} \pi_0 X(1) \times \pi_0 X(1)
    \]
    is a bijection. Here $\mu:2 \to 1$ is the morphism
    with $\mu(1) = \mu(2) =1$.
\end{defn}

\begin{rem}
    For any special $\gC$-space $X$ the set $\pi_0 X(1)$ has an abelian
    monoid structure with multiplication defined by
    \[
        m: \pi_0 X(1) \times \pi_0 X(1) \xleftarrow{\cong} \pi_0 X(2)
        \xrightarrow{X(\mu)} \pi_0 X(1).
    \]
    This uses that the Segal map $X(2) \to X(1) \times X(1)$ is a bijection
    on $\pi_0$.
    The special $\gC$-space is \emph{very special} if and only if
    this abelian monoid happens to be an abelian group.
\end{rem}

Following \cite{Ngu17} we use the tangential structure 
to define a $\gC$ structure on the cobordism category.
\begin{defn}
    For any dimension $d$ and tangential structure $\gt$ we define a 
    $\gC$-object in non-unital categories by setting
    $
        \Cob_{d,\gt}(n) := \Cob_{d, n\gt}
    $
    where $n\gt$ denotes $\{1,\dots,n\} \times \gt$.
    To a morphism $\gl: n \to m$ in $\gC^{op}$ we associate the functor 
    \[
        \gl_*: \Cob_{d,n\gt} \to \Cob_{d,m\gt} 
        \quad \text{with} \quad
        ((W,l),t) \mapsto ((W',(\gl \times \id_\gt) \circ l),t)
    \]
    where $W' \subset W$ is the preimage of $\{1,\dots,m\}$ under the
    map 
    $(\gl \circ \pr_{\{1,\dots,n\}} \circ l):W \to 
    \{*,1,\dots,m\}$.
\end{defn}

\begin{lem}[\cite{Ngu17}]\label{lem:BCob-vspecial}
    The $\gC$-space $B\Cob_{d,\gt}$ is very special.
\end{lem}

We also have infinite loop space structures on the classifying space 
of the homotopy categories.
\begin{lem}\label{lem:BhCob-vspecial}
    The $\gC$-space $Bh\Cob_{d,\gt}$ is very special.
\end{lem}
\begin{proof}
    Since $h\Cob_{d,\gt}$ is a symmetric monoidal category,
    it follows from Segal's original paper \cite[\S 2]{Seg74} 
    that $Bh\Cob_{d,\gt}$ is special.
    Being very special is a condition on the monoid 
    $\pi_0 Bh\Cob_{d,\gt}$ and follows by the considerations 
    in lemma \ref{lem:BCob-vspecial}, seeing as 
    $\pi_0 Bh\Cob_{d,\gt} \cong \pi_0 B\Cob_{d,\gt}$.
\end{proof}

\subsection{The Galatius-Madsen-Tillmann-Weiss theorem}

In \cite{GMTW06} the authors determined the infinite loop space
$B(\Cob_{d,\gt})$ in terms of modified Thom spectra. 
As our strategy is to compare $B(h\Cob_1)$ with $B(\Cob_1)$, 
we briefly recall their result.

\begin{defn}
    Let $E(d,n)$ denote the space of linear embeddings of $\IR^d$
    into $\IR^n$, with the $\GL_d$-action by precomposition,
    and let $\mi{Gr}(d,n) := E(d,n)/\GL_d$ be the Grassmannian.
    The canonical bundle $\gc_d$ on $\mi{Gr}(d,n)$ is defined by
    \[
        \gc_{d,n} := (E(d,n) \times \IR^d)/\GL_d \longrightarrow \mi{Gr}(d,n).
    \]
    This is a subbundle of the trivial bundle $\mi{Gr}(d,n) \times \IR^n$
    via the map $(e, v) \mapsto ([e], e(v))$.
    We let $\gc_{d,n}^\bot$ denote the orthogonal complement 
    of $\gc_{d,n}$ in $\mi{Gr}(d,n) \times \IR^n$.
\end{defn}

\begin{defn}
    For any $d$ and $\gt$ the $\gt$-structured Madsen-Tillmann spectrum
    $MT\gt$ is the sequential spectrum with 
    \[
        (MT\gt)_n := \mrm{Th}(a^* \gc_{d,n}^{\bot})
    \]
    for $a: (\gt \times E(d,n))/\GL_d \to \mi{Gr}(d,n)$ the projection.
    The structure maps $\gS(MT\gt)_n \to (MT\gt)_{n+1}$ are defined
    using the canonical bundle map 
    $\gc_{d,n}^\bot \oplus \ul{\IR} \to \gc_{d,n+1}^\bot$.
\end{defn}

The theorem by Galatius, Madsen, Tillmann, and Weiss
uses Segal's `scanning' method to construct an equivalence
from $B\Cob_{d,\gt}$ to a loop space of $MT\gt$.
For our purposes it will be sufficient to treat this  as a black box 
and we refer the interested reader to the original paper.

\begin{thm}[Main theorem of \cite{GMTW06}]\label{thm:GMTW}
    For any $d$ and $\gt$ the scanning map defines a weak equivalence of 
    infinite loop spaces
    \[
        B\Cob_{d,\gt} \simeq \gO^{\infty-1} MT\gt 
        = \colim_{n\to \infty} \gO^{n-1} (MT\gt)_n.
    \]
\end{thm}

This theorem has been used to great effect in the study 
of diffeomorphism groups. The essential ingredient is a natural map 
from $B\Diff^\gt(W)$ to $\gO B\Cob_{d,\gt}$ for any 
closed $d$-manifold $W$.
This map is defined by interpreting $W$ 
as a morphism $W:\emptyset \to \emptyset$ in $\Cob_d$,
which then yields a based loop in the classifying space.

\begin{defn}\label{defn:MW-map}
    The inclusion $\Psi_{d,\gt} \inj \Phi_{d,\gt} = N_1 \Cob_{d,\gt}$
    induces a map $\Psi_{d,\gt} \times |\gD^1| \to B\Cob_{d,\gt}$
    such that $\Psi_{d,\gt} \times |\partial \gD^1|$ is sent to the base point.
    Hence, it defines a map
    \(
        \ga' : \gS( \Psi_{d,\gt} )_+ \longrightarrow B\Cob_{d,\gt}
    \)
    and by adjunction a map
    \[ \ga: \Psi_{d,\gt} \to \gO B\Cob_{d,\gt}.\]
    Together with the equivalence from fact \ref{fact:Psi=BDiffs} 
    this yields, for every closed $d$-manifold $W$, a map 
    \[
        \ga_W: B\Diff^\gt(W) \longrightarrow \gO B\Cob_{d,\gt}.
    \]
\end{defn}

We conclude this section with the so-called `Genauer fiber sequence'.
Recall that $Q(X)$ denotes the free infinite loop space on a based space $X$
defined as the colimit $Q(X) := \colim_n \gO^n(\gS^n X)$. 
It has the property that any map of spaces $X \to Y$ into an infinite
loop space $Y$ induces a map of infinite loop spaces $Q(X) \to Y$.
\begin{lem}[{\cite{Gal06}}]\label{lem:ga=fiber}
    There is a homotopy fiber sequence of infinite loop spaces 
    \[
        \gO^{\infty-1} MTSO_d \longrightarrow Q(\gS (BSO_d)_+) 
        \longrightarrow \gO^{\infty-1} MTSO_{d-1}.
    \]
    Moreover, the right-hand map can be described as the composite:
    \[
            Q(\gS (BSO_d)_+) \xrightarrow{Q(\gS Bi)} 
            Q( \gS B\Diff^+(S^{d-1}) ) \xrightarrow{\ga_{S^{d-1}}}
            B\Cob_{d-1}^{\mrm{or}} \xrightarrow{\simeq} \gO^{\infty-1} MTSO_{d-1} 
    \]
    where $i: SO_d \to \Diff^+(S^{d-1})$ is the 
    canonical inclusion and $\ga_{S^{d-1}}$ is the map from
    definition \ref{defn:MW-map}.
\end{lem}
\begin{proof}[References]
    This homotopy fiber sequence is well-known through the paper \cite{GMTW06},
    but already appeared in Galatius' paper \cite{Gal06}.
    We refer the reader to \cite[Proposition A.4]{Ebe13} where the fiber sequence
    is explained in notation very similar to ours.
    The second part of the lemma, the identification of the infinite loop space map
    $Q(\gS(BSO_d)_+) \to \gO^{\infty-1} MTSO_{d-1}$,
    is harder to extract from the literature.
    It will suffice to understand the map of spaces $BSO_d \to \gO^\infty MTSO_{d-1}$.
    In \cite[Proposition A.4]{Ebe13} Ebert describes it as the Madsen-Tillmann-Weiss 
    map for the sphere bundle $BSO_{d-1} \to BSO_d$, 
    and indeed Galatius in \cite[Lemma 2.1]{Gal06} 
    constructs the relevant map as a parametrised Pontrjagin-Thom construction.
    This map factors as:
    \[
        BSO_d \to B\Diff^+(S^{d-1}) \to \gO^{\infty}MTSO_{d-1}
    \]
    where the left-hand map is the classifying map for the sphere bundle
    $BSO_{d-1} \to BSO_d$ and the right-hand map is the Pontrjagin-Thom 
    construction for the universal oriented $S^{d-1}$-bundle.
    The sphere bundle is classified by $B(i): BSO_d \to B\Diff^+(S^{d-1})$
    and so the claim follows.
\end{proof}

\subsection{A generalised change of base theorem}
In this section we will prove a small generalisation of 
\cite[Theorem 5.2]{ERW19} that will come in handy later.
First, recall the notion of base-change from 
\cite[section 5]{ERW19}.

\begin{defn}
    Let $\mcD$ be a 
    topological category and let $X$ be a space 
    with a map $f:X \to \Obj(\mcD)$. 
    The base-change of $\mcD$ along $f$ is the topological
    category, denoted by $\mcD^f$, with object space 
    $X$ and morphism space 
    $X^2 \times_{\Obj(\mcD)^2} \Mor(\mcD)$.
    Composition is induced by $\mcD$.
\end{defn}

\begin{lem}\label{lem:base-change}
    Let $\mcD$ be a weakly unital topological category
    and consider a zig-zag of maps
    $g: Y \to Z \leftarrow \Obj(\mcD):p$. 
    Let $f: Y \times_Z \Obj(\mcD) \to \Obj(\mcD)$ be the projection.
    Then the canonical map
    \[
        B(\mcD^f) \longrightarrow B\mcD
    \]
    is a weak equivalence, if the following conditions hold:
    \begin{itemize}
        \item the map $g:Y \to Z$ is surjective 
        on connected components,
        \item for all $n \ge 0$ the composite
        $p:N_n\mcD \to 
        \Obj(\mcD)^{n+1} \xrightarrow{p^{n+1}} Z^{n+1}$
        is a fibration.
    \end{itemize}
\end{lem}

Note that in the case of $Z=\Obj(\mcD)$ and $p= \id_{\Obj(\mcD)}$ 
the second condition is equivalent to asking that $\mcD$ be fibrant;
hence we recover \cite[Theorem 5.2]{ERW19} as a special case.

\begin{proof}
    For any space $X$ we let $T(X)$ denote the trivial 
    groupoid on $X$. Its space of object is $X$ and 
    its space of morphisms is $X\times X$.
    In other words, for every pair of objects $(x,y)$
    there is precisely one morphism $x \to y$.
    Composition is defined by $(y,z) \circ (x,y) =(x,z)$.
    
    This is a fibrant topological category.
    For $X$ non-empty we can pick a point $y$
    and define an extra degeneracy on the simplicial space
    $N(X)$ by $s_{-1}(x_0,\dots,x_n) = (y, x_0, \dots, x_n)$.
    This shows that $B(T(X))$ is contractible 
    for $X$ non-empty.
    
    Note that the basechange $\mcD^f$ of $\mcD$ along
    a map $f:X \to \Obj(\mcD)$ can be written as 
    a pullback
    $\mcD^f \cong T(X) \times_{T(\Obj(\mcD))} \mcD$
    where we use the canonical functor 
    $\mcD \to T(\Obj(\mcD))$, which is the identity on objects.
    
    To prove the lemma, we first observe that we can replace
    the map $g:Y \to Z$ by a fibration. To do this, factor
    $g$ as a composite 
    $Y \xrightarrow{i} Y_0 \xrightarrow{g'} Z$ 
    where $i$ is a level-wise weak equivalence and $g'$
    is a level-wise fibration.
    Then the simplicial map
    \[
        i: N(\mcD^f) 
        \cong N(T(Y)) \times_{N(T(Z))} N(\mcD) 
        \longrightarrow
        N(T(Y_0)) \times_{N(T(Z))} N(\mcD) 
    \] 
    is the pullback of a level-wise weak equivalence 
    along a level-wise fibration, and hence a level-wise 
    weak equivalence. In particular, it induces a 
    weak equivalence on geometric realisations.
    It now suffices to prove the lemma for $g'$.
    
    Assume from now on that $g$ was already a fibration.
    This implies that the pullback 
    $f:Y \times_Z \Obj(\mcD) \to \Obj(\mcD)$ 
    is a fibration, too.
    Since $g$ is also surjective on connected components,
    this implies that $g$ and $f$ are surjective.
    We now claim that the additivity theorem for bordism
    categories \ref{thm:local-add} 
    (see \cite{Stm18} and \cite{Stb19}) applies to the 
    pullback square:
    \[
    \begin{tikzcd}
        \mcD^f \ar[r] \ar[d] 
        & T(Y\times_Z \Obj(\mcD)) \ar[d, "T(f)"] \\
        \mcD \ar[r] & T(\Obj(\mcD)).
    \end{tikzcd}
    \]
    Since $f$ is a fibration, so is $T(f)$.
    The trivial groupoids $T(X)$ are always fibrant
    and unital, and the functors in the square are 
    always unital. 
    The only non-trivial condition to check is that
    $T(f)$ is indeed Cartesian and coCartesian. 
    But this is easy to see since $T(f)$ is surjective
    on objects and there is only one morphism between
    every two objects in $T(X)$.
    
    We may hence apply theorem \ref{thm:local-add}
    and obtain a homotopy pullback square 
    \[
    \begin{tikzcd}
        B\mcD^f \ar[r] \ar[d] \ar[rd, phantom, "\lrcorner^h" near start]
        & B(T(Y\times_Z \Obj(\mcD))) \ar[d, "B(T(f))"] \\
        B\mcD \ar[r] & B(T(\Obj(\mcD))).
    \end{tikzcd}
    \]
    As discussed in the beginning of the proof $B(T(X))$
    is contractible for all $X$ and so the right vertical map
    is an equivalence. Because the square is a homotopy pullback
    this implies that $B\mcD^f \to B\mcD$ is an equivalence.
\end{proof}

\section{Closed and reduced bordism categories}\label{sec:closed-reduced}

We now introduce the key idea of this paper: the decomposition of a 
cobordism into its closed and its reduced part.

\begin{defn}
    A bordism $W:M \to N$ is called \emph{closed} if both $M$ and $N$ are empty
    and it is called \emph{reduced} if the inclusion $M \amalg N \to W$ 
    is surjective on connected components.
    We define subspaces 
    \[
        \Phi_{d,\gt}^{\mrm{cl}} \subset \Phi_{d,\gt} 
        \qand
        \Phi_{d,\gt}^{\mrm{red}} \subset \Phi_{d,\gt} 
    \]
    of closed and reduced bordisms.
    Define retractions 
    $c: \Phi_{d,\gt} \to \Phi_{d,\gt}^{\mrm{cl}}$ and 
    $r: \Phi_{d,\gt} \to \Phi_{d,\gt}^{\mrm{red}}$
    by deleting those connected components of $W$ 
    that violate the respective condition.
    This is illustrated in figure \ref{fig:closed-and-reduced}.
\end{defn}

\begin{figure}[ht]
    \centering
    \def\svgwidth{.4\linewidth}
%% Accompanies image file 'closed-and-reduced.pdf' (pdf, eps, ps)
\begingroup%
  \makeatletter%
  \providecommand\color[2][]{%
    \errmessage{(Inkscape) Color is used for the text in Inkscape, but the package 'color.sty' is not loaded}%
    \renewcommand\color[2][]{}%
  }%
  \providecommand\transparent[1]{%
    \errmessage{(Inkscape) Transparency is used (non-zero) for the text in Inkscape, but the package 'transparent.sty' is not loaded}%
    \renewcommand\transparent[1]{}%
  }%
  \providecommand\rotatebox[2]{#2}%
  \newcommand*\fsize{\dimexpr\f@size pt\relax}%
  \newcommand*\lineheight[1]{\fontsize{\fsize}{#1\fsize}\selectfont}%
  \ifx\svgwidth\undefined%
    \setlength{\unitlength}{168.59276359bp}%
    \ifx\svgscale\undefined%
      \relax%
    \else%
      \setlength{\unitlength}{\unitlength * \real{\svgscale}}%
    \fi%
  \else%
    \setlength{\unitlength}{\svgwidth}%
  \fi%
  \global\let\svgwidth\undefined%
  \global\let\svgscale\undefined%
  \makeatother%
  \begin{picture}(1,0.58873158)%
    \lineheight{1}%
    \setlength\tabcolsep{0pt}%
    \put(0,0){\includegraphics[width=\unitlength,page=1]{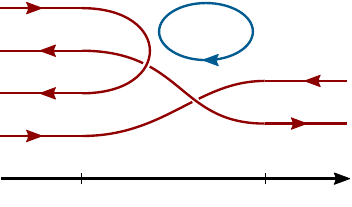}}%
    \put(0.21526239,0.00697379){\color[rgb]{0,0,0}\makebox(0,0)[lt]{\lineheight{1.25}\smash{\begin{tabular}[t]{l}$0$\end{tabular}}}}%
    \put(0.73475019,0.00697354){\color[rgb]{0,0,0}\makebox(0,0)[lt]{\lineheight{1.25}\smash{\begin{tabular}[t]{l}$t$\end{tabular}}}}%
    \put(0,0){\includegraphics[width=\unitlength,page=2]{closed-and-reduced.pdf}}%
    \put(0.75170076,0.52564805){\color[rgb]{0,0,0}\makebox(0,0)[lt]{\lineheight{1.25}\smash{\begin{tabular}[t]{l}$\color{blue}c(W)$\end{tabular}}}}%
    \put(0.50067317,0.18406001){\color[rgb]{0,0,0}\makebox(0,0)[lt]{\lineheight{1.25}\smash{\begin{tabular}[t]{l}$\color{red}r(W)$\end{tabular}}}}%
  \end{picture}%
\endgroup%

    \caption{An oriented $1$-dimensional cobordism decomposed into its closed and reduced components.}
    \label{fig:closed-and-reduced}
\end{figure}

Every bordism $W \in \Phi_{d,\gt}$ has a canonical decomposition as
$W = c(W) \amalg r(W)$.
A homotopically meaningful version of this statement 
is proved in lemma \ref{lem:c-r-decomposition}.
Our goal is to obtain a similar decomposition on the level of bordism categories.
The first step is to define categories of closed and of reduced bordisms.

\begin{defn}\label{defn:Cobcl/red}
    The \emph{closed} bordism category $\Cobcl_{d,\gt}$ is the full subcategory 
    of $\Cob_{d,\gt}$ on the single object $\emptyset \in \Cob_{d,\gt}$.
    The \emph{reduced} bordism category $\Cobred_{d,\gt}$ has the same objects
    as $\Cob_{d,\gt}$, but has as morphisms only reduced bordisms.
    Composition is defined by composing in $\Cob_{d,\gt}$ and then 
    applying the retraction $r: \Phi_{d,\gt} \to \Phi_{d,\gt}^{\mrm{red}}$.
\end{defn}

\begin{defn}
    The inclusion $\Phi_{d,\gt}^{\mrm{cl}} \inj \Phi_{d,\gt}$ defines a functor
    denoted by
    \[
        I: \Cobcl_{d,\gt} \longrightarrow \Cob_{d,\gt}
    \]
    and the retraction $r:\Phi_{d,\gt} \to \Phi_{d,\gt}^{\mrm{red}}$ 
    defines a functor denoted by
    \[ 
        R: \Cob_{d,\gt} \to \Cobred_{d,\gt}.
    \]
\end{defn}

\begin{rem}
    The functors $I$ and $R$ allow us to think of $\Cobcl_{d,\gt}$
    as a subcategory and of $\Cobred_{d,\gt}$ as a quotient category 
    of $\Cob_{d,\gt}$.
    It is important to note that the opposite is not the case: the inclusion 
    $\Phi_{d,\gt}^{\mrm{red}} \inj \Phi_{d,\gt}$ does not induce a functor
    $\Cobred_{d,\gt} \to \Cob_{d,\gt}$ 
    The reason for this is that the composite of two reduced bordisms 
    $W:M \to N$ and $V:N \to L$ is not necessarily reduced. 
    We could for example set $M = L = \emptyset$, $N= S^{d-1}$,
    and let $W$ and $V$ both be the $d$-dimensional disc; 
    once as a bordism $\emptyset \to S^{d-1}$ and once as a bordism
    $S^{d-1} \to \emptyset$.
    Then both $W$ and $V$ are reduced, but 
    $W \cup_N V = D^d \cup_{S^{d-1}} D^d = S^d: \emptyset \to \emptyset$ 
    is not.
    See figure \ref{fig:not-reduced} for another example showing
    that the composite of two reduced bordisms need not be reduced.
\end{rem}

\begin{figure}[ht]
    \centering
    \def\svgwidth{.6\linewidth}
%% Accompanies image file 'not-reduced.pdf' (pdf, eps, ps)
\begingroup%
  \makeatletter%
  \providecommand\color[2][]{%
    \errmessage{(Inkscape) Color is used for the text in Inkscape, but the package 'color.sty' is not loaded}%
    \renewcommand\color[2][]{}%
  }%
  \providecommand\transparent[1]{%
    \errmessage{(Inkscape) Transparency is used (non-zero) for the text in Inkscape, but the package 'transparent.sty' is not loaded}%
    \renewcommand\transparent[1]{}%
  }%
  \providecommand\rotatebox[2]{#2}%
  \newcommand*\fsize{\dimexpr\f@size pt\relax}%
  \newcommand*\lineheight[1]{\fontsize{\fsize}{#1\fsize}\selectfont}%
  \ifx\svgwidth\undefined%
    \setlength{\unitlength}{326.83926079bp}%
    \ifx\svgscale\undefined%
      \relax%
    \else%
      \setlength{\unitlength}{\unitlength * \real{\svgscale}}%
    \fi%
  \else%
    \setlength{\unitlength}{\svgwidth}%
  \fi%
  \global\let\svgwidth\undefined%
  \global\let\svgscale\undefined%
  \makeatother%
  \begin{picture}(1,0.40720418)%
    \lineheight{1}%
    \setlength\tabcolsep{0pt}%
    \put(0,0){\includegraphics[width=\unitlength,page=1]{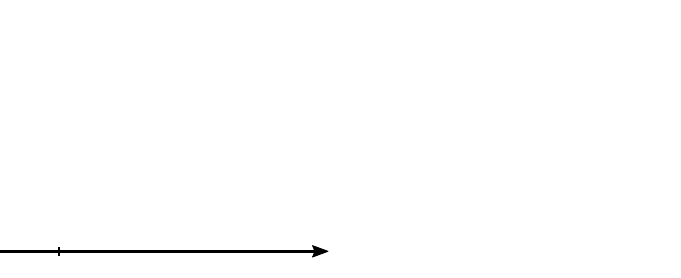}}%
    \put(0.07752801,-0.00659855){\color[rgb]{0,0,0}\makebox(0,0)[lt]{\lineheight{1.25}\smash{\begin{tabular}[t]{l}$0$\end{tabular}}}}%
    \put(0.39223061,-0.00659855){\color[rgb]{0,0,0}\makebox(0,0)[lt]{\lineheight{1.25}\smash{\begin{tabular}[t]{l}$t$\end{tabular}}}}%
    \put(0,0){\includegraphics[width=\unitlength,page=2]{not-reduced.pdf}}%
    \put(0.64522226,-0.00659855){\color[rgb]{0,0,0}\makebox(0,0)[lt]{\lineheight{1.25}\smash{\begin{tabular}[t]{l}$0$\end{tabular}}}}%
    \put(0.90419625,-0.00659855){\color[rgb]{0,0,0}\makebox(0,0)[lt]{\lineheight{1.25}\smash{\begin{tabular}[t]{l}$s$\end{tabular}}}}%
    \put(0,0){\includegraphics[width=\unitlength,page=3]{not-reduced.pdf}}%
  \end{picture}%
\endgroup%

    \caption{Two reduced bordisms whose composite is not reduced.}
    \label{fig:not-reduced}
\end{figure}

Locally the bordism category $\Cob_{d,\gt}$ decomposes as a product
of $\Cobcl_{d,\gt}$ and $\Cobred_{d,\gt}$:
\begin{lem}\label{lem:c-r-decomposition}
    For any two objects $M, N \in \Cob_{d,\gt}$ the map 
    \[
        (c,r): \hom_{\Cob_{d,\gt}}(M, N) \longrightarrow 
        \hom_{\Cobcl_{d,\gt}}(\emptyset, \emptyset) \times \hom_{\Cobred_{d,\gt}}(M, N) 
    \]
    is a weak equivalence.
\end{lem}
\begin{proof}
    We will show that the map that decomposes a bordism into 
    its closed and its reduced part
    \begin{equation}\label{eqn:c-r-decomp}
        (c,r): \Phi_{d,\gt} \longrightarrow 
        \Phi_{d,\gt}^{\mrm{cl}} \times \Phi_{d,\gt}^{\mrm{red}}
    \end{equation}
    is a weak equivalence.
    
    Once this is established the lemma follows by observing that both
    sides of (\ref{eqn:c-r-decomp}) are fibered over $\Phi_{d-1,\gt} \times \Phi_{d-1,\gt}$ 
    via the boundary map $\partial = (\partial_0, \partial_1)$,
    which is a Serre fibration because $\Cob_{d,\gt}$ is fibrant, 
    see \cite[Proposition 3.2.4(ii)]{ERW19b}.
    This implies the claim because a weak equivalence between
    Serre fibrations induces weak equivalences on the fibers.
    
    We begin by observing that $(c,r)$ is an open embedding.
    As a map of sets $(c,r)$ is an injection with image those tuples 
    $(W, V) \in \Phi_{d,\gt}^{\mrm{cl}} \times \Phi_{d,\gt}^{\mrm{red}}$
    such that $W$ and $V$ are disjoint.
    That it is open and continuous follows by a close examination of 
    the topology defined in \cite[Definition 2.1]{GRW10}. 
    We leave the details to the reader.
    
    To show that $(c,r)$ is a weak equivalence we construct two subspaces 
    $A, B \subset \Phi_{d,\gt}$ defined by: 
    \begin{itemize}
     \item $(W,l) \in A$ if for all $x=(x_0,x_1,\dots) \in W$ we have $x_1 = 0$.
     \item $(W,l) \in B$ if for all $x \in c(W)$ we have $x_1 = \oh$ 
     and for all $x \in r(W)$ we have $x_1 = -\oh$.
    \end{itemize}
    There is a map $s:A \to B$ that takes $(W,l)$ and moves up $c(W)$ by
    $\oh$ and moves down $r(W)$ by $\oh$ in the first coordinate.
    This map and the inclusions fit into a homotopy commutative diagram
    \[
        \begin{tikzcd}
            A \ar[r, "s"] \ar[d,hook, "i"] & B \ar[d, "k"] \ar[dl, hook, "j"']\\
            \Phi_{d,\gt} \ar[r, "{(c,r)}"] 
            & \Phi_{d,\gt}^{\mrm{cl}} \times \Phi_{d,\gt}^{\mrm{red}}.
        \end{tikzcd}
    \]
    The map $k: B \to \Phi_{d,\gt}^{\mrm{cl}} \times \Phi_{d,\gt}^{\mrm{red}}$
    is defined by restricting $(c,r)$, so the bottom triangle commutes by 
    definition. To see that the top triangle commutes up to homotopy 
    consider the map $s_\gl:A \to \Phi_{d,\gt}$ defined for $\gl \in [0,\oh]$
    by moving up $c(W)$ by $\gl$ and moving down $r(W)$ by $\gl$.
    This defines a homotopy between $i$ and $j \circ s$.
    
    Next, we observe that $i$ and $k$ are homotopy equivalences.
    In both cases this follows because the inclusion
    $\IR \times \{0\} \times (-1,1)^{\infty-1} \to \IR \times \cube$ 
    admits a deformation retraction via smooth embeddings preserving the 
    $x_0$ coordinate.
    Given a homotopy commutative diagram of the shape as above where 
    $i$ and $k$ are weak equivalences, 
    it follows from the $2$-out-of-$6$ property that all 
    other maps are also weak equivalences.
    Concretely, we can apply the functor $\pi_*$ and observe that because
    $\pi_*(i)$ and $\pi_*(k)$ are isomorphisms, $\pi_*((c,r))$ has
    to be an isomorphism, too.
\end{proof}

\begin{lem}\label{lem:allBCob-vspecial}
    The spaces 
        $B(\Cobcl_{d,\gt})$,
        $B(\Cobred_{d,\gt})$,
        $B(h\Cobcl_{d,\gt})$, and 
        $B(h\Cobred_{d,\gt})$
    are very special $\gC$-spaces.
\end{lem}
\begin{proof}
    This is proved just like lemma \ref{lem:BCob-vspecial} and 
    \ref{lem:BhCob-vspecial}. 
    To check that they are \emph{very} special, note that 
    $\pi_0 B(\Cobcl_{d,\gt}) = \pi_0 B(h\Cobcl_{d,\gt}) = *$ is trivial
    and that $\pi_0 B(\Cobred_{d,\gt}) = \pi_0 B(h\Cobred_{d,\gt}) 
    = \pi_0 B(\Cob_{d,\gt})$ is the $\gt$-structured bordism group.
\end{proof}

We conclude this section by a lemma indicating that the functor $R$ 
is homotopically well-behaved.
\begin{lem}\label{lem:r-Serre}
    The retraction
    $
        r:\Phi_{d,\gt} \to \Phi_{d,\gt}^{\mrm{red}}
    $
    is a Serre fibration
    and hence the functor $R:\Cob_{d,\gt} \to \Cobred_{d,\gt}$ 
    induces a level-wise fibration on the nerves.
\end{lem}
\begin{proof}
    To simplify notation we will ignore the tangential structures in
    this proof, though they can easily be added.
    
    For every map $W:[0,1]^k \to \Phi_{d}^{\mrm{red}}$ and
    any lift $V:[0,1]^{k-1} \to \Phi_{d}$ with 
    $r(V(s_1,\dots,s_{k-1})) = W(s_1,\dots,s_{k-1},0)$ 
    we need to find an extension $V':[0,1]^k \to \Phi_{d}$ 
    such that $r(V'(s_1,\dots,s_k)) = W(s_1,\dots,s_k)$.
    
    Recall that we defined $\Phi_{d}$ as a colimit over
    submanifolds of $\IR^N$ for $N \to \infty$. Since $[0,1]^k$ is compact,
    the image of $W$ and $V$ is contained in a finite stage of the colimit. 
    We can therefore find $N$ such that 
    $W(s_1,\dots,s_{k-1}), V(s_1,\dots,s_k) \subset \IR^{N-1}$ for all $s_i$.
    Let now $T_\gl: \IR^\infty \to \IR^\infty$ be the diffeomorphism 
    defined by applying $x_N \mapsto \gl + x_N$ to the $N$th coordinate.
    We use it to define
    \[
        V'(s_1, \dots, s_k) := W(s_1, \dots, s_k) \cup T_{s_k/2}( c(V(s_1,\dots,s_{k-1}))).
    \]
    The union of $W(\dots)$ and $T_{s_k/2}(\dots)$ is in fact a disjoint union:
    the first part has $N$th coordinate $0$ and the second part
    has $N$th coordinate $s_k/2$.
    For $s_k=0$ they are disjoint by assumption.
    Since the union is disjoint $V(s_1,\dots,s_k)$ is indeed a manifold
    and a well-defined point in $\Phi$.
    
    It remains to check that $V$ is actually continuous.
    As remarked in the proof of lemma \ref{lem:c-r-decomposition} 
    the topology on $\Phi$ 
    is equal to the subspace topology induced by the injection 
    $(c,r):\Phi \inj \Phi^{\mrm{cl}} \times \Phi^{\mrm{red}}$.
    It hence suffices to check that both $c(V)$ and $r(V)$ depend 
    continuously on $(s_1,\dots,s_k)$, but this is obvious by construction.
\end{proof}

\section{The reduction fiber sequence}
\label{sec:reduction-fibseq}

Our main technical theorem is what we call the ``reduction fiber sequence''.
It establishes a decomposition of the classifying space of the bordism category
into its closed and reduced components. 
\begin{thm}\label{thm:reduction-fiber-sequence}
    For any dimension $d$ and tangential structure $\gt$ the rows 
    in the commutative diagram of infinite loop spaces 
    \[
        \begin{tikzcd}
            B \Cobcl_{d,\gt} \ar[r] \ar[d] & B \Cob_{d,\gt} \ar[r, "R"] \ar[d] & B \Cobred_{d,\gt} \ar[d] \\
            B (h\Cobcl_{d,\gt}) \ar[r] & B (h\Cob_{d,\gt}) \ar[r, "R"] & B (h\Cobred_{d,\gt})
        \end{tikzcd}
    \]
    are homotopy fiber sequences.
\end{thm}

This is similar to lemma \ref{lem:c-r-decomposition} where we observed 
a splitting of the morphism spaces.
Note, however, that because $c$ is not functorial 
there is no splitting as a product on the level of categories,
but only a homotopy fiber sequence.

\begin{warning}
    The vertical maps in the theorem are in fact zig-zags 
    $B\Cob_{d,\gt} \leftarrow B(\gd \Cob_{d,\gt}) \to B(h\Cob_{d,\gt})$
    as described in definition \ref{defn:discrete-cat}. 
    We will usually omit the left-ward arrow as it depends naturally on $\Cob$
    and is a weak equivalence by lemma \ref{lem:gd-doesnt-change-B}.
\end{warning}

The essential technical ingredient for the proof of theorem 
\ref{thm:reduction-fiber-sequence} is the Steimle's
``additivity theorem for bordism categories'' from \cite{Stm18}.
As we will see in remark \ref{rem:Stm-not-sufficient} the version
proved by Steimle is not sufficient for establishing the reduction 
fiber sequence and we will instead need the following generalised 
``local'' version, which we proved in \cite{Stb19} 
to study the category of cospans.

\begin{thm}[Local additivity theorem for bordism categories, 
{\cite[Theorem 5.8]{Stb19}}]\label{thm:local-add}
    Let $P: \mcE \to \mcB$ be a weakly unital functor of weakly unital topological categories
    such that $\mcB$ is fibrant, $P$ is a local fibration, $P$ is \emph{locally} Cartesian,
    and $P$ is \emph{locally} coCartesian. 

    Then, for every weakly unital functor $F:\mcC \to \mcB$,
    the following is a homotopy pullback square:
    \[
        \begin{tikzcd}
            B(\mcC \times_{\mcB} \mcE) \ar[r] \ar[d] & B\mcE \ar[d, "P"] \\
            B\mcC \ar[r, "F"] & B\mcB .
        \end{tikzcd}
    \]
\end{thm}

\begin{rem}
    All non-unital categories in this paper have weak units in the form 
    of cylinders and all functors considered will preserve these.
    We will recall the other technical terms as needed.
\end{rem}

The main step in verifying the conditions of the additivity theorem
will be to check that the functor $R:\Cob_{d,\gt} \to \Cobred_{d,\gt}$ 
is indeed locally (co)Cartesian.
This is a consequence of lemma \ref{lem:c-r-decomposition}.

\begin{lem}\label{lem:reduced=locCart}
    A morphism $W:M \to N$ in $\Cob_{d,\gt}$ is locally $R$-Cartesian
    in the sense of \cite[Definition 5.6]{Stb19}, if and only if it is reduced.
\end{lem}
\begin{proof}
    According to \cite[Definition 5.6]{Stb19} $W$ is locally $R$-Cartesian
    if and only if for any (and by \cite[Lemma 5.9]{Stb19} equivalently all) equivalences
    $V:M' \to M$ in $\Cobred_{d,\gt}$ the following diagram
    \[
        \begin{tikzcd}[column sep = 4pc]
            \hom_{\Cob_{d,\gt}}(M', M) \ar[r, "(\blank \cup_M W)"] \ar[d,"R"] &
            \hom_{\Cob_{d,\gt}}(M, N) \ar[d,"R"] \\
            \hom_{\Cobred_{d,\gt}}(M', M) \ar[r, "r(\blank \cup_M r(W))"] &
            \hom_{\Cobred_{d,\gt}}(M, N) 
        \end{tikzcd}
    \]
    induces a weak equivalence between the homotopy fiber of $R$ at $V$
    and the homotopy fiber of $R$ at $r(V \cup_M W) = r(V \cup_M r(W))$.
    The choice of the equivalence $V$ does not matter, but to be concrete
    we can take $M' = M$ and $V:M \to M$ to be the cylinder of length one.
    Using lemma \ref{lem:c-r-decomposition} we can identify the homotopy fibers with 
    $\hom_{\Cobcl_{d,\gt}}(\emptyset, \emptyset)$. 
    Note that in light of lemma \ref{lem:r-Serre} these homotopy fibers are equivalent
    to the genuine fibers. The induced map between the
    homotopy fibers is exactly composition with $c(V \cup_M W)$ in $\Cobcl_{d,\gt}$:
    \[
        (\blank \cup_\emptyset c(V \cup_M W)): \hom_{\Cobcl_{d,\gt}}(\emptyset, \emptyset) 
        \longrightarrow \hom_{\Cobcl_{d,\gt}}(\emptyset, \emptyset).
    \]
    This map is homotopic to the identity if $c(V\cup_M W)= (\emptyset, t)$ 
    is the empty cobordism of some length $t>0$, and since $V$ is a cylinder
    this is the case if and only if $W$ is reduced.
    Therefore $W$ is $R$-Cartesian if it is reduced.
    
    Conversely, if $W$ is not reduced, then $c(V\cup_M W)$ is non-empty and the map 
    $(\blank \cup_\emptyset c(V\cup_M W))$ cannot be an equivalence as it does not 
    hit the connected component of 
    $(\emptyset, t) \in \hom_{\Cobcl_{d,\gt}}(\emptyset, \emptyset)$.
\end{proof}

\begin{lem}\label{lem:hreduced=locCart}
    A morphism $[W]:M \to N$ in $h\Cob_{d,\gt}$ is locally $(hR)$-Cartesian
    in the sense of \cite[Definition 5.6]{Stb19}, if and only if it is reduced.
\end{lem}
\begin{proof}
    This is proved just like lemma \ref{lem:reduced=locCart},
    but using the bijection 
    \[
        \hom_{h\Cob_{d,\gt}}(M,N)
        \cong \hom_{h\Cobcl_{d,\gt}}(\emptyset,\emptyset) 
        \times \hom_{h\Cobred_{d,\gt}}(M,N).
    \]
    It is important to note that, while the proofs are completely 
    analogous, the statement of this lemma is not a formal consequence 
    of the statement of lemma \ref{lem:reduced=locCart}. 
    See warning \ref{war:hCart}.
\end{proof}

\begin{rem}\label{rem:Stm-not-sufficient}
    As we will see in the proof of theorem \ref{thm:reduction-fiber-sequence},
    this lemma implies that $R:\Cob_{d,\gt} \to \Cobred_{d,\gt}$ is locally Cartesian
    and, by reversing bordisms, also locally coCartesian.
    Note, however, that usually $R$ is neither Cartesian nor coCartesian.
    A morphism $W:M \to N$ in $\Cob_{d,\gt}$ is $R$-Cartesian if 
    the square in lemma \ref{lem:reduced=locCart} is a homotopy pullback for all $M'$.
    In particular, the vertical homotopy fibers at $V$ and $r(V \cup_M r(W))$ 
    have to agree for all reduced $V:M' \to M$ and not just the equivalences.
    It follows from the proof of lemma \ref{lem:reduced=locCart} that 
    for this to be the case $c(V\cup_M W)$ has to be empty.
    Therefore, if $W:M \to N$ and $V:M' \to M$ are reduced morphisms 
    such that $V \cup_M W$ is not reduced, then there cannot be an $R$-Cartesian 
    morphism $\wt{W}:M \to N$ with $r(\wt{W}) = W$.
    
    Note that reduced bordisms $W$, $V$ with $V \cup_M W$ not reduced 
    exist in all bordism categories with $d>0$ and $\gt \neq \emptyset$.%
    \footnote{
        It suffices to construct a counterexample for $\gt = \mi{GL}_d$ 
        and all $d >0$, then the others can be produced using functoriality
        of $\Cob_{d,\gt}$ in the $\mi{GL}_d$-space $\gt$.
        So, consider the bordism category of framed $1$-manifolds $\Cob_{1, \mi{GL}_1}$.
        Here a counterexample is given by $V = D^1: \emptyset \to \{+,-\}$
        and $W = D^1: \{+, -\} \to \emptyset$ with the obvious framing.
        The counterexample in $\Cob_{d, \mi{GL}_d}$ is now obtained by
        taking the product with $(S^1)^{d-1}$ with the Lie-group framing.
    }
    This shows that Steimle's original version of the additivity theorem for bordism 
    categories \cite{Stm18} would not be sufficient for our purposes.
\end{rem}

\begin{proof}[Proof of Theorem \ref{thm:reduction-fiber-sequence}]
    Let $\mc{R}_+$ denote the topological semigroup $(\IR_{>0},+)$ thought of as a 
    non-unital topological category with one object. For all variants of the bordism 
    category there is a canonical functor $\Cob \to \mc{R}_+$ that records the 
    length of a bordism. Picking any object $M \in \Cob$ we also have a functor
    $\mrm{Cyl}_M: \mc{R}_+ \to \Cob$ that sends $t$ to $(t,\IR \times M)$,
    the cylinder of length $t$ on $M$.
    Using these functors we construct a commutative diagram
    \[
        \begin{tikzcd}
            \Cobcl_{d,\gt} \ar[d] \ar[r] & \Cob_{d,\gt} \ar[d, "R"] \\
            \mc{R}_+ \ar[r, "\mrm{Cyl}_\emptyset"] & \Cobred_{d,\gt}.
        \end{tikzcd}
    \]
    This is trivially a pullback diagram on the spaces of objects.
    By the proof of lemma \ref{lem:c-r-decomposition} it also is a pullback diagram
    on the spaces of morphisms,
    and hence it is a pullback diagram of categories internal to topological spaces.
    
    We now wish to apply the local additivity theorem for bordism categories 
    of \cite{Stb19} to this, which we recalled as theorem \ref{thm:local-add}.
    The fibrancy condition on $\Cobred_{d,\gt}$ follows from 
    the fibrancy for $\Cob_{d,\gt}$, see \cite[Proposition 3.2.4(ii)]{ERW19b}
    and the fibrancy condition on $R$ follows from lemma \ref{lem:r-Serre}.
    The functor $R$ is locally Cartesian since any morphism $W:M \to N$ 
    in $\Cobred_{d,\gt}$ has a lift to $\Cob_{d,\gt}$ given by $W$ itself.
    This lift is reduced and hence by lemma \ref{lem:reduced=locCart} locally Cartesian.
    Turning around all bordisms involved we see by the same argument that $R$
    is locally coCartesian.
    Finally, we note that the categories involved are weakly unital with the 
    weak units given by the cylinders $(t, \IR \times M)$ 
    described above. 
    The functors in the diagram preserve the cylinders and are hence weakly unital.
    This checks all conditions of theorem \ref{thm:local-add}.
    
    As a result of the additivity theorem we see that 
    \[
        \begin{tikzcd}
            B\Cobcl_{d,\gt} \ar[d] \ar[r] & B\Cob_{d,\gt} \ar[d, "B(R)"] \\
            B\mc{R}_+ \ar[r] & B\Cobred_{d,\gt}
        \end{tikzcd}
    \]
    is a homotopy pullback diagram. The non-unital topological category 
    $\mc{R}_+$ is weakly equivalent to the terminal unital category 
    and hence $B\mc{R}_+$ is weakly contractible. 
    This shows that 
    \[
        B\Cobcl_{d,\gt} \longrightarrow B\Cob_{d,\gt} \longrightarrow B\Cobred_{d,\gt}
    \]
    is a homotopy fiber sequence.
    
    It remains to discuss the homotopy fiber sequence for the homotopy categories
    $h\Cob_{d,\gt}$. Basically all arguments can be copied from the topological case,
    (using lemma \ref{lem:hreduced=locCart} instead of lemma \ref{lem:reduced=locCart})
    with the additional simplification that all morphism spaces are discrete and
    hence all fibrancy conditions are trivially satisfied.
    In particular this shows that the second sequence in the theorem is a homotopy fiber sequence.
    
    The vertical maps in the theorem are induced by the canonical 
    zig-zag of functors $\mcD \leftarrow \gd\mcD \to h\mcD$ 
    that we have for any topologically category
    and it is clear that the diagram commutes.
    We know from lemma \ref{lem:BCob-vspecial}, \ref{lem:BhCob-vspecial}, 
    and \ref{lem:allBCob-vspecial} that all six spaces are infinite loop spaces 
    and that the maps respect this structure. Hence the diagram is a diagram of infinite 
    loop spaces and in particular the two fiber sequences are fiber sequences 
    of infinite loop spaces.
\end{proof}

\begin{warning}\label{war:hCart}
    There are topologically enriched functors $F:\mcE \to \mcB$ satisfying the 
    conditions of the (local) additivity theorem for bordism categories such that
    $hF:h\mcE \to h\mcB$ does not satisfy the conditions of the theorem.
    The reason for this is that the property of being of (locally) (co)Cartesian
    is not preserved under taking homotopy categories.
    This for example fails when considering Genauer's sequence: Steimle's proof 
    of the Genauer fiber sequence \cite[Theorem 1.1]{Stm18} does not imply that there
    is an analogous homotopy fiber sequence
    \[
        B(h\Cob_d) \longrightarrow B(h\Cob_d^\partial) 
        \xrightarrow{\ \partial\ } B(h\Cob_{d-1})
    \]
    for the classifying spaces of the homotopy categories.
    In fact, the author believes that this 
    cannot be a homotopy fiber sequence for any $d \ge 1$.
\end{warning}

\section[Computations for d=1]{Computations for $d=1$}\label{sec:Computations-d=1}

\subsection{The classifying space of the closed bordism category}
The closed bordism category is a topological monoid 
and we think of it as the coherently commutative monoid freely generated
by the connected closed manifolds. In lemma \ref{lem:BCobcl} we describe its 
classifying space as the free infinite loop space 
generated by $B\Diff^\gt(W)$ for all \emph{connected} 
closed diffeomorphism types $W$.
In dimension one, the only connected closed manifold is the circle
and we can describe $B\Cobcl_{1,\gt}$ more explicitly in terms of the 
homotopy orbits of the space of free loops, see corollary \ref{cor:Bcobcl1}.

\begin{lem}\label{lem:BCobcl}
    The maps
    $\ga_{W}: \gS B\Diff^\gt(W) \to B\Cobcl_{d,\gt}$
    from definition \ref{defn:MW-map} induce a weak equivalence of
    infinite loop spaces:
    \[
        Q\left( \gS \left(
        \coprod_{[W] \text{ connected closed}} B\Diff^\gt(W)
        \right)_+ \right) 
        \xrightarrow{\ \simeq\ } B\Cobcl_{d,\gt}.
    \]
\end{lem}
\begin{proof}
    Let us first consider the case $d=0$ with $\gt = X$ some space.
    Then a morphism in $\Cobcl_{0,X}$ is a finite subset of 
    $(0,t) \times \cube$, equipped with a map to $X$. 
    In other words, the space of morphisms of $\Cobcl_{0,X}$ 
    is the $X$-labelled configuration space $\Conf(\cube; X)$ 
    with the usual composition defined by putting 
    configurations side by side.
    Segal's improved Barrat-Priddy-Quillen theorem 
    \cite[Proposition 3.6]{Seg74} states that the group-completion 
    of the special (but not very special) $\gC$-space $\Conf(\cube; X)$ is 
    \[ 
        \gO B( \Conf(\cube; X) ) \simeq Q(X_+).
    \]
    The inverse map comes from the map $X \to \Conf(\cube; X)$ 
    that sends $x$ to the point $0$ labelled by $x$.
    Since $\Cobcl_{0,X}$ is this monoid, thought of as a category with one object,
    we conclude that $B\Cobcl_{0,X}$ is the delooping
    $\gO^{-1}Q(X_+) \simeq Q(\gS X_+)$. 
    Seeing as $X \cong \Bun^\gt(*) \cong B\Diff^\gt(*)$ 
    we therefore have an equivalence 
    $Q(\gS X_+) \simeq B\Conf(\cube; X) \simeq B\Cobcl_{0,X}$.
    Restricted to $\gS X \subset Q(\gS X_+)$ this equivalence
    agrees with the map $\ga_{W=*}: \gS B\Diff^\gt(*) \to B\Cobcl_{0,X}$
    and therefore the equivalence we constructed agrees with 
    the one in the claim.
    The lemma follows in the case $d=0$.

    To reduce the case $d>0$ to the case $d=0$ we define $X$ as
    the space of closed connected $\gt$-structured submanifolds of $\cube$ 
    so that \( X \simeq \coprod_{[W] \text{ con.}} B\Diff^\gt(W) \).
    If we can show that $\Cobcl_{d,\gt}$ and $\Cobcl_{0, X}$ 
    are weakly equivalent as non-unital topological $\gC$-categories,
    then the claim for $d>0$ follows from the first part of the proof:
    \[
        Q\left( \gS \left(
        \coprod_{[W] \text{ connected}} B\Diff^\gt(W)
        \right)_+ \right) 
        \simeq Q(\gS X_+) 
        \simeq B\Cobcl_{0, X} 
        \stackrel{?}{\simeq} B\Cobcl_{d,\gt}.
    \]
    
    To obtain the weak equivalence 
    $\Cobcl_{d,\gt} \simeq \Cobcl_{0,X}$ we first convince ourselves
    that the morphism spaces are abstractly equivalent.
    To see this take any closed $d$-dimensional manifold $W$ and 
    decompose it as $W = \amalg_{i=1}^n \amalg_{j = 1}^{m_i} W_{i,j}$
    where each $W_{i,j}$ is connected and $W_{i,j} \cong W_{i',j'}$ 
    iff $i = i'$.
    Then 
    \begin{align*}
        B\Diff^\gt(W) & = \Bun^{\gt}(W) \doublebs \Diff(W)
        \cong \left(\prod_{i=1}^n (\Bun^\gt(W_{i,1}))^{m_i}\right)\doublebs
        \left(\prod_{i=1}^n \Diff(W_{i,1}) \wr \gS_{m_i}\right) \\
        & \simeq \prod_{i=1}^n \left( (\Diff^\gt(W_{i,1}))^{m_i})\doublebs
         \gS_{m_i}\right)
        \simeq \prod_{i=1}^n \left( \Conf_{m_i}(\cube; \Diff^\gt(W_{i,1}) \right).
    \end{align*}
    This shows that the respective connected components of 
    $\mrm{Mor}(\Cobcl_{d,\gt})$ and $\mrm{Mor}(\Cobcl_{0,X})$ 
    are abstractly equivalent.  
    
    To obtain the desired equivalence of infinite loop spaces
    $B\Cobcl_{d,\gt} \simeq B\Cobcl_{0,X}$ we will construct a 
    zigzag $\Cobcl_{d,\gt} \leftarrow \mcD \rightarrow \Cobcl_{0,X}$ 
    of $\gC$-categories inducing the equivalence on the morphisms spaces.
    The new $\gC$-category $\mcD$ also has one object.
    A morphism in $\mcD$ is a tuple $(W, i, j, l, t)$ 
    where $t>0$ is the length, $W$ is a closed $d$-dimensional manifold,
    $i:W \inj \cube \times (0,t)$ and $j:(\pi_0 W) \inj \cube \times (0,t)$
    are embeddings, and $l \in \Bun^\gt(W)$ is a $\gt$-structure on $W$.
    Such a tuple is identified with another tuple $(W', i', j', l', t')$
    if $t=t'$ and there is a diffeomorphism $\gp: W \cong W'$ such that
    $i = i' \circ \gp$, $j = j' \circ (\pi_0 \gp)$, and $l = \gp^*l'$.
    The $\gC$-structures are defined just like for $\Cobcl_{d,\gt}$.

    There is a projection map $p:\mcD \to \Cobcl_{d,\gt}$ that forgets $i$.
    There also is a projection map $q:\mcD \to \Cobcl_{0,X}$ that sends
    $[W,i,j,l]$ to the configuration $j(\pi_0 W) \subset \cube$,
    where each point $j([V])$ is labelled by the connected
    component $V\subset W$ with tangential structure $l_{|V}\in \Bun^\gt(V)$.
    Both maps $p$ and $q$ are functors of $\gC$-categories by construction.
    It is not hard to see that they induce the abstract equivalence 
    described above.
\end{proof}

\begin{defn}
    We will let $\IT$ denote the Lie group $SO(2)$.
    The free loop space $LX:= \Map(S^1, X)$ admits an action of $\IT$ 
    by precomposition. We will denote the homotopy orbits of this action by
    \[
       (LX)_{h\IT} := LX\doublebs \IT = (E\IT \times LX)/\IT.
    \]
\end{defn}

\begin{lem}\label{lem:LXhIT}
    Consider the tangential structure $\gt = X \times \gt^{or}$ where
    $\gt^{or} = \{-1,1\}$ with the non-trivial action of $\GL_1$.
    The moduli space of connected closed $1$-manifolds with $\gt$-structure is
    \[
        B\Diff^\gt(S^1) \simeq (LX)_{h\IT}.
    \]
\end{lem}
\begin{proof}
    We compute
    \begin{align*}
        B\Diff^{\gt}(S^1) 
        &= \Bun_{\gt}(S^1)\doublebs \Diff(S^1)
        \simeq \Map_{\GL_1}(S^1 \times \{\pm 1\}, X \times \{\pm 1\})\doublebs \Diff(S^1)\\
        &\simeq \left( \Map(S^1, X) \times \Map(S^1, \{\pm 1\}) \right) \doublebs \Diff(S^1).
    \end{align*}
    Inclusion of the subgroup $\IZ/2 \ltimes \IT \cong O(2) \subset \Diff(S^1)$ 
    is a weak equivalence and we may hence compute the homotopy orbits 
    by first taking $(\blank)_{h\IT}$ and then taking $(\blank)_{h\IZ/2}$.
    As $(\blank)_{h\IT}$ commutes with coproducts this results in
    \[  
        B\Diff^\gt(S^1) \simeq \left( (LX)_{h\IT} \times \{\pm 1\} \right)_{h\IZ/2} 
        \simeq (LX)_{h\IT}
    \]
    as claimed.
\end{proof}

Combining lemma \ref{lem:BCobcl} and \ref{lem:LXhIT} we recover a computation 
of $B\Cobcl_{1, X \times \gt^{\mrm{or}}}$ that was stated in 
\cite[Proposition 5.1]{Gia17} without proof:
\begin{cor}\label{cor:Bcobcl1}
    For all spaces $X$ there is a weak equivalence of infinite loop spaces
    \[
        Q(\gS_+ (LX)_{h\IT}) \simeq B\Cobcl_{1, X \times \gt^{\mrm{or}}}.
    \]
\end{cor}

\subsection{The homotopy category of the $1$-dimensional bordism category}\label{subsec:1d-oriented}
We now have all the tools ready to prove \autoref{theorem:1}.
We will first compute the classifying space of the reduced bordism 
category $\Cobred_1$, then compare it to $h\Cobred_1$, and finally
compute $B(h\Cob_1)$.
\begin{thm}\label{thm:BCobred1}
    There is an equivalence of infinite loop spaces
    \[
        B\Cobred_{1} \simeq \gO^{\infty - 2} MTSO_2.
    \]
\end{thm}

\begin{proof}
    We will use the reduction fiber sequence of theorem 
    \ref{thm:reduction-fiber-sequence} for $\Cob_{1}$:
    \[
        \begin{tikzcd}
            B \Cobcl_{1} \ar[r] 
            & B \Cob_{1} \ar[r, "R"] 
            & B \Cobred_{1}.
        \end{tikzcd}
    \]
    Since $R$ is surjective on connected components 
    (in fact $R:\pi_0 B \Cob_1 \to \pi_0 B\Cobred_1$ is a bijection)
    this remains a homotopy fiber sequence after we deloop 
    each of the infinite loop spaces once:%
    \footnote{
        To see that this is indeed a homotopy fiber sequence, 
        consider the canonical map 
        $B(B\Cobcl_{1}) \to \mi{hofib}(B(B\Cob_{1}) \to B(B\Cobred_{1})$.
        We know that it becomes an equivalence after applying $\gO$
        and we know that the left-hand space is connected.
        So all that is left to show is that the right-hand space 
        is also connected, but this is a consequence of $\pi_0R$ being surjective.
    }
    \[
        B(B\Cobcl_{1}) \longrightarrow
        B(B\Cob_{1}) \longrightarrow
        B(B\Cobred_{1}).
    \]
    We can therefore write $B\Cobred_1 \simeq \gO(B(B\Cobred_1))$ 
    as the homotopy fiber of the left-hand map.
    We have a homotopy commutative diagram of infinite loop spaces as follows:
    \[
    \begin{tikzcd}
        B\Cobred_1 \ar[r] & B(B\Cobcl_1) \ar[r] 
        & B(B\Cob_1) \ar[d, "\simeq"', "{\ref{thm:GMTW}}"] \\
        \gO^{\infty-2}MTSO_2 \ar[r] 
        & Q(\gS^2(BSO_2)_+) \ar[u, "\simeq"', "{\ref{lem:BCobcl}}"] \ar[r] \ar[ru, "\ga_{S^1}"] 
        & \gO^{\infty-2} MTSO_1
    \end{tikzcd}
    \]
    Here lemma \ref{lem:BCobcl} is our computation of $B\Cobcl_1$ and
    theorem \ref{thm:GMTW} is the main theorem of \cite{GMTW06}.
    The bottom row is the Genauer fiber sequence from lemma \ref{lem:ga=fiber}.
    Since the top sequence is also a homotopy fiber sequence, 
    we obtain a zig-zag of equivalences of infinite loop spaces:
    \[
        B\Cobred_1 \leftarrow 
        \mi{hofib}\left(Q(\gS^2(BSO_2)_+) \xrightarrow{\ga_{S^1}} B(B \Cob_1) \right) 
        \rightarrow \gO^{\infty-2}MTSO_2.
    \]
\end{proof}

To obtain the desired computation of $Bh\Cob_{1}$ we first 
need to show that in dimension $1$
the reduced bordism category is equivalent to its homotopy category.
\begin{lem}\label{lem:red1=hred1}
    The natural functors 
    $\Cobred_{1} \leftarrow \gd(\Cobred_{1}) \to h(\Cobred_{1})$ 
    induce equivalences on classifying spaces.
\end{lem}
\begin{proof}
    For the left-ward pointing functor this is a consequence of 
    lemma \ref{lem:gd-doesnt-change-B} and the fibrancy of $\Cobred_{d,\gt}$,
    which in turn is a consequence of \cite[Proposition 3.2.4(ii)]{ERW19}.
    
    The right-ward pointing functor
    $
        \gd(\Cobred_{1}) \longrightarrow h\Cobred_{1}
    $
    is the identity on the (discrete) space of objects. 
    So to show that it is an equivalence, we will only have to 
    show that for any two objects $M, N \in \Cobred_{1}$ the projection 
    \[
        \hom_{\Cobred_{1}}(M, N) \longrightarrow 
        \pi_0 \hom_{\Cobred_{1}}(M, N) = \hom_{h\Cobred_{1}}(M,N)
    \]
    is a weak equivalence. To see this, recall from fact \ref{fact:hom-in-Cob}
    that the left-hand-side can be written as 
    \[
        \hom_{\Cobred_{1}}(M, N) 
        \simeq \coprod_{[W]} B\Diff^+(W \text{ rel } \partial W)
    \]
    where $[W]$ runs over diffeomorphism classes of reduced bordisms from $M$ to 
    $N$. By the classification of $1$-manifolds, every such reduced bordism
    is the disjoint union of intervals: $W \cong \amalg^k [0,1]$.
    A diffeomorphism of $W$ relative to its boundary cannot permute the 
    intervals and therefore the diffeomorphism group decomposes as a product. 
    Since the diffeomorphism group of the interval relative to its boundary 
    is contractible we have that
    \[
        B\Diff^+(W \text{ rel } \partial W) 
        \cong \prod_{i=1}^k B\Diff^+([0,1] \text{ rel } \{0,1\}) \simeq *.
    \]
    Therefore the connected components of the hom spaces of $\gd\Cobred_{1}$
    are contractible, and the category is equivalent to its homotopy category.
\end{proof}

\begin{thm}\label{thm:BhCob1} 
    There is a homotopy fiber sequence of infinite loop spaces
    \[
        S^1 \longrightarrow B(h\Cob_{1}) \longrightarrow \gO^{\infty-2} MTSO_2.
    \]
    The infinite loop space map $\gO^{\infty-2} MTSO_2 \to K(\IZ,2)$ 
    that continues this fiber sequence 
    corresponds to the generator $\gS^2\gk_0 \in H^2(\gS^2 MTSO_2) \cong \IZ$.
\end{thm}

\begin{proof}
    Consider the two compatible reduction fiber sequences of 
    Theorem \ref{thm:reduction-fiber-sequence}:
    \[
         \begin{tikzcd}
            B \Cobcl_{1} \ar[r] \ar[d] 
            & B \Cob_{1} \ar[r, "R"] \ar[d] 
            & B \Cobred_{1} \ar[d] \\
            B (h\Cobcl_{1}) \ar[r] 
            & B (h\Cob_{1}) \ar[r, "R"] 
            & B (h\Cobred_{1})
         \end{tikzcd}
    \]
    By lemma \ref{lem:red1=hred1} the right vertical map 
    is an equivalence and hence theorem \ref{thm:BCobred1} implies
    \[
        B(h\Cobred_{1}) \simeq B\Cobred_{1} 
        \simeq \gO^{\infty-2} MTSO_2.
    \]
    The category $h\Cobcl_{1}$ has one object,
    the endomorphisms of which are the natural numbers. 
    Therefore its classifying space is 
    $B(h\Cobcl_{1}) = B\IN \simeq S^1$. 
    Therefore the bottom fiber sequence of the diagram
    now reads as
    \[
        S^1 \longrightarrow B(h\Cob_{1}) \longrightarrow \gO^{\infty-2} MTSO_2,
    \]
    which proves the first claim of the theorem.
    
    Continuing the fiber sequences once to the right we have 
    \[
         \begin{tikzcd}
            B \Cob_{1} \ar[r, "R"] \ar[d] 
            & B \Cobred_{1} \ar[d, "p_1"] \ar[r, "f"]
            & Q(\gS^2 (BSO_2)_+) \ar[d, "p_2"] \\
            B (h\Cob_{1}) \ar[r, "R"] 
            & B (h\Cobred_{1}) \ar[r, "g"] 
            & K(\IZ, 2).
         \end{tikzcd}
    \]
    The map $f$ is an isomorphism for spectrum 
    cohomology in positive degree, because its fiber $QS^0$ has
    spectrum cohomology concentrated in degree $0$.
    The class $\gk_0 \in H^0(MTSO_2)$ is defined as the pullback
    of the basepoint class in $H^0(\gS^\infty (BSO_2)_+)$.
    The map $p_2$ is $3$-connected because it is the delooping
    of a $2$-connected map, and as noted above $p_1$ is an equivalence.
    Therefore $g$ pulls back the canonical class
    in $H^2(\gS^2 H\IZ)$ to $\gk_0$.
\end{proof}

As a consequence of the computation of 
$B(h\Cob_{1})$ we can compute its rational
cohomology using the following standard fact:

\begin{fact}\label{fact:H-of-gOinfty}
    Let $Y$ be a spectrum such that the spectrum cohomology 
    $H^i(Y; \IQ)$ is finite dimensional for all $i\ge 0$. 
    Then the rational cohomology 
    of its infinite loop space is 
    \[
        H^*(\gO^\infty Y; \IQ)
        \cong \prod_{\pi_0 Y} S[ H^{*>0}(Y; \IQ)] 
    \]
    where $S[V]$ denotes the free symmetric algebra on a graded vector space $V$.
    The Hopf-algebra structure on the cohomology of the identity component
    $H^*(\gO^\infty_0 Y; \IQ) \cong S[ H^{*>0}(Y; \IQ)]$ 
    is such that $H^{*>0}(Y;\IQ)$ consists of primitive elements.
\end{fact}
\begin{proof}
    This is well-known, but we provide a short sketch of proof.
    First, we have equivalences of spaces 
    $\gO^\infty Y \simeq \pi_0 Y \times \gO^\infty_0 Y
    \simeq \coprod_{\pi_0 Y} \gO^\infty_0 Y$
    and so it will suffice to show that
    $H^*(\gO^\infty_0 Y; \IQ) \cong S[H^{*>0}(Y;\IQ)]$.
    Because it has degree-wise finite dimensional homology
    $\tau_{\ge 1} Y$ is rationally equivalent to a direct sum 
    of Eilenberg-Mac Lane spectra $\gS^n H\IQ$, $n>0$.
    We may hence assume $Y=\gS^n H\IQ$, in which case we have 
    $\gO^{\infty-n}H\IQ = K(\IQ;n)$ and $H^*(K(\IQ;n);\IQ) = \IQ[\gb]$
    where $|\gb|=n$.
\end{proof}

\begin{cor}\label{cor:H(hCob1)}
    The rational cohomology rings of $h\Cob_1$ and $h\Cobred_1$ are 
    \[
        H^*(B(h\Cob_1)_0; \IQ) \cong 
        \IQ[\dlgk_1, \dlgk_2, \dots] 
        \qand
        H^*(B(h\Cobred_1)_0; \IQ) \cong 
        \IQ[\dlgk_0, \dlgk_1, \dlgk_2, \dots] 
    \]
    where $|\dlgk_i| = 2i+2$.
    Moreover, the $\dlgk_i$ are primitive with respect to 
    the Hopf-algebra structure.
\end{cor}
\begin{proof}
    The spectrum cohomology of $\gS^2 MTSO_2$ is 
    \[
        H^*(\gS^2 MTSO_2; \IQ) \cong \IQ\gle{\dlgk_{-1}, \dlgk_0, \dlgk_1, \dlgk_2, \dots}
    \]
    where $|\dlgk_i| = 2i+2$. 
    By fact \ref{fact:H-of-gOinfty} we have
    \[
        H^*(\gO^\infty_0 \gS^2 MTSO_2; \IQ) 
        \cong S[\IQ\gle{\dlgk_0, \dlgk_1, \dlgk_2, \dots}].
    \]
    This implies the second claim seeing as 
    $B(h\Cobred_1) \simeq \gO^{\infty-2} MTSO_2$ by 
    theorem \ref{thm:BCobred1}.

    Theorem \ref{thm:BhCob1} states that $B(h\Cob_1)$ is a circle bundle
    over $B\Cobred_1$, so we can compute its cohomology using the Gysin sequence.
    The Euler class of the circle bundle is $\dlgk_0 \in H^2(B(h\Cobred_1))$,
    so by the rational Gysin sequence
    \[
        H^*(B(h\Cob_1)_0; \IQ) \cong H^*(B(h\Cobred_1)_0; \IQ)/\gle{\dlgk_0}
        \cong \IQ[\dlgk_1, \dlgk_2, \dots]. 
    \]
    Alternatively we could have computed 
    $
        H^*(\mrm{hofib}(\gS^2 MTSO_2 \to \gS^2 H\IZ); \IQ)
        \cong \IQ\gle{\ga, \dlgk_1, \dlgk_2, \dots} 
    $
    and applied fact \ref{fact:H-of-gOinfty} again.
\end{proof}

\subsection{The reduced $1$-dimensional bordism category 
and topological cyclic homology of a simply connected spaces}

We will now compute the homotopy type of $\Cobred_{1,\gt}$ 
for more general tangential structures in terms of the so-called 
\emph{circle transfer} map.
For any space $Y$ with $\IT$-action the \emph{circle transfer} 
is an infinite loop space map 
\[
    \mrm{trf}_{\IT}: Q(\gS (Y_{h\IT})_+) \longrightarrow Q(Y_+)
\]
natural with respect to $\IT$-equivariant maps.

We will treat the circle transfer as a black box and 
refer the reader to \cite{Gia17} for a more detailed discussion
and pointers to the literature.
The circle transfer for the free loop space $Y = LX$ turns up naturally 
as a map between classifying spaces of $1$-dimensional bordism categories:
\begin{thm}[{\cite[Proposition 6.1]{Gia17}}]\label{thm:incl=trf}
    Let $\gt = X \times \gt^{\mrm{or}}$ be the tangential structure as before,
    then the following diagram of infinite loop spaces commutes up to homotopy:
    \footnote{
        Note that there is a small misprint in \cite[Proposition 6.1]{Gia17}: 
        The diagram should say $\op{ev} \circ \op{trf}$, not $\op{trf} \circ \op{ev}$,
        as is evident from the proof.
    }
    \[
        \begin{tikzcd}
            B\Cobcl_{1,\gt} \ar[rr] && B\Cob_{1,\gt} \ar[d, "\simeq"]\\
            Q(\gS_+ (L X)_{h\IT}) \ar[r, "\mrm{trf}_\IT"] \ar[u, "\simeq"] 
            & Q(LX_+) \ar[r, "Q(\op{ev})"] 
            & Q(X_+). 
        \end{tikzcd}
    \]
\end{thm}

\begin{cor}\label{cor:Cobredgt}
    For the tangential structure $\gt = X \times \gt^{or}$ there is 
    a homotopy fiber sequence of infinite loop spaces:
    \[
        Q(\gS (LX)_{h\IT}) 
        \xrightarrow{\ Q(\op{ev}) \circ \mrm{trf}_\IT \ } Q(X)
        \longrightarrow B\Cobred_{1,\gt} 
    \]
\end{cor}
\begin{proof}
    Theorem \ref{thm:reduction-fiber-sequence} gives us the reduction fiber
    sequence 
    \[
        B\Cobcl_{1,\gt} \longrightarrow B\Cob_{1,\gt} 
        \longrightarrow B\Cobred_{1,\gt} .
    \]
    By corollary \ref{cor:Bcobcl1} 
    $
        B\Cobcl_{1,\gt} 
        \simeq Q(\gS_+ (LX)_{h\IT})
    $
    and the main theorem of \cite{GMTW06} implies $B\Cob_{1,\gt} \simeq Q(X_+)$.
    Inserting these into the reduction fiber sequence we obtain a homotopy
    fiber sequence with the desired terms, and theorem \ref{thm:incl=trf}
    identifies the relevant map.
\end{proof}

\begin{cor}\label{cor:TC=Cobred}
    For any simply connected space $X$ there is an equivalence
    \[
        \mi{TC}(\IS[\gO X]; p) \simeq Q(X_+)_p^\wedge 
        \times \left( \gO B\Cobred_{1,X \times \gt^\mrm{or}} \right)_p^\wedge.
    \]
    The left-hand-side denotes the topological cyclic homology of the 
    ring spectrum $\IS[\gO X] := \gS^\infty (\gO X)_+$.
\end{cor}
\begin{proof}
    According to \cite[Proposition 3.9]{BCCGHM96} there is a splitting
    \[
        \mi{TC}(\IS[\gO X]; p) \simeq 
        Q(X_+)_p^\wedge \times \op{hofib}\left(Q(\gS (LX)_{h\IT}) 
        \xrightarrow{\ \mi{ev} \circ \mrm{trf}_\IT \ } Q(X)
        \right)_p^\wedge.
    \]
    Using corollary \ref{cor:Cobredgt} we can rewrite the second term as
    \[
        \gO \op{hofib}\left(Q(\gS^2 (LX)_{h\IT}) \xrightarrow{\ \mi{ev} \circ \gO^{-1}\mrm{trf}_\IT \ } Q(\gS X) \right)
        \simeq \gO B \Cobred_{d,X \times \gt^\mrm{or}}
    \]
    and the claim follows.
\end{proof}

\subsection{The unoriented bordism category}
In this section we consider the trivial tangential structure 
$\mrm{unor} := \{*\}$. We have results similar to the 
oriented case: 
\begin{thm}\label{thm:BhCob1-unor} 
    There is an equivalence of infinite loop spaces
    $ B\Cobred_{1, \mrm{unor}} \simeq \gO^{\infty - 2} MTO_2 $ 
    and there is a homotopy fiber sequence of infinite loop spaces
    \[
        S^1 \longrightarrow B(h\Cob_{1, \mrm{unor}}) 
        \longrightarrow \gO^{\infty-2} MTO_2.
    \]
    The rational cohomology ring of $h\Cob_{1, \mrm{unor}}$ is:
    \[
        H^*(B(h\Cob_{1,\mrm{unor}})_0; \IQ) 
        \cong \IQ[\dlgk_2, \dlgk_4, \dots] 
    \]
    where $|\dlgk_i| = 2i+2$, and $\pi_0 B(h\Cob_{1,\mrm{unor}}) \cong \IZ/2$.
\end{thm}
\begin{proof}
    The first part is proved just like theorem \ref{thm:BCobred1}
    and theorem \ref{thm:BhCob1}. We only need to observe that 
    the equivalence of lemma \ref{lem:ga=fiber}
    still holds in the unoriented case.
    
    To compute the rational cohomology, we need to 
    understand the map $f:\gO^{\infty-2}MTO_2 \to K(\IZ,2)$
    that continues the fiber sequence. 
    Since $\pi_1 B(h\Cob_{1,\mrm{unor}}) = 0$ the map
    is necessarily surjective on $\pi_2$.
    So, by the same arguments as in corollary \ref{cor:H(hCob1)}
    we compute the cohomology. Here we use that
    \[
        H^*(MTO_2;\IQ) = \IQ\gle{\gk_0, \gk_2, \gk_4, \dots }.
    \]
    To determine $\pi_0 B(h\Cob_{1,\mrm{unor}})$ recall that the unoriented $0$-dimensional 
    bordism group is $\IZ/2$.
\end{proof}

\part{Cocycles on the cobordism category}

In the first part of the paper we computed the homotopy types of $B(h\Cob_1)$
and $B(h\Cobred_1)$ and as a result showed that their rational cohomology rings
are polynomial algebras on the generators $\dlgk_i \in H^{2i+2}(B(h\Cob_1); \IQ)$.
This computation, however, was achieved abstractly 
and is perhaps unsatisfying in that it does not give us 
a concrete understanding of what the classes $\dlgk_i$ actually are.
The purpose of this second part is to gain a more concrete understanding
of $B(h\Cob_1)$ and $B(h\Cob_1^{\mrm{red}})$, and in particular 
the $\dlgk$-classes on them.
All the $\dlgk_i$ can be obtained as a pullback along
the connecting homomorphism of the reduction fiber sequence:
\[
    B(h\Cobred_1) \xleftarrow{\simeq} B\Cobred_1 
    \xrightarrow{\ f\ } Q(\gS^2 (BSO_2)_+).
\]
In section \ref{sec:2-cocycle} give a hands-on construction of 
the $2$-cocycle representing
the class $\dlgk_0 \in H^2(B(h\Cobred_1))$ coming from $H^2(f)$.
In subsection \ref{subsec:continued-fib} we construct a simplicial space
of ``cuts'' $\Cut$ 
and show that the canonical quotient map $N\Cobred_1 \to \Cut$ 
is a geometric model for $f$. 
This is then used in the final section to give the cocycle formulas.

\section[The 2-cocycle on the reduced bordism category]%
{The $2$-cocycle on the reduced bordism category}
\label{sec:2-cocycle}
We can think of $h\Cob_{d,\gt}$ as a central extension of $h\Cobred_{d,\gt}$
by the abelian monoid $h\Cobcl_{d,\gt}$.
In this section we construct the $2$-cocycle $\ga$ on $h\Cobred_{d,\gt}$
corresponding to the central extension. 
This is the first step in understanding the cocycles representing the 
cohomology classes on $h\Cob_{1}$.

\begin{defn}
    Let $\mcA_{d,\gt}$ denote the abelian group $\pi_1 B\Cobcl_{d,\gt}$.
\end{defn}
    
\begin{rem}
    As a consequence of lemma \ref{lem:BCobcl} 
    $\mcA_{d,\gt}$ is 
    the free abelian group on the set 
    of diffeomorphism classes of closed connected $\gt$-structured $d$-manifolds.
    Recall that $\pi_0(\Phi_{d,\gt})$ is the commutative monoid 
    of diffeomorphism classes of closed $\gt$-structured $d$-manifolds 
    under disjoint union. This monoid injects into $\mcA_{d,\gt}$,
    and $\mcA_{d,\gt}$ is the group completion $\pi_0(\Phi_{d,\gt})$.
    When $d=1$ and $\gt=\{\pm1\}$ is the tangential structure for orientation, then 
    $\mcA_{1,\mrm{or}} \cong \IZ$ with generator $[S^1]$. Similarly,
    $
        \mcA_{2,\mrm{or}} \cong \bigoplus_{g \ge 0} \IZ
    $
    where there is one summand for every genus $g \ge 0$.
\end{rem}

\begin{defn}\label{defn:ga-2-cochain}
    We define a $2$-chain $\ga: N_2(h\Cobred_{d,\gt}) \to \mcA_{d,\gt}$ by
    \[
        \ga: (M_0 \xrightarrow{W_0} M_1 \xrightarrow{W_1}  M_2)
        \mapsto
        [c(W_0 \cup_{M_1} W_1)] \in \pi_0\Phi_{d,\gt} \subset \mcA_{d,\gt}.
    \]
\end{defn}

\begin{figure}[h]
    \centering
    \def\svgwidth{0.6\linewidth}
%% Accompanies image file 'kappa0.pdf' (pdf, eps, ps)
\begingroup%
  \makeatletter%
  \providecommand\color[2][]{%
    \errmessage{(Inkscape) Color is used for the text in Inkscape, but the package 'color.sty' is not loaded}%
    \renewcommand\color[2][]{}%
  }%
  \providecommand\transparent[1]{%
    \errmessage{(Inkscape) Transparency is used (non-zero) for the text in Inkscape, but the package 'transparent.sty' is not loaded}%
    \renewcommand\transparent[1]{}%
  }%
  \providecommand\rotatebox[2]{#2}%
  \newcommand*\fsize{\dimexpr\f@size pt\relax}%
  \newcommand*\lineheight[1]{\fontsize{\fsize}{#1\fsize}\selectfont}%
  \ifx\svgwidth\undefined%
    \setlength{\unitlength}{541.18215486bp}%
    \ifx\svgscale\undefined%
      \relax%
    \else%
      \setlength{\unitlength}{\unitlength * \real{\svgscale}}%
    \fi%
  \else%
    \setlength{\unitlength}{\svgwidth}%
  \fi%
  \global\let\svgwidth\undefined%
  \global\let\svgscale\undefined%
  \makeatother%
  \begin{picture}(1,0.27799059)%
    \lineheight{1}%
    \setlength\tabcolsep{0pt}%
    \put(0,0){\includegraphics[width=\unitlength,page=1]{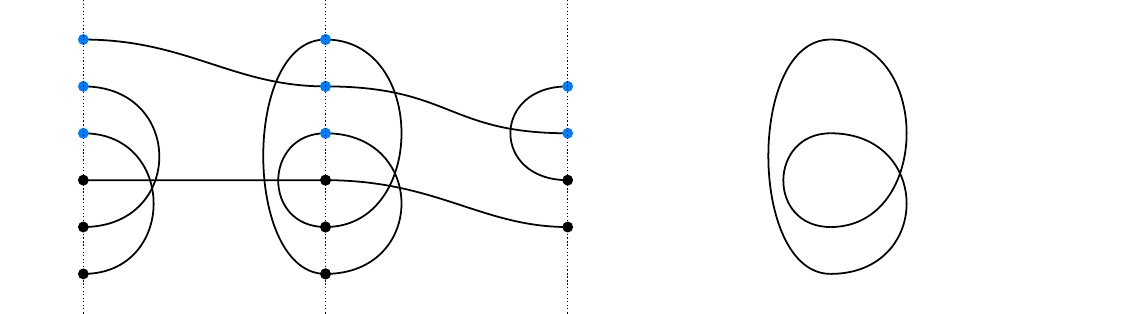}}%
    \put(0.58932336,0.13094491){\color[rgb]{0,0,0}\makebox(0,0)[lt]{\lineheight{1.25}\smash{\begin{tabular}[t]{l}$\mapsto$\end{tabular}}}}%
    \put(-0.00083909,0.13094491){\color[rgb]{0,0,0}\makebox(0,0)[lt]{\lineheight{1.25}\smash{\begin{tabular}[t]{l}$\alpha:$\end{tabular}}}}%
    \put(0.83211667,0.13094491){\color[rgb]{0,0,0}\makebox(0,0)[lt]{\lineheight{1.25}\smash{\begin{tabular}[t]{l}$\hat{=}\ 1 \in \mathbb{Z}$\end{tabular}}}}%
  \end{picture}%
\endgroup%

    \caption{The $2$-cochain $\ga$ evaluated on a $2$-simplex
    in $\Cut_{d,\gt}$ for $d=1$ and $\gt=\{\pm1\}$.}
    \label{fig:kappa0}
\end{figure}

Letting $d=1$ and $\gt=\{\pm1\}$ the $2$-chain is $\IZ$-valued 
and assigns to two reduced one-bordisms $W:M \to N$ and $V:N \to L$ 
the number of circles in the glued bordism $W \cup_N V$.
See figure \ref{fig:kappa0} for an illustration of this case.
This measures the failure of the canonical section 
$h\Cobred_1 \to h\Cob_1$ to be functorial. In this sense it
is reminiscent of the group $2$-cocycle one usually assigns 
to a central extension of groups. 
We make this precise:

\begin{lem}\label{lem:ga-is-kappa0}
    The cochain $\ga$ is a $2$-cocycle and the cohomology class 
    $[-\ga] \in H^2(B(h\Cobred_{d,\gt}); \mcA_{d,\gt})$ corresponds to
    the map 
    \[
        a: B(h\Cobred_{d,\gt}) \longrightarrow K(\mcA_{d,\gt}, 2)
    \]
    that continues the fiber sequence for homotopy categories of 
    theorem \ref{thm:reduction-fiber-sequence}.
    For $d=1$ and $\gt = \mrm{or}$ we have $[-\ga] = \dlgk_0$.
\end{lem}
\begin{proof}
    We first verify that $\ga$ is indeed a cocycle. Consider a $3$-simplex in 
    the nerve of $h\Cobred_{d,\gt}$: 
    $x= (M_0 \xrightarrow{W_1} M_1 \xrightarrow{W_2} M_2 \xrightarrow{W_3} M_3)$.
    Evaluating $\ga$ on $\partial x = \sum_i (-1)^i d_i x$ gives
    \[
        [c(W_2 \cup_{M_2} W_3)] 
        - [c(r(W_1\cup_{M_1} W_2) \cup_{M_2} W_3)] 
        + [c(W_1 \cup_{M_1} r(W_2 \cup_{M_2} W_3))] 
        - [c(W_1 \cup_{M_1} W_2)] .
    \]
    To understand this, first consider $c(W_1 \cup_{M_1} W_2 \cup_{M_2} W_3)$.
    This is the submanifold of $W_1 \cup_{M_1} W_2 \cup_{M_2} W_3$ given 
    by those components that do not intersect $M_0$ or $M_3$.
    We can further decompose it as
    \[
        c(W_1 \cup_{M_1} W_2 \cup_{M_2} W_3) 
        = 
        c(r(W_1\cup_{M_1} W_2) \cup_{M_2} W_3) \amalg c(W_1 \cup_{M_1} W_2)
    \]
    into those components that do, or do not, intersect $M_2$.
    Similarly we have a decomposition
    \[
        c(W_1 \cup_{M_1} W_2 \cup_{M_2} W_3) 
        = 
        c(W_1\cup_{M_1} r(W_2 \cup_{M_2} W_3)) \amalg c(W_2 \cup_{M_2} W_3).
    \]
    This shows that the two terms in $\ga(\partial x)$ with a minus sign 
    cancel the two terms with a plus sign.
    Hence $\ga$ indeed is a cocycle.

    For the second claim consider the map of homotopy fiber sequences 
    \[
        \begin{tikzcd}
            B(h\Cobcl_{d,\gt}) \ar[r] \ar[d, "\simeq"] & B(h\Cob_{d,\gt}) 
        \ar[r] \ar[d] & B(h\Cobred_{d,\gt}) \ar[d, "a"] \\
            K(\mcA_{d,\gt},1) \ar[r] & * \ar[r] & K(\mcA_{d,\gt},2).
        \end{tikzcd}
    \]
    where the top row is the reduction fiber sequence
    from Theorem \ref{thm:reduction-fiber-sequence}.
    The category $h\Cobcl_{d,\gt}$ has one object and the morphisms 
    form the monoid of diffeomorphism classes of closed manifolds
    under disjoint union. There is a canonical $1$-cocycle
    $\gb \in H^1(h\Cobcl_{d,\gt}; \mcA_{d,\gt})$, which sends 
    a morphism $[W: \emptyset \to \emptyset]$ to $[W] \in \mcA_{d,\gt}$.
    This cocycle corresponds to the left vertical map in the above diagram.
    In the Serre spectral sequence for the bottom fiber sequence the canonical 
    element $\gb \in H^1(K(\mcA_{d,\gt}, 1); \mcA_{d,\gt})$ transgresses to the canonical
    element $\gb' \in H^2(K(\mcA_{d,\gt}, 2); \mcA_{d,\gt})$. 
    By naturality of the Serre spectral sequence this implies 
    that $d_2 [\gb] = a^* d_2[\gb] = a^* \gb'$.
    We will prove the lemma by showing that the $d_2$-differential 
    in the cohomological Serre spectral sequence
    \[
        d_2: H^1(h\Cobcl_{d,\gt}; \mcA_{d,\gt}) 
        \longrightarrow H^2(h\Cobred_{d,\gt}; \mcA_{d,\gt})
    \]
    sends $[\gb]$ to $-[\ga]$.
    
    Recall 
    that the transgression 
    $d_2[\gb]$ of $[\gb]$ can be uniquely characterised by requiring that in
    the following diagram we have $R^*(d_2[\gb]) = \gd([\gb])$.
    \[
        \begin{tikzcd}
            & H^2(h\Cobred_{d,\gt}, [\emptyset]; \mcA_{d,\gt}) \ar[r, "\cong"] \ar[d, "R^*"]
            & H^2(h\Cobred_{d,\gt}; \mcA_{d,\gt}) \\
            H^1(h\Cobcl_{d,\gt}; \mcA_{d,\gt}) \ar[r, "\gd"] 
            & H^2(h\Cob_{d,\gt}, h\Cobcl_{d,\gt}; \mcA_{d,\gt}) &
        \end{tikzcd}
    \]
    The coboundary operator $\gd$ is defined on $\gb$ by choosing any
    extension $\wh{\gb}: N_1(h\Cob_{d,\gt}) \to \mcA_{d,\gt}$ 
    and setting $\gd(\gb) = \wh{\gb} \circ \partial$.
    It will be convenient to choose $\wh{\gb}(W) := [c(W)]$.
    The $2$-cocycle $\gd(\gb):N_2(h\Cob_{d,\gt}) \to \mcA_{d,\gt}$ is then
    \[
       \gd(\gb)(M_0 \xrightarrow{W_1} M_1 \xrightarrow{W_2} M_2) 
       = \wh{\gb} ( \partial (M_0 \xrightarrow{W_1} M_1 \xrightarrow{W_2} M_2))
       = [c(W_2)] - [c(W_1\cup_{M_1}W_2)] + [c(W_1)].
    \]
    We can now compute $\gd(\gb) + R^*\ga$
    on some $2$-simplex $x=(M_0 \xrightarrow{W_1} M_1 \xrightarrow{W_2} M_2)$:
    \begin{align*}
        (\gd(\gb) + R^*\ga)(x) &= \wh{\gb}(\partial x) + \ga(R(x)) \\
        &= \gb([W_2] - [W_1\cup_M W_2] + [W_1]) 
        + \ga(M_0 \xrightarrow{r(W_1)} M_1 \xrightarrow{r(W_2)} M_2) \\
        &= [c(W_2)] - [c(W_1\cup_M W_2)] + [c(W_1)]
        + [c(r(W_1) \cup_{M_1} r(W_2))].
    \end{align*}
    To see that this vanishes observe that the closed components of 
    $(W_1 \cup_{M_1} W_2)$ can be decomposed as
    \[
        c(W_1 \cup_{M_1} W_2) = c(r(W_1) \cup_{M_1} r(W_2)) 
        \amalg  c(W_1) \amalg c(W_2).
    \]
    This shows that $\gd(\gb) = -R^*\ga$ and hence $d_2[\gb] = -\ga$.
    As we argued previously, this implies the second claim.
\end{proof}

\section{The continued reduction sequence and the space of cuts}
\label{sec:continued-and-cuts}

\subsection{The continued fiber sequence}
\label{subsec:continued-fib}
In the case of homotopy categories the reduction fiber sequence can be thought
of as a central extension that is classified by the $2$-cocycle $\ga$ 
constructed in the previous section. Similarly, the topologically enriched
version of the reduction fiber sequence leads to a ``central extension of 
infinite loop spaces'' classified an infinite loop space map 
$B\Cobred_{d,\gt} \to Q(\gS^2 (\coprod_{[W]\text{ con.}} B\Diff^\gt(W))_+)$.
In the remainder of this section we will explicitly describe this map,
so that we can later use it to understand the higher cocycles on $h\Cob_1$.

It will be convenient to consider the following subspace,
which in some sense freely generates $\Psi_{d,\gt}$:
\begin{defn}\label{defn:Psi-con}
    Let $\Psi_{d,\gt}^{con} \subset \Psi_{d,\gt}$ be the subspace
    of those $(W,l)$ where $W \subset \cube$ is connected or empty.
    We let $W = \emptyset$ be the base-point of this space.
    By fact \ref{fact:Psi=BDiffs} there is an equivalence:
    \[
        \Psi_{d,\gt}^{con} \simeq
        \{\emptyset\} \amalg 
        \coprod_{[W] \text{ connected}} B\Diff^\gt(W).
    \] 
\end{defn}

To construct the map $B\Cobred_{d,\gt} \to Q(\gS^2 \Psi_{d,\gt}^{con})$ 
we will need a new semi-simplicial space:
\begin{defn}
    For any $n \ge 0$ we define a subspace
    \[
        (\Cut_{d,\gt})_n \subset N_n \Cob_{d,\gt} 
        = \{ (M_0 \xrightarrow{W_1} M_1 \to \dots \to M_{n-1} \xrightarrow{W_n} M_n) \text{ in } \Cob_{d,\gt}\}
    \]
    to contain those $n$-tuples where $M_0 = \emptyset$, $M_n = \emptyset$,
    and all the $W_i$ are reduced.
    
    There is a retraction $S_n: N_n \Cob_{d,\gt} \to (\Cut_{d,\gt})_n$ for this inclusion,
    defined by deleting all connected components of 
    $(W_1 \circ \dots \circ W_n)$ that intersect $M_0 \amalg M_n$,
    or lie in $c(W_i)$ for some $i$.
    We define face 
    operators for $\Cut_{d,\gt}$ by 
    $d_i^{\Cut} := S_n \circ d_i^{N\Cob}$.
\end{defn}

\begin{rem}
    One needs to check that the $d_i$ in fact satisfy 
    all the simplicial relations. This follows immediately once one checks
    $S_n \circ d_i^{N\Cob} = S_n \circ d_i^{N\Cob} \circ S_n$
    and $S_n \circ s_i^{N\Cob} = s_i^{N\Cob} \circ S_n$.
    See lemma \ref{lem:C'-simplicial} for more details.
    Conceptually, $d_i^\Cut$ acts by taking the face map 
    as usual in the nerve and then deleting all connected components $V$
    that violate the condition $V \cap M_0 \amalg M_n = \emptyset$
    or the condition $V \not\subset \coprod_i c(W_i)$.
    See figure \ref{fig:Cut-faces} for an illustration.
    
    Note that it follows from the definition that 
    $S_\cd:N\Cob_{d,\gt} \to \Cut_{d,\gt}$ is a map of semi-simplicial spaces
    and that it identifies $\Cut_{d,\gt}$ as a quotient semi-simplicial space 
    of $N\Cob_{d,\gt}$.
\end{rem}

\begin{figure}[ht]
    \centering
    \def\svgwidth{.9\linewidth}
%% Accompanies image file 'Cut.pdf' (pdf, eps, ps)
\begingroup%
  \makeatletter%
  \providecommand\color[2][]{%
    \errmessage{(Inkscape) Color is used for the text in Inkscape, but the package 'color.sty' is not loaded}%
    \renewcommand\color[2][]{}%
  }%
  \providecommand\transparent[1]{%
    \errmessage{(Inkscape) Transparency is used (non-zero) for the text in Inkscape, but the package 'transparent.sty' is not loaded}%
    \renewcommand\transparent[1]{}%
  }%
  \providecommand\rotatebox[2]{#2}%
  \newcommand*\fsize{\dimexpr\f@size pt\relax}%
  \newcommand*\lineheight[1]{\fontsize{\fsize}{#1\fsize}\selectfont}%
  \ifx\svgwidth\undefined%
    \setlength{\unitlength}{819.37324103bp}%
    \ifx\svgscale\undefined%
      \relax%
    \else%
      \setlength{\unitlength}{\unitlength * \real{\svgscale}}%
    \fi%
  \else%
    \setlength{\unitlength}{\svgwidth}%
  \fi%
  \global\let\svgwidth\undefined%
  \global\let\svgscale\undefined%
  \makeatother%
  \begin{picture}(1,0.40670914)%
    \lineheight{1}%
    \setlength\tabcolsep{0pt}%
    \put(0,0){\includegraphics[width=\unitlength,page=1]{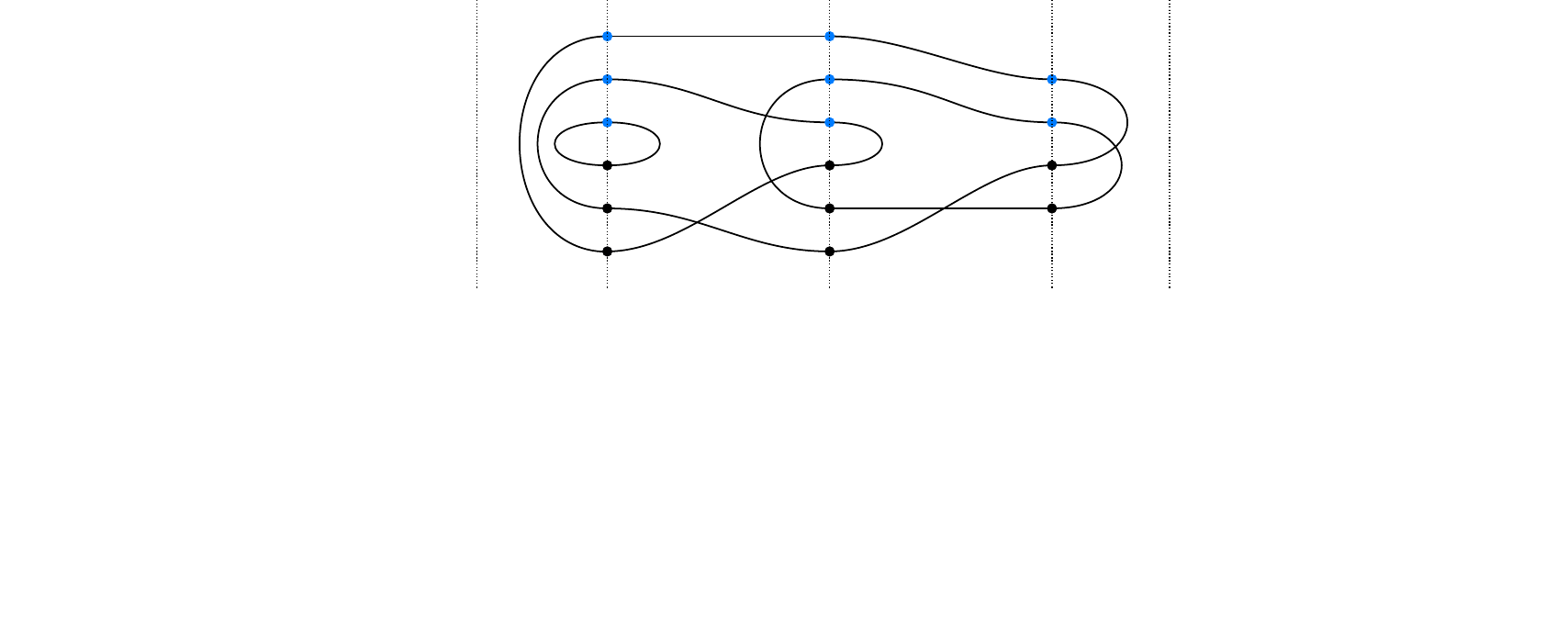}}%
    \put(0.23947931,0.30851099){\color[rgb]{0,0,0}\makebox(0,0)[lt]{\lineheight{1.25}\smash{\begin{tabular}[t]{l}$x\ =$\end{tabular}}}}%
    \put(0,0){\includegraphics[width=\unitlength,page=2]{Cut.pdf}}%
    \put(-0.02126546,0.08818384){\color[rgb]{0,0,0}\makebox(0,0)[lt]{\lineheight{1.25}\smash{\begin{tabular}[t]{l}$d_0x\ =$\end{tabular}}}}%
    \put(0,0){\includegraphics[width=\unitlength,page=3]{Cut.pdf}}%
    \put(0.46878676,0.08818384){\color[rgb]{0,0,0}\makebox(0,0)[lt]{\lineheight{1.25}\smash{\begin{tabular}[t]{l}$d_1x\ =$\end{tabular}}}}%
  \end{picture}%
\endgroup%

    \caption{
        A $4$-simplex in $\Cut_{d,\gt}$ for $d=1$ and $\gt=\{\pm1\}$
        and two of its simplicial faces.
    }
    \label{fig:Cut-faces}
\end{figure}

\begin{thm}\label{thm:continue-fib-seq}
    The natural quotient map $N \Cobred_{d,\gt} \to \Cut_{d,\gt}$ induces,
    after geometric realisation, a continuation of the reduction fiber sequence 
    from Theorem \ref{thm:reduction-fiber-sequence}.
    More precisely, there is a map of homotopy fiber sequences of 
    infinite loop spaces:
    \[
        \begin{tikzcd}
            B \Cobcl_{d,\gt} \ar[r] \ar[d, equal] 
            & B \Cob_{d,\gt} \ar[d] \ar[r, "R"] 
            & B \Cobred_{d,\gt} \ar[d, "S"] \\
            B \Cobcl_{d,\gt} \ar[r] 
            & * \ar[r] 
            & \| \Cut_{d,\gt} \|.
        \end{tikzcd}
    \]
    In particular, there is an equivalence of infinite loop spaces:
    \[
        \| \Cut_{d,\gt} \| \simeq B(B\Cobcl_{d,\gt}) 
        \simeq Q(\gS^2 \Psi_{d,\gt}^{con}).
    \]
\end{thm}

In dimension $d=1$ with $\gt = \{\pm 1\}$ we see that 
$\|\Cut_{1,\mrm{or}}\| \simeq Q(\gS^2 (\CP^\infty)_+)$
and the map $S$ comes from the cofiber sequence of spectra
\[
    \IS \longrightarrow \gS^2\CP^\infty_{-1} 
    \xrightarrow{\ s\ } \gS^{\infty+2}(\CP^\infty)_+.
\]
In particular, it is a rational equivalence on connective covers.
This will be particularly useful
in section \ref{sec:identifying-cocycles} where we want 
to write the $\dlgk_i$-classes as pullbacks along $S$.
We record:
\begin{cor}\label{cor:rational-equiv}
    For $d = 1$ and $\gt = \{\pm1\}$ the map $S$ induces a rational equivalence
    \[
        B\Cobred_1 \xrightarrow{(\pi_0,S)} \IZ \times \|\Cut_1\|.
    \]
\end{cor}
If we introduce a background space $X$, then it follows from 
lemma \ref{lem:BCobcl} that $\Psi_{1,\{\pm1\}\times X}^{\mrm{con}} \simeq ((LX)_{h\IT})_+$.
Let us denote $\Cut_{1,\{\pm1\}\times X}$ by $\Cut_1(X)$.
By theorem \ref{thm:continue-fib-seq} we have 
$\|\Cut_1(X)\| \simeq Q(\gS^2 (LX_{h\IT})_+)$.
Recall that $\gS^{\infty}_+ LX$ is also the topological Hochschild homology 
(THH) of the spherical group ring $\IS[\gO X]$.
Indeed, we will see in section \ref{subsec:fact-vs-Connes} that $\Cut_1$ 
is closely related to Connes' category $\gL$ and in the presence of a background space
$\Cut_{1}(X)$ should be thought of as a variant 
of the cyclic bar construction.
This identification of $\Cut$ with THH is compatible with the relation
between $\Cobred_1$ and TC:
\begin{cor}\label{cor:map-to-THH}
    There is a homotopy commutative diagram of $p$-complete infinite loop spaces:
    \[
    \begin{tikzcd}
        \gO^\infty \TC(\IS[\gO X]; p) \ar[d,"b"] \ar[r, "\ga"] 
        & Q(\gS_+ (LX)_{h\IT})_p^\wedge \ar[d, "\simeq"] \\
        \gO B\Cobred_1(X)_p^\wedge \ar[r, "\gO S"] 
        & \gO \|\Cut_1(X)\|_p^\wedge
    \end{tikzcd}
    \]
    where map $\ga$ is the top map in \cite[Diagram (0.3)]{BHM93},
    see also \cite[Theorem IV.3.6]{NS18}, and the map $b$ is obtained
    by composing the equivalence from corollary \ref{cor:TC=Cobred}
    with the projection to the second factor.
\end{cor}
\begin{proof}
    As a consequence of the identifications made in corollary \ref{cor:Cobredgt}
    and the continued fiber sequence in theorem \ref{thm:continue-fib-seq} 
    we have a homotopy commutative square:
    \[
    \begin{tikzcd}
        \op{hofib}\left(Q(\gS_+ (LX)_{h\IT}) 
        \xrightarrow{\ \mi{ev} \circ \mrm{trf}_\IT \ } Q(X) \right)
        \ar[d, "\simeq"] \ar[r] 
        & Q\left(\gS_+ (LX)_{h\IT}\right)\ar[d, "\simeq"] \\
        \gO B\Cobred_1(X) \ar[r, "\gO S"] 
        & \gO \|\Cut_1(X)\| .
    \end{tikzcd}
    \]
    Now we can, just like in corollary \ref{cor:TC=Cobred},
    use \cite[Proposition 3.9]{BCCGHM96} to obtain an equivalence 
    \[
        \mi{TC}(\IS[\gO X]; p) \simeq 
        Q(X_+)_p^\wedge \times \op{hofib}\left(Q(\gS (LX)_{h\IT}) 
        \xrightarrow{\ \mi{ev} \circ \mrm{trf}_\IT \ } Q(X)
        \right)_p^\wedge.
    \]
    The square in the claim is now obtained by $p$-completing the above square
    and composing with the projection of $\mi{TC}$ onto the second factor.
\end{proof}

The remainder of this section will be concerned with proving 
theorem \ref{thm:continue-fib-seq}.

\subsection{The category of factorizations}
Before we can prove theorem \ref{thm:continue-fib-seq} we need 
a good understanding of how the space $\Psi_{d,\gt}^{con}$ relates
to $\|\Cut_{d,\gt}\|$ and $B(\Cobcl_{d,\gt})$.
This will make use of the following poset:
\begin{defn}\label{defn:F}
    The topological poset $F_{d,\gt}$ consists of tuples $((W,l), t)$
    where $(W,l) \in \Psi_{d,\gt}^{con}$ is a closed $d$-dimensional submanifold
    of $\cube$ with $W$ empty or connected 
    and $t \in (-1,1)$ is a regular value of $\pr_1:W \to (-1,1)$ 
    with $t \in \pr_\IR(W)$ or $W$ empty.
    The relation is such that $((W,l),t) \le ((W',l'),s)$ iff 
    $(W,l)=(W',l')$ and $t \le s$.
\end{defn}
One should think of $F_{d,\gt}$ as the category of possible 
factorizations of connected manifolds in $\Cob_{d,\gt}$.
An object is a tuple of bordisms $W_0:\emptyset \to M$
and $W_1:M \to \emptyset$ such that $W_0 \cup_M W_1$ is connected
and a morphism $(M,W_0,W_1) \to (N,V_0,V_1)$ is a bordism 
$X:M \to N$ such that $V_0 = W_0 \cup_M X$ and $W_1 = X\cup_N V_1$.
In this sense $F_{d,\gt}$ is a full subcategory of the 
``over-and-under-category''; we will explore this perspective in 
section \ref{subsec:fact-vs-Connes}.

There is an augmentation map $F_{d,\gt} \to \Psi_{d,\gt}^{con}$ defined
by forgetting the regular value $t$.
\begin{proposition}\label{prop:BF}
    The augmentation induces an equivalence
    $
        B(F_{d,\gt}) \simeq 
        \Psi_{d,\gt}^{con} 
    $.
\end{proposition}

We will prove this using the following technical lemma about
classifying space of certain posets.
\begin{lem}[\cite{GRW14}]\label{lem:B-of-poset}
    Let $X$ be a space and $P \subset X \times \IR$ an open subset such that 
    $\pr_{X}: P \to X$ is surjective. We equip $P$ with the 
    poset structure defined by 
    \[
        (x,t) \le (x', s) \Leftrightarrow (x = x' \text{ and } t \le s).
    \]
    Then the canonical augmentation $N_kP \to X$ induces a weak equivalence
    \[
        B(P, \le) \simeq X.
    \]
\end{lem}
\begin{proof}
    This is easily seen to be a special case of \cite[Theorem 6.2]{GRW14}.
    % 
    % To apply the theorem we check:
    % \begin{itemize}
    %     \item[(i)] \emph{The map $\eps:A_0 = P \to X = A_{-1}$ has local sections. }
    %     This follows from $P \subset X \times \IR$ being open and the definition 
    %     of the product topology.
    %     \item[(ii)] \emph{The map $\eps:A_0 = P \to X = A_{-1}$ is surjective.}
    %     This is one of the conditions of the lemma.
    %     \item[(iii)] \emph{For any $p \in A_{-1}$ and 
    %     $\{v_1, \dots, v_n\} \subset \eps^{-1}(p)$ there is $v \in A_0$
    %     with $(v_i,v) \in A_1$ for all $i$.}
    %     We can write all the $v_i$ as $(p, t_i)$ for some 
    %     $t_i \in \pr_\IR(\pr_X^{-1}(p))$. 
    %     The set $\pr_\IR(\pr_X^{-1}(p)) = \pr_\IR(P \cap \{p\} \times \IR)$
    %     is an open subset of $\IR$, and so for any finite subset 
    %     $\{t_1, \dots, t_n\}$ we can find $t$ in there such that $t>t_i$.
    %     Then $v= (p,t)$ is the desired point of $A_0$.
    % \end{itemize}
    % Hence \cite[Theorem 6.2]{GRW14} applies and the augmentation map induces
    % a weak equivalence $B(P, <) = \|A\| \simeq A_{-1} = X$.
\end{proof}

\begin{proof}[Proof of proposition \ref{prop:BF}]
    We want to apply the proceeding lemma \ref{lem:B-of-poset} to 
    the poset $F_{d,\gt}$. The subset 
    $F_{d,\gt} \subset \Psi_{d,\gt}^{con} \times \IR$ is open because
    for $t$ to be a regular value for $\pr_1:W \to \IR$ is an open 
    condition that continuously depends on $W$.
    The problem is that we also require $t \in \pr_1(W)$; 
    this leads to the projection map $F_{d,\gt} \to \Psi_{d,\gt}^{con}$
    not being surjective.
    Lemma \ref{lem:B-of-poset} therefore only shows that $BF_{d,\gt}$
    is equivalent to the subspace $X \subset \Psi_{d,\gt}^{con}$ 
    of those manifolds $(W,l) \in \Psi_{d,\gt}^{con}$ such that 
    there is a $t \in \pr_1(W)$ that is a regular value for 
    $\pr_1:W \to \IR$.
    
    Since $W$ is connected and the regular values of $\pr_1$ 
    are dense in $\IR$, the only way for $(W,l)$ not to lie in $X$
    is if $\pr_1(W) = \{t\}$ is a single point.
    To prove the proposition we have to show that the inclusion 
    $X \subset \Psi_{d,\gt}^{con}$ is a weak equivalence.
    We can write this inclusion as:
    {
    \scriptsize
    \[
        X \cong 
        \coprod_{[W] \text{ con.}} \left(\Emb(W, \cube)' \times \Bun^\gt(W)\right)/\Diff(W)
        \inj
        \coprod_{[W] \text{ con.}} \left(\Emb(W, \cube) \times \Bun^\gt(W)\right)/\Diff(W)
        \cong \Psi_{d,\gt}^{con}
    \]
    }%
    Here $\Emb(W, \cube)'$ is the subspace of those embeddings 
    $\gi:W \inj \cube$ where the first coordinate $\gi_0:W \to (-1,1)$
    is not a constant map. 
    Since each $\Emb(W,\cube) \times \Bun^\gt(W)$ is a $\Diff(W)$-principal
    bundle over the respective connected component of $\Psi_{d,\gt}$ 
    it will suffice to check that each of the inclusions 
    \[
        \Emb(W, \cube)' \times \Bun^\gt(W)
        \inj
        \Emb(W, \cube) \times \Bun^\gt(W)
    \]
    is a weak equivalence. 
    We will do so by showing that $\Emb(W, \cube)'$ is contractible. 
    
    Let $\mcC^\infty(W, (-1,1))$ be the space of smooth functions 
    with the Whitney $\mcC^\infty$-topology and write $\mcC^\infty(W, (-1,1))'$
    for the subspace of non-constant smooth functions.
    Evaluating the first coordinate yields a map 
    \[
        E : \Emb(W, \cube)' \longrightarrow \mcC^\infty(W, (-1,1))'.
    \]
    Pick a preferred embedding $j:W \inj \cube$ and define an right-inverse $J$
    to $E$ by setting $J(f)(w) = (f(w), j_1(w), j_2(w),\dots)$.
    It is clear that $E \circ J$ is the identity on $\mcC^\infty(W, (-1,1))'$
    and $J \circ E$ is homotopic to the identity by the same standard argument
    that shows that $\Emb(W,\cube)$ is homotopy equivalent to 
    $\mcC^\infty(W, (-1,1))'$.
    
    We still need to check that $\mcC^\infty(W, (-1,1))'$ weakly contractible.
    Since $(-1,1)$ is diffeomorphic to $\IR$, the space $\mcC^\infty(W, (-1,1))$ 
    is homeomorphic to the vector space of smooth functions $\mcC^\infty(W, \IR)$.
    $\mcC^\infty(W,\IR)$ is a separable infinite dimensional Fr\'echet space
    and subspace of constant functions $\IR \subset \mcC^\infty(W,\IR)$
    is the union of countably many compacts.
    It therefore follows from \cite[Theorem VI.5.2, Corollary V.6.2, and Theorem V.6.3]{BP75}
    that $\mcC^\infty(W,(-1,1))' \cong \mcC^\infty(W, \IR) \setminus \IR$
    is homeomorphic to $\mcC^\infty(W,\IR)$ and in particular contractible.
\end{proof}

\subsection{Poset models}
\label{subsec:poset-models}

When working with bordism categories it is often convenient 
to replace the topological category $\Cob_d$ by a topological 
poset $P\Cob_d$ with an equivalent classifying space. 
We learnt this trick from the work of Galatius--Randal-Williams 
(see e.g.\ \cite{GRW10}).
We recall $P\Cob_d$ and introduce four simplicial spaces,
related to $\Cobcl$, $\Cobred$, $\Cut$, and a new simplicial space $D$. 
Note that these will in fact be simplicial and not just semi-simplicial 
spaces since the poset $P\Cob_d$ (unlike $\Cobcl_d, \Cob_d, \Cobred_d$) 
does have identity morphisms.
In what follows we will often suppress the data $(d,\gt)$ from the notation
to make space for simplicial indices.

\begin{defn}
    The topological poset $P\Cob_{d,\gt}$ has underlying space the space
    of tuples $((W,l),t)$ where $(W,l) \in \Psi_{d,\gt}(\IR \times \cube)$ 
    is a submanifold of the ``tube'' $\IR \times \cube$ and 
    $t \in \IR$ is a regular value of the projection $W \to \IR$.
    The poset structure is defined via
    \[
        ((W,l), t) \le ((W',l'), s) \Leftrightarrow
        (W,l) = (W',l') \text{ and } t \le s.
    \]
\end{defn}

\begin{rem}\label{rem:Pcyl}
To compare this with $\Cob_{d,\gt}$ one defines another 
(non-unital) topological poset $P_{\mrm{cyl}}\Cob_{d,\gt}$
where for every object $((W,l),t)$ there is an $\eps>0$ 
such that $(W,l)$ is cylindrical over $(t-\eps,t+\eps)$.
Then there are functors
    \[
        \Cob_{d,\gt} \leftarrow P_{\mrm{cyl}}\Cob_{d,\gt} \to P\Cob_{d,\gt}
    \]
and by \cite[Theorem 3.9]{GRW10} they both induce equivalences
on the classifying spaces. All of these constructions are
compatible with the $\gC$-space structures we described in
section \ref{subsec:infinite-loop}.
\end{rem}

Let $C$ denote the nerve of the poset $P\Cob_{d,\gt}$.
An $n$-simplex in $C$ is a tuple $((W,l), t_0 \le \dots \le t_n)$ where 
$(W,l) \in \Psi_{d,\gt}(\IR\times \cube)$ is a $d$-dimensional 
$\gt$-structured manifold in the tube $\IR \times \cube$ 
and the $t_i \in \IR$ are regular values of $\pr_\IR: W \to \IR$.

\begin{defn}\label{defn:C'}
    For all $n$ we define subspaces
    $C^{\mrm{cl}}_n, C^{\mrm{red}}_n, \Cut_n', D_n \subset C_n$
    to contain those $n$-simplices $w = ((W,l), t_0\le \dots \le t_n) \in C_n$ 
    satisfying certain conditions:
    \begin{align*}
        w \in C^{\mrm{cl}}_n & \quad \Leftrightarrow \quad
        \forall V \subset W: \pr_\IR(V) \cap \{t_0, \dots, t_n\} = \emptyset, \\
        w \in C^{\mrm{red}}_n  & \quad \Leftrightarrow \quad
        \forall V \subset W: \pr_\IR(V) \cap \{t_0, \dots, t_n\} \neq \emptyset, \\
        w \in \Cut_n'  & \quad \Leftrightarrow \quad
        \forall V \subset W: \pr_\IR(V) \subset (t_0,t_n) \text{ and } 
        \pr_\IR(V) \cap \{t_1, \dots, t_{n-1}\} \neq \emptyset, \\
        w \in D_n & \quad \Leftrightarrow \quad
        \forall V \subset W: \pr_\IR(V) \subset (t_0,t_n).
    \end{align*}
    Here the $\forall$ quantifier runs over all connected components
    $V \subset W$. When $n=0$ or, more generally, when $t_0=t_n$,
    the third and the fourth condition are interpreted as $W=\emptyset$.
    
    Each of the subspaces $C^{\mrm{cl}}_n, C^{\mrm{red}}_n, \Cut_n', D_n \subset C_n$
    is a union of connected components. 
    The inclusion $D_n \subset C_n$ admits a canonical retraction $r:C_n \to D_n$
    defined by sending $w = ((W,l), t_0\le \dots \le t_n)$ to 
    $((W',l_{|W'}), t_0\le \dots \le t_n)$ where $W' \subset W$ is the union
    of those connected components $V \subset W$ satisfying $\pr_\IR \subset (t_0,t_n)$. 
    Similarly, define  $r:C_n \to C^{\mrm{cl}}_n$, $r:C_n \to C^{\mrm{red}}_n$, 
    and $r:C_n \to \Cut_n'$ by deleting the connected components
    that violate the relevant condition.
    
    Define face operators $d_i$ on 
    $C^{\mrm{cl}}_\cd$, $C^{\mrm{red}}_\cd$, $\Cut_\cd'$, and $D_\cd$ 
    by including to $C_\cd$, applying the face operator of $C_\cd$,
    and then applying the retraction.
    Degeneracy operators are defined similarly.
\end{defn}

See figure \ref{fig:Cut-faces} for an illustration of how face maps
work in $\Cut_\cd'$ and figure \ref{fig:D} for an example of a 
$4$-simplex in $D_\cd$.
\begin{figure}[h]
    \centering
    \includegraphics[width=0.5\textwidth]{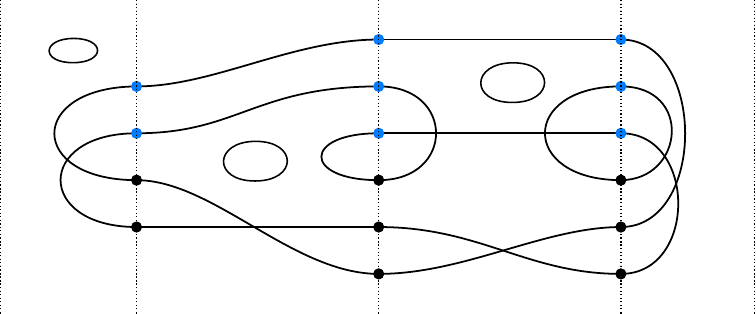}
    \caption{
        A $4$-simplex in $D$ for $d=1$ and $\gt=\{\pm1\}$.
    }
    \label{fig:D}
\end{figure}

Note that while the inclusion $C_\cd^{\mrm{red}} \subset C_\cd$ 
is not a simplicial map, the retraction $C_\cd \to C_\cd^{\mrm{red}}$ is:
\begin{lem}\label{lem:C'-simplicial}
    With the above definition the inclusion makes 
    $C^{\mrm{cl}}_\cd \subset C_\cd$ a simplicial subspace
    and the retractions make $C^{\mrm{red}}_\cd$, $\Cut_\cd'$,
    and $D_\cd$ into simplicial quotient spaces of $C_\cd$.
    In each of the four cases the $i$th degeneracy operator
    simply repeats $t_i$.
\end{lem}
\begin{proof}
    Note that for any retraction $(i:A \to B, r:B \to A)$ of topological
    spaces the space $A$ carries the subspace topology with respect 
    to $i$ and the quotient topology with respect to $r$.
    Therefore each of $C^{\mrm{red}}_n$, $\Cut_n$, and $D_n$ is indeed
    a quotient space of $C_n$.
    
    To check the first claim, we need to show that 
    $C^{\mrm{cl}}_\cd \subset C_\cd$ is closed under face maps.
    But this is clear because $\pr_\IR(V) \cap \{t_0, \dots, t_n\} = \emptyset$
    implies $\pr_\IR(V) \cap \{t_0, \dots \wh{t_i} \dots, t_n\} = \emptyset$
    for any $i$.
    
    For the second claim we need to check that $r \circ d_i = r \circ d_i \circ r$ 
    holds for $r:C_\cd \to C^{\mrm{red}}_\cd$ the retraction 
    and $d_i$ a face operator on $C_\cd$.
    Once this is shown, it follows that the induced face operators 
    $d_i^{\mrm{red}} := r \circ d_i$ satisfy the simplicial identities
    and that $r: C_\cd \to C^{\mrm{red}}_\cd$ is a simplicial quotient map.
    The cases of $\Cut_\cd'$ and $D_\cd$ are then shown similarly.
    
    By definition the retraction deletes
    all connected components $V \subset W$ such that 
    $V \cap \{t_0, \dots, t_n\} = \emptyset$.
    We need to make sure that if a component is deleted by $r$,
    then it is also deleted by $r \circ d_i$.
    But if $V \cap \{t_0, \dots, t_n\} = \emptyset$,
    then $V \cap \{t_0, \dots, \wh{t_i}, \dots, t_n\} = \emptyset$,
    so the claim follows.
    As similar argument works in the other two cases
    because for each of them the conditions on $V \subset W$
    get stricter as the set $\{t_0,\dots,t_n\}$ gets smaller.
\end{proof}

\begin{cor}\label{cor:poset=cat}
    There are zig-zags of level-wise weak equivalences
    of semi-simplicial spaces:
    \[
        N(\Cobcl_{d,\gt}) \simeq C^{\mrm{cl}}, \quad
        N(\Cob_{d,\gt}) \simeq C, \quad
        N(\Cobred_{d,\gt}) \simeq C^{\mrm{red}}, \qand
        \Cut_{d,\gt} \simeq \Cut'.
    \]
    These are compatible with the canonical maps 
    $C^{\mrm{cl}} \inj C \surj C^{\mrm{red}} \surj \Cut'$.
\end{cor}
\begin{proof}
    This is proved exactly as in \cite[Theorem 3.9]{GRW10}.
    Recall from remark \ref{rem:Pcyl} that we can define a (non-unital)
    topological poset $P_{\mrm{cyl}}\Cob$ that fits in a zig-zag
    \[
        \Cob \leftarrow P_{\mrm{cyl}}\Cob \to P\Cob.
    \]
    In \cite[Theorem 3.9]{GRW10} the authors show that both functors
    induce level-wise weak equivalences on the nerves.
    This is exactly the second case.
    The other three cases follow from this case because level-wise
    each of $C^{\mrm{cl}}$, $C^{\mrm{red}}$, and $\Cut_n$
    is a union of connected components of $C_n$
    and the weak equivalence restricts to a weak equivalence between
    these connected components.
\end{proof}

\subsection{Proof of theorem \ref{thm:continue-fib-seq}}
\label{subsec:proof-of-continued}
Each of the simplicial spaces $C_\cd'$ has a $\gC$-structure 
defined using the tangential structure just like we did for $\Cob_{d,\gt}$
in section \ref{subsec:infinite-loop}.
We will continue to suppress $d$ and $\gt$ from the notation
and we will write $\cd$ as a placeholder for simplicial indices.
By construction there is a commutative diagram of simplicial 
$\gC$-spaces
    \[
        \begin{tikzcd}
            C^{\mrm{cl}}_\cd \ar[r] \ar[d, equal] 
            & C_\cd \ar[d] \ar[r] 
            & C^{\mrm{red}}_\cd \ar[d] \\
            C^{\mrm{cl}}_\cd \ar[r] 
            & D_\cd \ar[r] 
            & \Cut_\cd'
        \end{tikzcd}
    \]
where the composition of the two horizontal arrows in each row
is canonically null-homotopic.
By corollary \ref{cor:poset=cat} this diagram yields, after geometric
realisation, a diagram of infinite loop spaces where the top row
is the reduction fiber sequence of theorem \ref{thm:reduction-fiber-sequence}.
To prove theorem \ref{thm:continue-fib-seq} it will therefore suffice to 
show that the bottom row is a homotopy fiber sequence after 
geometric realisation and that $\|D_\cd\|$ is contractible.

\bigskip
We begin by simplifying the problem from understanding a homotopy
fiber sequence of infinite loop spaces to understanding a 
homotopy cofiber sequence of spaces. To do so, we show that each
of the three simplicial $\gC$-spaces $\Ccl_\cd$, $D_\cd$,
and $\Cut_\cd'$ is freely generated by its subspace of connected
manifolds, which we now define.
\begin{defn}
    For $X=\Ccl_\cd$, $D_\cd$, or $\Cut_\cd'$, let 
    $X_\cd^{con} \subset X_\cd$ denote the sub-simplicial space 
    of those tuples $((W,l),t_0 \le \dots \le t_n)$ where 
    $W$ is either connected or empty.
\end{defn}

Observe that $X_\cd^{con}$ defines a subsimplicial space
because face operators can only delete, but not create connected
components. However, $X_\cd^{con}$ is not a special 
$\gC$-space anymore. 
This makes sense as the $\gC$-space structure is supposed to capture
the operation defined by disjoint union of manifolds.
The subspace $X_\cd^{con} \subset X_\cd$ should be thought of
as the space of ``indecomposables'' for this operation and
we will see that $X_\cd$ is in fact freely generated by them.
This is a generalisation of lemma \ref{lem:BCobcl}.

\begin{lem}\label{lem:free-on-con}
    For $X$ any of the three spaces as above the inclusion 
    $\|X_\cd^{con}\| \to \|X_\cd\|$ induces an equivalence 
    of infinite loop spaces
    $
        Q(\|X_\cd^{con}\|) \simeq \|X_\cd\|
    $.
    These equivalences are compatible in the sense that 
    there is a homotopy commutative diagram of infinite loop spaces
    \[
    \begin{tikzcd}
        Q(\|(\Ccl_\cd)^{con}\|) \ar[r] \ar[d, "\simeq"]
        & Q(\|D^{con}\|) \ar[r] \ar[d, "\simeq"]
        & Q(\|(\Cut_\cd')^{con}\|) \ar[d, "\simeq"] \\
        \|\Ccl_\cd\| \ar[r]
        & \|D\| \ar[r]
        & \|\Cut_\cd'\|.
    \end{tikzcd}
    \]
\end{lem}
\begin{proof}
    For $Y_\cd$ a pointed simplicial space, let $\Conf_*(\cube; Y_\cd)$ 
    be the simplicial space whose $n$-th level is the space of unordered configurations in 
    $\cube$ with labels in $Y_n$, modulo the equivalence relation that deletes
    configuration points labelled by the basepoint. 
    This is a simplicial $\gC$-space with the $\gC$-structure defined 
    as usual for configuration spaces.
    There is a canonical simplicial inclusion $Y_\cd \to \Conf_*(\cube; Y_\cd)$ 
    defined by sending $y \in Y_n$ to the configuration $\{0\} \subset \cube$
    labelled by $y$.
    
    We construct for each $X$ a zig-zag of equivalences of simplicial $\gC$-spaces:
    \[
        X_\cd \longleftarrow Z_\cd \longrightarrow \Conf_*(\cube, X_\cd^{con}) .
    \]
    Then the technical lemma \ref{lem:free-gamma} implies that 
    $Q(\|X_\cd^{con}\|)\simeq \|X_\cd\|$ seeing as $\|X_\cd^{con}\|$
    is always connected.
    If we construct the zig-zags compatibly with the maps 
    $C^{\mrm{cl}}_\cd \to D_\cd \to \Cut_\cd'$, then it also follows
    that the diagram of infinite loop spaces in the lemma is homotopy commutative.
   
    In the case of $X = C^{\mrm{cl}}$ the relevant zig-zag 
    was already constructed in lemma \ref{lem:BCobcl}, 
    though under a different name.
    We now recall the construction to see that it works in all three cases.
    A point in $Z_n$ is represented by a tuple $(W, i, j, l, \ul{t})$ 
    where $W$ is a $d$-manifold with $l \in \Bun^\gt(W)$ a $\gt$-structure,
    $\ul{t} = (t_0 \le \dots \le t_n)$, 
    $i:W \inj \IR \times \cube$ is an embedding and
    $j:\pi_0 W \inj \cube$ is a configuration.
    Moreover, we require that $(i(W),i_*l, t_0 \le \dots \le t_n)$ 
    is a well-defined point in $X_n$. 
    (In particular, $i$ has to be cylindrical in the appropriate places.)
    $Z_n$ is the space of equivalence classes, where we identify
    $(W, i, j, l, \ul{t}) \sim (W', i \circ \gp, j \circ (\pi_0 \gp), \gp^*(l), \ul{t})$
    for any diffeomorphism $\gp:W' \cong W$.
    The $\gC$-space is, as always, defined via the tangential structure $\gt$.
    
    The projection maps in the zig-zag are defined by
    \[
        X_n \ni (i(W), i_*, t_0 \le \dots \le t_n) 
        \longmapsfrom [W, i, j, l, \ul{t}] \longmapsto 
        (j(\pi_0 W), q) \in \Conf_*(\cube; X_n^{con})
    \]
    where the labelling $q: j(\pi_0 W) \to X_n^{con}$ labels 
    the configuration point $j([V])$ by $(i(V), i_*(l_{|V}), t_0 \le \dots \le t_n)$
    for any connected component $V \subset W$.
    
    It is not hard to see that $Z$ is a simplicial $\gC$-space and 
    that the two maps are compatible with this structure. 
    The same arguments as in lemma \ref{lem:free-gamma}
    now show that they both are equivalences.
\end{proof}

\begin{lem}\label{lem:free-gamma}
    For $Y_\cd$ a pointed simplicial space, let $\Conf_*(\cube; Y_\cd)$ 
    be the simplicial $\gC$-space from the proof of lemma \ref{lem:free-on-con}.
    If $\|Y\|$ is connected,
    then there is an equivalence of infinite loop spaces
    \[
        Q(\|Y\|) \xrightarrow{\ \simeq\ } \|\Conf_*(\cube; Y_\cd)\|
    \]
    compatible with the inclusion of $\|Y\|$ on both sides.
\end{lem}
\begin{proof}
    Consider the operad $\mcO$ on the category of spaces where 
    the space of $n$-ary operations is the set of $n$-tuples
    $(i_1, \dots, i_n)$ of embeddings $i_j: \cube \to \cube$ with disjoint images,
    such that there are $r_j \in (0,1)$ and $a_j \in \cube$
    with $i_j(x) = r_j(x + a_j)$ for all $x \in \cube$.
    This is the variant of the \emph{little $\infty$-cubes} operad,
    and as all of the spaces $\mcO(n)$ are contractible, it is an $E_\infty$-operad.
    The associated monad $T_\mcO$ on the category of based spaces
    has a canonical map to the labelled configuration space:
    \[
        T_\mcO(X) = \bigvee_{n \ge 1} (X^n \times \mcO(n))_{/\gS_n}
        \to \Conf_*(\cube; X), 
        \quad
        [(x_1,\dots,x_n), (i_1,\dots,i_n)] \mapsto (\amalg_j i_j(0), l)
    \]
    where the labelling $l$ is $l(i_j(0)) = x_j$.
    This map is clearly an equivalence for all based spaces $X$, 
    as it only forgets the information of the $r_j \in (0,1)$.
    
    By \cite[Theorem 12.2]{May72} any monad associated to an operad
    commutes with geometric realisation up to homeomorphism,
    and so we have an equivalence
    \[
        T_\mcO(\|Y_\cd\|) \cong \|T_\mcO(Y_\cd)\| 
        \xrightarrow{\simeq} \Conf_*(\cube; \|Y_\cd\|).
    \]
    But $\mcO$ is an $E_\infty$-operad and $\|Y_\cd\|$ is connected,
    so by \cite[Corollary 6.3]{May72} $T_\mcO(\|Y_\cd\|)$
    is equivalent to $Q(\|Y_\cd\|)$ and the claim follows.
\end{proof}

In light of the previous lemma we can prove theorem 
\ref{thm:continue-fib-seq} by showing that
\[
    \|(\Ccl_\cd)^{con}\| \longrightarrow \|D_\cd^{con}\|
    \longrightarrow \|(\Cut_\cd')^{con}\|
\]
is a homotopy cofiber sequence and that $\|D^{con}\|$ is contractible. 
To see that this is a homotopy cofiber sequence%
\footnote{
    For this to be a homotopy cofiber sequence we need to specify 
    a nullhomotopy of the composite. The simplicial map
    $(\Ccl_\cd)^{con} \to (\Cut_\cd')^{con}$ 
    sends $((W,l), t_0 \le \dots \le t_n)$ to 
    $((\emptyset, \emptyset), t_0 \le \dots \le t_n)$,
    so there is a canonical nullhomotopy given by contracting the $t_i$ to $0$.
}
is not difficult:
for each $n$ we actually have 
$D^{con}_n \simeq (\Ccl_n)^{con} \vee (\Cut_n')^{con}$.
The crucial step, however, is to check that $\|D^{con}\|$ is contractible,
indeed we will show the following identification of homotopy cofiber sequences:
\begin{lem}\label{lem:identify-cofib-seq}
    Let $F_{d,\gt} \to \Psi_{d,\gt}^{con}$ be the augmented topological poset
    from definition \ref{defn:F}. Then there is an equivalence of homotopy cofiber sequences:
    \[
    \begin{tikzcd}
        \|(\Ccl_\cd)^{con}\| \ar[r] \ar[d, "\simeq"]
        & \|D_\cd^{con}\| \ar[r] \ar[d, "\simeq"]
        & \|(\Cut_\cd')^{con}\| \ar[d, "\simeq"] \\
        \gS(\Psi_{d,\gt}^{con}) \ar[r]
        & \gS(C(BF_{d,\gt} \to \Psi_{d,\gt}^{con})) \ar[r]
        & \gS^2(BF_{d,\gt})
    \end{tikzcd}
    \]
    where $C(BF_{d,\gt} \to \Psi_{d,\gt}^{con})$ denotes the cone of 
    the augmentation map $BF_{d,\gt} \to \Psi_{d,\gt}^{con}$.
    In particular, since this map is a weak equivalence 
    by proposition \ref{prop:BF}, the space $\|D^{con}\|$
    is contractible.
\end{lem}

The rest of this section will be concerned with proving this lemma.

\begin{proof}[Part 1 of the proof of lemma \ref{lem:identify-cofib-seq}]
\label{proof:D-con}
    The crucial observation for the proof of the lemma is 
    that each of the simplicial spaces $(\Ccl_\cd)^{con}$,
    $D_\cd^{con}$, and $(\Cut_\cd')^{con}$ decomposes level-wise
    as a wedge of simpler spaces. For example in the first case we
    have a canonical equivalence:
    \[
        (\Ccl_\cd)^{con}_n 
        \simeq \bigvee_{k=1}^n \Psi_{d,\gt}^{con} \times \{k\}
    \]
    defined by sending an $n$-simplex $((W,l),t_0\le \dots \le t_n)$ 
    to the tuple $((W,l),k)$ where $k \in \{1,\dots,n\}$ is the unique
    number such that $\pr_1(W) \subset (t_{k-1},t_k)$.
    This makes sense since $W$ is assumed to be connected
    and $\pr_1(W)$ cannot contain any of the $t_i$.
    If $W = \emptyset$ we send $((\emptyset,\emptyset),t_0\le \dots \le t_n)$
    to the basepoint.
    This map is a fiber bundle with fiber 
    the convex subset of $\IR^{n+1}$ containing the possible choices of $t_i$ such that 
    $t_0 \le \dots \le t_{k-1} \le \min(\pr_1(W))$
    and $\max(\pr_1(W)) \le t_k \le \dots \le t_n$.
    
    There are similar decompositions for $D_n^{con}$ and 
    $(\Cut_n')^{con}$, with the additional complication 
    that we need to keep track of a tuple $(a \le b)$ such 
    that $\pr_1(W) \subset (t_{a-1}, t_{b})$.
    Indeed, using the nerve of the augmented topological poset%
    \footnote{
        An augmented topological poset is a topological poset $(P, \le)$
        together with a map $a:P \to X$ satisfying $a(x)= a(y)$
        for all $x,y \in P$ with $x \le y$.
        Its nerve is the augmented simplicial space $NP$ with 
        $N_nP = \{p \in P^{n+1}\;|\; p_0 \le \dots \le p_n\}$
        for $n\ge 0$ and $N_{-1}P = X$. The map $a$ is used as the 
        face operator $f_0:N_0P = P \to X = N_{-1}P$.
    }
    $F_{d,\gt} \to \Psi_{d,\gt}^{con}$ we can write 
    \[
        D^{con}_n \simeq \bigvee_{1\le a \le b \le n} N_{a-b-1}F_{d,\gt} 
        \qand
        (\Cut_{d,\gt}^{con})_n \simeq \bigvee_{1\le a < b \le n} N_{a-b-1}F_{d,\gt} 
    \]
    where $N_{-1}F_{d,\gt} = \Psi_{d,\gt}^{con}$.
    We complete this proof below.
\end{proof}

To complete the proof we need to properly understand how these wedges
of space fit together to form a simplicial space. 
For this we introduce the notion of a simplicial (relative) cone:

\begin{defn}\label{defn:simplicial-cone}
    Let $(X_\cd, X_{-1}, x_0)$ be a pointed augmented simplicial 
    space. Then the \emph{relative cone} and the 
    \emph{opposite relative cone} are the pointed augmented
    simplicial spaces $C(X,X_{-1})$ and $C^{op}(X,X_{-1})$ defined 
    as follows. In both case the space of $n$-simplices is
    \[
        C(X,X_{-1})_n := \bigvee_{k=-1}^n (X_k \times \{k\})
        =: C^{op}(X,X_{-1})_n 
    \]
    The face and degeneracy operators are defined for the cone as:
    \[
        d_i (x,k) = \begin{cases}
            (d_i x, k-1) & \text{ if } i \le k \\
            (x, k) & \text{ if } i > k 
        \end{cases}
        \qand
        s_i (x,k) = \begin{cases}
            (s_i x, k+1) & \text{ if } i \le k \\
            (x, k) & \text{ if } i > k 
        \end{cases}
    \]
    and for the opposite cone as:
    \[
        d_i^{op} (x,k) = \begin{cases}
            (d_{i-(n-k)} x, k-1) & \text{ if } i \ge n-k \\
            (x, k) & \text{ if } i < n-k 
        \end{cases}
        \qand
        s_i^{op} (x,k) = \begin{cases}
            (s_{i-(n-k)} x, k+1) & \text{ if } i \ge n-k \\
            (x, k) & \text{ if } i < n-k .
        \end{cases}
    \]
    These definitions are chosen such that 
    $C^{op}(X,X_{-1}) = (C(X^{op},X_{-1}))^{op}$.
    
    We define the reduced cone $C_{red}^{(op)}(X,X_{-1})$ as the quotient
    of $C^{(op)}(X,X_{-1})$ by $X$, included in the obvious way.
    Moreover, we write $\gS^{(op)}X := C_{red}^{(op)}(X,*)$ when
    $X$ has the trivial augmentation.
\end{defn}

\begin{lem}\label{lem:extra-deg}
    For any augmented pointed simplicial space $(X, X_{-1}, x_0)$ 
    there is a cofiber sequence of simplicial spaces
    \[
        \op{const}(X_{-1}) \longrightarrow C_{red}(X,X_{-1})
        \longrightarrow C_{red}(X,*)
    \]
    that realises to the cofiber sequence
    \[
        X_{-1} \longrightarrow \mrm{Cone}(\|X\| \to X_{-1})
        \longrightarrow \gS \|X\|.
    \]
\end{lem}
\begin{proof}
    We begin by showing that the augmented simplicial space
    $C^{op}(X, X_{-1})$ admits an extra degeneracy 
    $s_{-1}: C^{op}(X,X_{-1})_n \to C^{op}(X,X_{-1})_{n+1}$ by 
    $s_{-1}(x,k) = (x,k)$.
    This satisfies $d_0s_{-1}(x,k) = (x,k)$ because 
    $s_{-1}(x,k) = (x,k)\in C^{op}(X,X_{-1})_{n+1}$ falls under
    the second case of the definition of $d_0$. 
    (Since $0 < (n+1)-k$.)
    We also observe that $d_{i+1} s_{-1} = s_{-1} d_i$
    and $s_{j+1} s_{-1} = s_{-1} s_j$ hold by construction.
    Hence $s_{-1}$ is indeed an extra degeneracy and so the augmentation
    map $C^{op}(X,X_{-1}) \to X_{-1}$ induces a homotopy equivalence:
    $\|C^{op}(X,X_{-1})\| \simeq X_{-1}$.
    Taking opposites appropriately, we obtain the same result
    for $C(X,X_{-1})$.
    This also implies that $\|C(X,*)\|$ is contractible.
    
    Next, we observe that the canonical inclusion $X \to C(X,X_{-1})$ 
    defined by $(x\in X_n) \mapsto (x,n)$ is a level-wise cofibration
    and hence induces a cofibration $\|X\| \to \|C(X,X_{-1})\|$.
    Here we write $X$ for the non-augmented simplicial space.
    By the first part of the proof $\|C(X,X_{-1})\|$ is equivalent
    to $X_{-1}$ via the augmentation map and hence 
    $\|C(X,X_{-1})/X\|$ is cone for the augmentation map.
    The inclusion $\op{const}(X_{-1}) \to C(X,X_{-1})/X$ is 
    in each level the inclusion of a wedge summand and hence
    a cofibration. The quotient of this map is 
    $C(X,X_{-1})/(X \vee \op{const}(X_{-1})) = C(X,*)/X = \gS X$.
\end{proof}

\begin{proof}[Part 2 of the proof of lemma \ref{lem:identify-cofib-seq}]
    We will use $NF$ as a short-hand for the simplicial nerve
    $N(F_{d,\gt})$ and we let $\psi_\cd$ denote the constant simplicial 
    space $\psi_n :=\Psi_{d,\gt}^{con}$.
    We will construct simplicial maps 
    \[
    \begin{tikzcd}
        (\Ccl_\cd)^{con} \ar[r] \ar[d, "f"]
        & D_\cd^{con} \ar[r] \ar[d, "g"]
        & (\Cut_\cd')^{con} \ar[d, "h"] \\
        \gS(\psi_\cd) \ar[r]
        & \gS(C_{red}(NF, \psi)) \ar[r]
        & \gS(C_{red}(NF, *))
    \end{tikzcd}
    \]
    such that each of $f$, $g$, and $h$ is a level-wise equivalence.
    The first map 
    $f:(N\Cobcl_{d,\gt})^{con} \to \gS(\psi_\cd)$ 
    is defined by 
    \[
        f:((W,l), t_0\le \dots \le t_n)  \mapsto ((W,l), k) 
        \quad
        \in (\psi \times \{k\}) \subset (\gS\psi_\cd)_n
    \]
    where $k \in \{0,\dots, n-1\}$ is the unique number such 
    that $\pr_1(W) \subset (t_k, t_{k+1})$.
    
    Before we define the second map $g:D^{con} \to \gS(C^{op}(NF,\psi)/NF)$,
    recall that an $n$-simplex in $\gS(C_{red}^{op}(NF,\psi))$ can be written
    as $(x,l,k)$ where $k\in \{0,\dots,n-1\}$, $l\in \{-1,\dots,k-1\}$,
    and $x \in NF_l$.
    We can hence define the map as
    \[
        g: ((W,l),t_0\le \dots \le t_n) 
        \mapsto \left(((W,l), t_{k-l} \le \dots \le t_k), l, k\right)
        \quad
        \in (NF_l \times \{l\} \times \{k\}) 
    \]
    where $(k,l)$ are the smallest numbers such that
    $\pr_1(W) \subset (t_{k-l-1}, t_{k+1})$.
    The face maps for $\gS(C_{red}^{op}(NF,\psi))$ can be derived
    from the formulas in definition \ref{defn:simplicial-cone} as
    \[
        d_i (x,l,k) = \begin{cases}
            (x,l, k) 
            &\text{ for }  i > k\\
            (x,l, k-1) 
            &\text{ for }  i \le k \text{ and } i \le k-l\\
            (d_{i-(k-l)} x,l-1, k-1) 
            &\text{ for }  k-l < i \le k.
        \end{cases}
    \]
    Here we implicitly identify an $n$-simplex $(x,l,k)$ with 
    the base-point $*$ if $l=k$ or $k=n$.
    Using these formulas it it not hard to check that $g$ is indeed 
    simplicial.
    The map $h$ is defined by the same formula as $g$, which makes
    sense because $(\Cut_\cd')^{con}$ is a quotient of $D_\cd^{con}$.
    In other words $h$ is induced by the fact that both rows in 
    the diagram are level-wise cofiber sequences.
    
    This concludes the construction of the diagram.
    The maps $f$, $g$, and $h$ are level-wise homotopy equivalences
    by the observations in the first part of the proof.
\end{proof}

\section{Identifying cocycles}
\label{sec:identifying-cocycles}

By our main theorem the rational cohomology of $h\Cob_1$ is a polynomial
algebra on a generator $\ga$ in degree $0$ and generators $\dlgk_i$ in 
degree $(2i+2)$. We wish to give combinatorial formulas for these cocycles 
representing $\dlgk_i$.

Our strategy is as follows:
The computations of section \ref{sec:Computations-d=1} imply that all the $\dlgk_i$
are pulled back from the space of cuts $\|\Cut_1\|$ and in the previous section
we gave an equivalence $\|\Cut_1\| \simeq Q(\gS^2 BF_1)$. 
We will now show that the topological poset $F_1$ of `factorizations of circles'
has an equivalent classifying space to Connes' category $\gL$ of cyclicly ordered sets.
It is well-known that $B\gL \simeq \CP^\infty$ and Igusa has 
constructed cocycles $\gc_k \in C^{2k}(\gL; \IQ)$ on $\gL$ that represent
$c^k \in \IQ[c] = H^*(\CP^\infty;\IQ)$.
Once the necessary identifications are made it is only a matter of correctly
pulling and pushing the cocycles through our equivalences
in order to obtain the desired formulas for $\dlgk_k$.

\subsection[The category of factorizations and Connes' category Lambda]%
{The category of factorizations and Connes' category $\gL$}
\label{subsec:fact-vs-Connes}

In this section we compare the category of factorizations $F_{1,\{\pm1\}}$
in dimension $d=1$ with tangential structure $\gt^{\mrm{or}} =\{\pm1\}$
to Connes' category $\gL$ of cyclic sets.
Concretely, we will construct a zig-zag of topological functors 
\[
    F_1 \xleftarrow{J} F_1^\gd \xrightarrow{D} \mcF_1 
    \xleftarrow{H} \gL
\]
such that $J$ is (almost) a continuous bijection 
that induces an equivalence on classifying spaces, 
$D$ is a level-wise equivalence on nerves,
and $H$ is an equivalence of ordinary categories.
One can show that after taking nerves each of these functor 
induces a weak equivalence in the complete Segal space model structure, 
so that $F_1$ and $\gL$ represent equivalent $(\infty,1)$-categories.
However, we will only show that they all induce equivalences 
on classifying spaces, as that is all we need.

\begin{defn}
    Connes' category $\gL$ has as objects natural numbers $n \ge 1$.
    A morphism $[f]:n \to m$ is represented by a weakly monotone map 
    $[f]:\IZ \to \IZ$
    satisfying $f(x+n) = f(x)+m$ for all $x \in \IZ$.
    Two such maps $f,f':\IZ \to \IZ$ represent the same morphism 
    if and only if there is a $k \in \IZ$ such that $f(x)= f'(x) + km$ 
    for all $x \in \IZ$.
\end{defn}

\begin{rem}\label{rem:gL-big}
    There also a non-skeletal version $\gL^{big}$ of $\gL$ 
    such that $\gL \inj \gL^{big}$ is an equivalence of categorie.
    A \emph{cyclic ordering} on a finite set $A$ is an equivalence class 
    of bijections $\gc:A \cong \{1, \dots, |A|\}$, where two such bijections
    are identified whenever they differ by a cyclic permutation of $\{1,\dots,|A|\}$.
    However, a cyclic morphism $(A,[\gc]) \to (B,[\gd])$ has slightly
    more data than just cyclic order preserving maps $f:A \to B$.
    Namely, when $|f(A)|=1$ one also has to keep track of a ``minimal preimage"
    $a_f \in A$. We will not go into detail here, but the reader may
    check themselves that the forgetful functor $\gL \to \mi{Sets}$ 
    is indeed not faithful.
\end{rem}

The poset $F_{d,\gt}$ is difficult to work with because it is not fibrant.
There is, however, a map $F_{d,\gt} \to \Psi_{d-1,\gt}$
defined by sending $((W,l),t)$ to the preimage $\pr_1^{-1}(t) \cap W$ 
equipped with the induced tangential structure. 
We can use this map to define a version of $F_{d,\gt}$
that is better behaved:
\begin{defn}
    The topological poset $F_{d,\gt}^\gd$ has as underlying space
    $F_{d,\gt} \times_{\Psi_{d-1,\gt}} \gd(\Psi_{d-1,\gt})$ where
    the latter denotes $\Psi_{d-1,\gt}$ with the discrete topology.
    The partial ordering on $F_{d,\gt}^\gd$ is the same as on $F_{d,\gt}$,
    except that we remove the identity morphisms.
\end{defn}

\begin{defn}
    For $W \in \Psi_{d,\gt}$ and $s < t \in \IR$ regular values
    we write:
    \[
        W_{|t} := W \cap (\{t\}\times \cube) 
        \qand
        W_{|[s,t]} := W \cap ([s,t]\times \cube) .
    \]
\end{defn}

\begin{lem}
    The canonical map $J:F_{d,\gt}^\gd \to F_{d,\gt}$ induces 
    a weak equivalence on classifying spaces. 
\end{lem}
\begin{proof}
    Let $F_{d,\gt}'$ denote the non-unital subcategory of 
    $F_{d,\gt}$ containing those objects $(W,t)$ where
    $W$ is cylindrical over $(t-\eps,t+\eps)$ for some $\eps>0$
    and those morphisms $(W,t \le s)$ where $t < s$.
    Similarly, let $F_{d,\gt}^{\gd\prime} \subset F_{d,\gt}^\gd$
    be the non-unital subcategory defined by the same conditions.
    
    It follows from standard rescaling arguments 
    (e.g. \cite[Proof of 3.9]{GRW10}) that the two inclusions
    $F_{d,\gt}' \inj F_{d,\gt}$ and 
    $F_{d,\gt}^{\gd\prime} \inj F_{d,\gt}^\gd$
    induce level-wise weak equivalences on nerves
    and hence weak equivalences on geometric realisations.
    
    The functor $F_{d,\gt}^{\gd\prime} \to F_{d,\gt}'$ is a
    base-change in the sense of lemma \ref{lem:base-change}.
    The relevant map 
    \[
        N(F_{d,\gt}')_n \longrightarrow (\Psi_{d,\gt})^n,
        \quad
        (W, t_0 < \dots < t_n) \mapsto (W_{|t_0}, \dots, W_{|t_n})
    \]
    is a fibration by the same arguments as in the proof of 
    \cite[Proposition 3.2.4(ii)]{ERW19b}.
    Therefore, by the base-change lemma \ref{lem:base-change},
    $F_{d,\gt}^{\gd\prime} \to F_{d,\gt}'$ is a weak
    equivalence on classifying spaces.
    By $2$-out-of-$3$ this implies that
    $F_{d,\gt}^{\gd} \to F_{d,\gt}$ 
    is also a weak equivalence on classifying spaces.
\end{proof}

Specializing to dimension $1$ we now introduce an ordinary category
$\mcF_1$ that will interpolate between $F_1^\gd$ and Connes' $\gL$.

\begin{defn}
    The category $\mcF_1$ has as objects triples $(M,W_0,W_1)$ 
    where $(W_0:\emptyset \to M,W_1: M \to \emptyset) \in N_2(h\Cob_1)$
    is a composable tuple of morphisms in $h\Cob_1$ such that $M$ is non-empty 
    and the composite $W_0 \cup_M W_1$ is a circle.
    A morphism $(M, W_0, W_1) \to (N, V_0, V_1)$
    is a morphism $X:M \to N$ in $h\Cob_1$ such that 
    $W_0 \cup_M X = V_0$ and $W_1 = X \cup_N V_1$.
\end{defn}

\begin{defn}
    Define a functor $P:F_1^\gd \to \mcF_1$ by sending
    $((W,l),t) \in F_{1}^\gd$ to 
    \[
        ([W_{|t_0},l], [W_{|[-1,t_0]},l], [W_{|[t_0,1]},l]),
    \]
    and on morphisms by sending $((W,l),t_0 < t_1)$ 
    to $[W_{[t_0,t_1]},l]$.
\end{defn}        

\begin{lem}
    The canonical functor $F_1^\gd \to \mcF_1$ induces a level-wise 
    equivalence on nerves.
\end{lem}
\begin{proof}
    For every $k$ there is a (non-simplicial) map
    \begin{align*}
        N_k(F_1) &\longrightarrow N_{k+2}(\Cobred_{1}), \\
        ((W,l), t_0 < \dots <t_n) &\mapsto 
        (\emptyset \xrightarrow{W_{|[-1,t_0]}} W_{|t_0} 
        \xrightarrow{W_{|[t_0,t_1]}} W_{|t_1}
        \to \dots 
        \to W_{|t_n} \xrightarrow{W_{|[t_n,1]}} \emptyset).
    \end{align*}
    which is an equivalence onto the connected components it hits.
    This also defines a map for $F_1^\gd$:
    \begin{align*}
        N_k(F_1^\gd) &\longrightarrow N_{k+2}(\gd\Cobred_{1}),
    \end{align*}
    which again is an equivalence onto the connected components it hits.
    
    To prove the comparison with $\mcF_1$ note that the nerve
    of $\mcF_1$ embeds level-wise a subset 
    $N_k(\mcF_1) \inj N_{k+2}(h\Cobred_1)$ in the same way that 
    $N_k(F_1^\gd)$ embeds into $N_{k+2}(\gd\Cobred_1)$.
    Both maps in fact hit the same connected components;
    namely those $((W,l),t_0\le \dots \le t_{k+2})$
    where $W \cong S^1$ and $W_{|t_i} =\emptyset$ iff $i\in \{0,k+2\}$.
    But we already observed in lemma \ref{lem:red1=hred1} that in dimension $1$
    with tangential structure $\gt = \{\pm1\}$ the simplicial space
    $N(\gd\Cobred_1)$ is level-wise equivalent to the simplicial set
    $N(h\Cobred_1)$. This shows that $N_k(F_1^\gd)$ is indeed 
    equivalent to the discrete space $N_k(\mcF_1)$.
\end{proof}

\begin{defn}\label{defn:functor-G}
    For every oriented $0$-manifold $(M,l:M \to \{\pm1\})$ let 
    $M^\pm := \{p \in M \;|\; l(p)= \pm1\}$ denote the set
    of positively or negatively oriented points, respectively. 
\end{defn}

\begin{rem}
    In the following lemma we will construct an equivalence of categories
    $H: \mcF_1 \to \gL$. If we allow ourselves to use the non-skeletal
    $\gL^{big} \supset \gL$ from remark \ref{rem:gL-big},
    the inverse equivalence $\wt{H}:\mcF_1 \to \gL^{big}$ admits 
    a conceptually easier description:
    it sends an object $(M, W_0, W_1) \in \mcF_1$ to the finite 
    set $\pi_0(W_0)$, equipped with the cyclic ordering coming from 
    the fact that $W_0$ is a disjoint union of intervals in an oriented circle 
    $W_0 \subset W_0 \cup_M W_1 \cong S^1$.
    To a morphism $X: (M, W_0, W_1) \to (N, V_0, V_1)$ the functor $\wt{H}$
    assigns the cyclic morphism $\pi_0(W_0) \to \pi_0(W_0 \cup_M X) \cong \pi_0(V_0)$.
    The problem with this description is that, as pointed out in remark \ref{rem:gL-big},
    cyclic morphism have a subtle bit of extra data, 
    which makes it hard to verify functoriality.
    Instead, we will explicitly construct an equivalence between the skeleta.
\end{rem}

\begin{lem}\label{lem:Lambda-F1}
    There is an equivalence of ordinary categories $H:\gL \to \mcF_1$.
\end{lem}
\begin{proof}
    We begin by fixing some notation for this proof.
    For all $n \ge 1$ and $[k] \in \IZ/n$ we choose two points in 
    $\cube$, denoted by $[k]_n^+$ and $[k]_n^-$ such that all of these
    points are disjoint.
    Using these we define oriented $0$-manifolds for all $n \ge 1$ by 
    $M(n) :=\{ [1]_n^+, [1]_n^-, \dots, [n]_n^+, [n]_n^-\} \subset \cube$
    with orientation $l([i]_n^\pm) := \pm 1$.
    Next, define diffeomorphism classes of bordisms 
    $W(n): \emptyset \to M(n)$ and $V(n): M(n) \to \emptyset$ in 
    $h\Cob_1$ by requiring that $W(n)$ is a disjoint union of $n$ intervals
    with boundary $\{[i]_n^+,[i+1]_n^-\}$ and $V(n)$ is a disjoint union
    of $n$ intervals with boundary $\{[i]_n^-,[i]_n^+\}$.
    By construction the glued manifold $W(n) \cup_{M(n)} V(n)$ 
    is a circle and therefore $(M(n),W(n),V(n))$ defines 
    an object of $\mcF_1$. In fact, these
    objects define a skeleton for $\mcF_1$.
    
    Before we begin with the actual proof, we need to understand morphisms
    in the category $h\Cobred_1$. For any $n,m \ge 0$ there is
    a canonical bijection
    \[
        \gs: \Hom_{h\Cobred_1}(M(n), M(m)) 
        \cong \Hom_{\Fin^{bij}}(M(n)^+ \amalg M(m)^-, 
        M(n)^- \amalg M(m)^+).
    \]
    This map sends a bordism $X:M(n) \to M(m)$ to the bijection
    $\gs_X: M(n)^+ \amalg M(m)^- \to  M(n)^- \amalg M(m)^+$
    with $\gs_X(a) = b$ whenever there is an edge in $X$ 
    connecting $a$ and $b$.
    It is possible to implicitly describe $\gs_{Y \cup_{M(n)} X}$ in terms 
    of $\gs_X$ and $\gs_Y$, but we leave this to the reader.

    We want to define a functor $H:\gL \to \mcF_1$ that sends the object
    $n$ to the object $(M(n),W(n),V(n)) \in \mcF_1$.
    To do so, we need to give for every $[f]:n \to m$ in $\gL$
    a bordism $X:M(n) \to M(m)$ such that $W(n)\cup_{M(n)} X = W(m)$
    and $V(n) = X \cup_{M(m)} V(m)$.
    Equivalently, we need to give a bijection 
    $\gs:M(n)^+ \amalg M(m)^- \cong M(n)^- \amalg M(m)^+$
    satisfying certain conditions.

    Fix a representative $f: \IZ \to \IZ$. 
    We define $X$ via $\gs_X$ as
    \begin{align*}
        \gs_X([i]_n^+) &= \begin{cases}
            [i]_n^- & \text{ if } f(i) = f(i+1) \\
            [f(i)]_m^+ & \text{ if } f(i) \neq f(i+1)
        \end{cases}\\
        \gs_X([j]_m^-) &= \begin{cases}
            [j+1]_m^+ & \text{ if } j+1 \not\in f(\IZ) \\
            [k]_n^+ & \text{ if } k+1 = \min(f^{-1}(j+1)).
        \end{cases}
    \end{align*}
    Since $f(x+n) = f(x)+m$ this is well defined on 
    $[i] \in \IZ/n$ and $[j]\in \IZ/m$.
    Moreover, $\gs_X$ does not change if we replace $f$ by 
    $f+m$ and therefore $X$ only depends on the equivalence
    class $[f] \in \Hom_\gL(m,n)$.
    This construction is illustrated in 
    figure \ref{fig:3-easy-examples}.
    One checks by hand that $X$ defines a morphism in $\mcF_1$
    and that the construction is functorial.
    
    As noted before, every object of $\mcF_1$ is isomorphic
    to one of the form $(M(n),W(n),V(n))$,
    and hence the functor $H$ is essentially surjective.
    
\begin{figure}[h]
    \centering
    \tiny
    \def\svgwidth{\linewidth}
    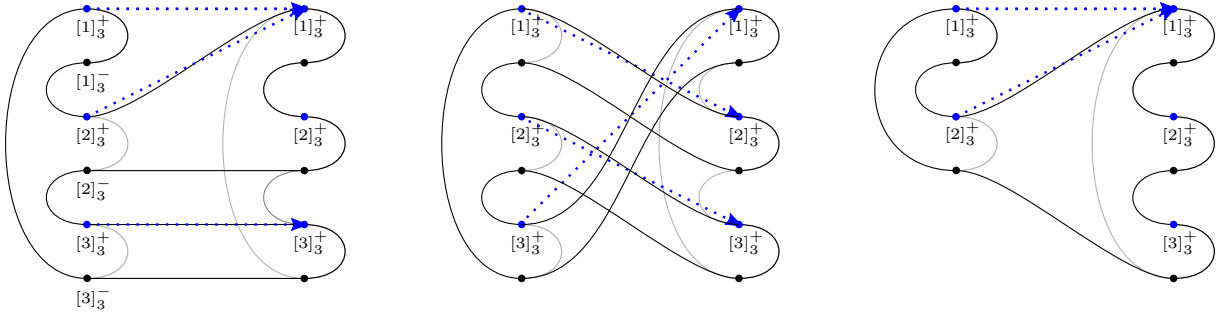
    \caption{Three morphisms 
    $X_i:(M(n_i),W(n_i),V(n_i)) \to (M(m_i),W(m_i),V(m_i))$ in $\mcF_1$ 
    and their associated morphisms $f_i:n_i \to m_i$ in $\gL$.
    The bordisms $W(n_i), X_i, V(m_i)$ are drawn in black,
    the bordisms $V(n_i)$ and $W(m_i)$ are indicated in grey,
    and the maps $f_i:\IZ/m_i \to \IZ/n_i$ are shown in blue.}
    \label{fig:3-easy-examples}
\end{figure}
    
    We still need to show that $H$ is fully faithful.
    To do so, we will construct from a given 
    $X:(M(n),W(n),V(n)) \to (M(m),W(m),V(m))$ a morphism $[f]:n \to m$
    and show that this construction is inverse to the definition of $H$.
    
    Let $A \subset \{1,\dots,n\} \cong M(n)^+$ be the (non-empty!) subset of 
    those points where the relevant edge of $X$ ends in $M(n)^+$.
    For the minimal element $a_{\min}$ of $A$ we pick any value 
    $f(a_{\min}) = j \in \IZ$ such that 
    $\gs_X[a_{\min}]_n^+ = [j]_m^+$.
    The set of such $j$ is of the form $j + m\IZ$.
    For all other $a \in A + n\IZ$ we set $f(a) = j_a$ 
    for the unique $j_a \in \IZ$ with 
    $\gs_X([a]_n^+) = [j_a]_m^+$ and $f(a_{\min}) < j_a < f(a_{\min})+m$.
    Then, for all $i\in \IZ \setminus (A + n\IZ)$ we define 
    $f(i) := f(i+1)$, which makes sense recursively.
    
    The map $f:\IZ \to \IZ$ we constructed satisfies $f(x+n) = f(x)+m$
    by construction. That $f$ is weakly monotone is enforced by the 
    condition that the manifold $W(n) \cup_{M(n)} X \cup_{M(m)} W(m)$
    is a circle. (See figure \ref{fig:3-easy-examples}.)
    We have therefore constructed mutually inverse bijections 
    between the relevant morphisms in $\mcF$ and $\gL$.
\end{proof}

\subsection{Cocycles on the cyclic category}

In this section we recall the description of Igusa's rational $2k$-cocycles 
$\gb_k$ on $N\gL$ that represent the powers of the first Chern class
$c_1 \in H^*(B\gL; \IQ) \cong \IQ[c_1]$.
It will be useful to first define the cocylces on a certain 
simplicial set $\mc{U}$, which admits compatible maps 
\[
    \begin{tikzcd}
        N(F_1) \ar[drr] & N(F_1^\gd) \ar[l, "J"'] \ar[r, "D"]\ar[dr]  &
        N(\mcF_1) \ar[d, "s"] & N\gL \ar[l, "H"'] \ar[dl]  \\
        & & \mc{U}. & 
    \end{tikzcd}
\]
By the previous section the realisations of the top row 
are all equivalent to $\CP^\infty$ and the horizontal maps are equivalences.
We think of $\mc{U}$ as a rational model for $\CP^\infty$ similar
to Kontsevich's combinatorial $BU(1)^{\mrm{comb}}$, 
see \cite[section 2.2]{Kon92}.
However, we will not actually show $\mc{U}$ is rationally
equivalent to $\CP^\infty$. Instead we will only show that 
the maps $B\gL \to \|\mc{U}\|$ are rationally surjective.
Concretely, we define cocycles $\cl{\sign}_{2k}$ on $\mc{U}$
and show that their pullback to the top row is a certain
multiple of $(c_1)^k$.
The advantage of working with $\mc{U}$ is that it is much easier 
to construct the diagram as described above than it is 
to lift Igusa's cocycle against the maps $H$ and $J$, see remark \ref{rem:U-useful}.

\begin{defn}\label{defn:U}
    The $n$-simplices in $\mc{U}$ are represented by $n$-tuples 
    $(A_0, \dots, A_n)$ of finite non-empty disjoint subsets $A_i \subset S^1$.
    We identify two such $n$-tuples $(A_0,\dots,A_n)$ and
    $(B_0,\dots,B_n)$ if there is an orientation 
    preserving diffeomorphism $\gp:S^1 \cong S^1$
    with $\gp(A_i) = B_i$ for all $i$.
    The $i$th face map forgets $A_i$ and the $i$th degeneracy map sends 
    $[A_0, \dots, A_n]$ to $[A_0, \dots, A_i, A_i^\eps, \dots, A_n]$
    where $A_i^\eps$ is obtained from $A_i$ by rotating by a sufficiently
    small angle $\eps>0$.
\end{defn}

\begin{defn}
    We define a simplicial map $s:N\mcF_1 \to \mc{U}$ as follows.
    A $k$-simplex of $N\mcF_1$ can be represented by a $(k+2)$-simplex
    $(M_0 \xrightarrow{[W_1]} M_1 \to \dots \to M_{k+2}) \in N_{k+2}h\Cob_1$
    where $W_1 \cup_{M_1} \dots \cup_{M_{k+1}} W_{k+2}$ is diffeomorphic
    to $S^1$. (In particular $M_0 = \emptyset = M_{k+2}$.)
    To define $s$, pick an orientation preserving diffeomorphism 
    $\gp: W_1 \cup_{M_1} \dots \cup_{M_{k+1}} W_{k+2} \cong S^1$
    and set
    \[
        s(M_0 \xrightarrow{[W_1]} M_1 \to \dots \to M_{k+2})
        := [\gp(M_1^+), \dots, \gp(M_{k+1}^+)]
    \]
    where $M_i^+ \subset M_i$ denotes the subset of 
    positively oriented points.
    The resulting $k$-simplex in $\mc{U}_k$ does not depend on the choice
    of $\gp$ nor on the choice of representatives for morphisms
    $W_i:M_{i-1} \to M_i$. 
\end{defn}

\begin{rem}\label{rem:U-useful}
    Note that given $s:N(\mcF_1) \to \mc{U}$ there is a unique map 
    $s': N(F_1) \to \mc{U}$ such that the diagram above commutes.
    Uniqueness holds since $J$ is surjective except for identity morphisms,
    and to check existence we just observe that the same formula
    \[
        s'(M_0 \xrightarrow{W_1} M_1 \to \dots \to M_{k+2})
        := [\gp(M_1^+), \dots, \gp(M_{k+1}^+)]
    \]
    for some orientation preserving diffeomorphism 
    $\gp:W_1 \cup_{M_1} \dots \cup_{M_{k+1}} W_{k+1}$ 
    is still well-defined and independent of $\gp$.
    In fact, one can check that $\pi_0 N_n(F_1) \to \mc{U}_n$ is a bijection for all $n$,
    and so all discrete cocylces on $N(F_1)$ have to come from $\mc{U}$.
    This is the crucial advantage of using $\mc{U}$:
    it is much easier to construct a simplicial map to $\mc{U}$,
    then a simplicial map $N(F_1) \to N\gL$.
    In fact, any such map has to factor through $\mc{U}$
    because $N\gL$ is level-wise discrete and there seem to be 
    no interesting maps $\mc{U} \to N\gL$.
\end{rem}

We now define the cocycles on $\mc{U}$.
\begin{defn}\label{defn:sign}
    The sign of a $(2k+1)$-tuple of disjoint points 
    $a_0, \dots, a_{2k} \in S^1$ is defined as
    the sign of any permutation $\gs$ of $\{0,\dots,2k\}$ such that the sequence
    $(a_{\gs(0)}, \dots, a_{\gs(2k)})$ is in cyclic order.
    This is well-defined because cyclic permutations on an set with
    $(2k+1)$ elements have sign $+1$.
    
    We extend this definition to disjoint finite subsets 
    $A_0, \dots, A_{2k} \subset S^1$ by averaging:
    \[
        \op{sign}_{2k}(A_0, \dots, A_{2k}) := 
        \frac{1}{\prod_{i=0}^{2k} |A_i|} \sum_{a_0 \in A_0} \dots \sum_{a_{2k} \in A_{2k}}
        \op{sign}(a_0, \dots, a_{2k})
        \in \IQ.
    \]
\end{defn} 

We will also need a reduced version of the averaged sign,
where some summands are omitted:
\begin{defn}\label{defn:reduced-sign}
    Given $(A_0, \dots, A_l)$ as above and $a_i \in A_i$ 
    and $a_j \in A_j$, let $I \subset S^1$ be the positively 
    oriented arc that starts at $a_{\min\{i,j\}}$ and ends 
    at $a_{\max\{i,j\}}$.
    We say $a_i$ and $a_j$ are neighbours if
    $A_k \cap I$ has at most one element for all $k$.
    
    We say an $(l+1)$-tuple 
    $(a_0,\dots,a_l)$ contains neighbours if there are $i\neq j$
    such that $a_i$ and $a_j$ are neighbours.
    Using this the reduced sign is defined as:
    \[
        \cl{\op{sign}}_{2k}(A_0, \dots, A_{2k}) := 
        \frac{1}{\prod_{i=0}^{2k} |A_i|} 
        \sum_{{(a_0, \dots, a_{2k}) \in \prod_i A_i}
        \atop {\text{contains no neighbours}}} 
        \op{sign}(a_0, \dots, a_{2k})
        \in \IQ.
    \]
\end{defn}

\begin{rem}
    Note that, given $(A_0, \dots, A_n)$ as above the notion 
    of being neighbours induces an equivalence relation on the disjoint
    union $\coprod_{i=0}^n A_i$. 
\end{rem}

\begin{lem}\label{lem:sign-cocyc}
    For all $k\ge 1$ the maps
    \[
        \op{sign}_{2k}: \mc{U}_{2k} \to \IQ 
        \qand 
        \cl{\op{sign}}_{2k}: \mc{U}_{2k} \to \IQ
    \]
    are well-defined $2k$-cocycles.
\end{lem}
\begin{proof}
    Observe that $\sign_{2k}(A_0, \dots, A_{2k})$ and
    $\cl{\sign}_{2k}(A_0, \dots, A_{2k})$ do not change if we act
    on $(A_0,\dots,A_{2k})$ by an orientation preserving diffeomorphism,
    as they only depend on the relative cyclic ordering of the elements
    of $\coprod_{i=0}^{2k} A_i$. Therefore $\sign_{2k}$ and 
    $\cl{\sign}_{2k}$ are well-defined cochains on $\mc{U}$.
    
    The rest of the proof will be concerned with proving 
    that they satisfy the cocycle condition. First we will show that 
    for any $(2k+2)$-tuple $(a_0, \dots, a_{2k+1})$ of disjoint points
    in $S^1$ we have:
    \[
        \sum_{i=0}^{2k+1} (-1)^i \op{sign}(a_0, \dots \wh{a_i} \dots, a_{2k+1}) = 0.
    \]
    Indeed, if we exchange two consecutive $a_j$ on the left-hand side,
    then this changes the sign of the overall sum. We therefore have, 
    for any permutation $\gs \in \gS_{2k+2}$:
    \[
        \sum_{i=0}^{2k+1} (-1)^i \op{sign}(a_0, \dots \wh{a_i} \dots, a_{2k+1}) 
        = \op{sign}(\gs) \sum_{i=0}^{2k+1} (-1)^i \op{sign}(a_{\gs(0)}, \dots \wh{a_{\gs(i)}} \dots, a_{\gs(2k+1)}).
    \]
    We can choose $\gs$ such that the resulting tuple 
    $(a_{\gs(0)}, \dots, a_{\gs(2k+1)})$ is in cyclic order,
    in which case it follows that 
    \[
        \op{sign}(\gs) \sum_{i=0}^{2k+1} (-1)^i \op{sign}(a_{\gs(0)}, \dots \wh{a_{\gs(i)}} \dots, a_{\gs(2k+1)})
        = \op{sign}(\gs) \sum_{i=0}^{2k+1} (-1)^i 1 = 0.
    \]
    Let now $A = (A_0, \dots, A_{2k+1})$ represent a $(2k+1)$-simplex in $\mc{U}$.
    We need to show that $\sign_{2k}(\partial A) = 0$.
    Spelling out the definition we have
    \begin{align*}
        \sign_{2k}(\partial A) 
        &= \sum_{i=0}^{2k+1} (-1)^i \sign_{2k}(A_0, \dots \wh{A_i} \dots, A_{2k+1}) \\
        &= \sum_{i=0}^{2k+1} (-1)^i
        \frac{1}{\prod_{j=0, j \neq i}^{2k+1} |A_j|} 
        \sum_{a_0 \in A_0} \dots \wh{\sum_{a_{i} \in A_{i}}}
        \dots \sum_{a_{2k} \in A_{2k}}
        \op{sign}(a_0, \dots, a_{2k}) \\
        &= \sum_{i=0}^{2k+1} (-1)^i
        \frac{1}{\prod_{j=0}^{2k+1} |A_j|} 
        \sum_{a_0 \in A_0} \dots \sum_{a_{2k+1} \in A_{2k+1}}
        \op{sign}(a_0, \dots \wh{a_i} \dots, a_{2k+1}).
    \end{align*}
    In the last step we introduced a sum $\sum_{a_i \in A_i}$ even
    though the variable $a_i$ is not used in the summand.
    This amounts to multiplication by $|A_i|$, cancelling the 
    $\frac{1}{|A_i|}$ that was introduced before the sum.
    We can now rearrange the sum to
    \begin{align*}
        \sign_{2k}(\partial A) 
        &= \frac{1}{\prod_{j=0}^{2k+1} |A_j|} 
        \sum_{a_0 \in A_0} \dots \sum_{a_{2k+1} \in A_{2k+1}}
        \sum_{i=0}^{2k+1} (-1)^i
        \op{sign}(a_0, \dots \wh{a_i} \dots a_{2k+1}).
    \end{align*}
    By the first part of the proof the innermost sum is $0$
    and hence $\sign_{2k}(\partial A) = 0$ as claimed.
    
    It remains to check that the reduced sign $\cl{\sign}_{2k}$
    is a cocycle, too. We can attempt to run the same argument,
    leading us to the expression
    \begin{align*}
        \cl{\sign}_{2k}(\partial A) 
        &= \sum_{i=0}^{2k+1} (-1)^i
        \frac{1}{\prod_{j=0}^{2k+1} |A_j|} 
        \sum_{{(a_0, \dots, a_{2k+1}) \in \prod_j A_j}
        \atop {\text{no neighbours except for }a_i}} 
        \op{sign}(a_0, \dots \wh{a_i} \dots, a_{2k+1}).
    \end{align*}
    The problem with this is that we are also summing over 
    the $(2k+2)$-tuples $(a_0, \dots, a_{2k+1})$
    where $a_i$ is a neighbour of one of the other $a_j$.
    By the argument for $\sign_{2k}$ all other terms sum up to $0$
    and so we are left with:
    \begin{align}\label{align:sign-bar}
        \cl{\sign}_{2k}(\partial A) 
        &= \frac{1}{\prod_{j=0}^{2k+1} |A_j|} 
        \sum_{i=0}^{2k+1} (-1)^i
        \sum_{j\neq i} \sum_{{(a_0, \dots, a_{2k+1}) \in \prod_j A_j}
        \atop {a_i \text{ and }a_j \text{ neighbours}}} 
        \op{sign}(a_0, \dots \wh{a_i} \dots, a_{2k+1}).
    \end{align}
    Note that the $a_j$ is uniquely determined since the tuple
    $(a_0, \dots \wh{a_i} \dots, a_{2k+1})$ is not allowed to 
    contain neighbours and being neighboured is an equivalence
    relation.
    
    The lemma will now follow from the observation that whenever
    $(a_i,a_j)$ is the unique tuple of neighbours in 
    $(a_0,\dots,a_{2k+1})$ then 
    \begin{equation}\label{eqn:unique-neighbours}
        \sign(a_0, \dots \wh{a_i} \dots, a_{2k+1}) 
        = (-1)^{j-i-1} \sign(a_0, \dots \wh{a_j} \dots, a_{2k+1}).
    \end{equation}
    Indeed, once we establish this, one can see that every term 
    in equation \ref{align:sign-bar} appears exactly twice, 
    with opposite sign.
    We only have to check equation \ref{eqn:unique-neighbours}
    in the case $i < j$. Both expressions only depend on the cyclic
    ordering of the set $\{a_0,\dots,a_{2k+1}\} \subset S^1$.
    If we, in the left-hand side of the equation, move the entry $a_j$ 
    to the left by $(j-i-1)$, then this cancels the sign $(-1)^{j-i-1}$
    on the right-hand side.
    Since we assumed that $a_i$ and $a_j$ are neighbours and 
    that they are the only neighbours, there cannot be any $a_l$
    between $a_i$ and $a_j$ on $S^1$. 
    Therefore we have
    \[
        \sign(a_0, \dots, a_{i-1}, a_j, a_{i+1}, \dots \wh{a_j} \dots a_{2k+1})
        =
        \sign(a_0, \dots, a_{i-1}, a_i, a_{i+1}, \dots \wh{a_j} \dots a_{2k+1}).
    \]
    This proves equation \ref{eqn:unique-neighbours} and hence
    completes the proof.
\end{proof}

\begin{proposition}\label{prop:sign=Chern}
    For all $k \ge 1$ the cocycle
    \[
        \gb_k := \frac{(-1)^k k!}{(2k)!} \cdot s^*\cl{\op{sign}}_{2k}
    \]
    represents, possibly up to a factor of $(-1)^k$,
    the cohomology class of $(c_1)^k \in H^*(B\mcF_1; \IQ) \cong \IQ[c_1]$.
\end{proposition}
\begin{proof}
    Consider the subcategory $\gL^{inj} \subset \gL$ where we only allow
    morphisms represented by maps $\IZ \to \IZ$ that are injective.
    In \cite{Igu04} Igusa provides combinatorial formulas for 
    rational cocycles $c_\mcZ^k$ that represent the powers 
    of the first Chern class on $B(\gL^{inj}) \simeq \CP^\infty$.
    He denotes $\gL^{inj}$ by $\mcZ$.
    We will consider the composite
    \[
        q: N(\gL^{inj}) \inj N\gL \xrightarrow{\ N(H)\ } N\mcF_1 
        \xrightarrow{\ s\ } \mc{U}
    \]
    and we will show that the pullback cocycles
    \[
        \left(H^*\gb_k\right)_{|\gL^{inj}} 
        = \frac{(-1)^k k!}{(2k)!} \cdot q^* \cl{\op{sign}}_{2k}
    \]
    agree with the cocycle $c_{\mcZ}^k$ of 
    in \cite[(1) on page 478]{Igu04}.
    Then the claim will follow if we can show that the composite 
    $\gL^{inj} \to \gL \to \mcF_1$ induces an equivalence on 
    classifying spaces.
    This works because every self-equivalence of $\CP^\infty$ induces
    either the identity or $(c_1)^k \mapsto (-c_1)^k$ on cohomology.
    
    We saw in lemma \ref{lem:Lambda-F1} that $\gL \to \mcF_1$ is an
    equivalence of categories and so we only need to 
    show that $\gL^{inj} \inj \gL$ is an
    equivalence on classifying spaces. 
    For this, let $\gL_\infty$ be the paracyclic category,
    which is defined just like $\gL$ except that we 
    do not quotient by the equivalence relation on hom sets.
    There is a free action of the simplicial abelian group
    $N\IZ$ on $N\gL_\infty$ with quotient $N\gL$.
    In \cite[Theorem B.3]{NS18} the authors show that
    $\gL_\infty$ has a contractible classifying space 
    and conclude that $B\gL \simeq \CP^\infty$. 
    The same proof applies to show that $\gL_\infty^{inj}$
    has a contractible classifying space and that hence
    $B\gL^{inj} \simeq \CP^\infty$. Moreover, since the map
    $B\gL_\infty^{inj} \to B\gL_\infty$ is $B\IZ$-equivariant,
    we see that $B\gL^{inj} \to B\gL$ is indeed an equivalence.
    
    Note that the category $\gL^{inj}$ is easier to work with 
    than $\gL$ because the functor $F:\gL^{inj} \to \Sets$ 
    that sends $n$ to $\IZ/n$ and $[f]:n \to m$ to 
    the induced map $[f]:\IZ/n \to \IZ/m$ is faithful. 
    In other words, for an injective map between two cyclic
    sets being cyclic is a \emph{property}, whereas in general
    it is a structure.\footnote{%
        The reason that a general morphism in $\gL$ has more 
        structure than just a map $f:\IZ/n \to \IZ/m$ is because
        it comes with total orderings of the fibers $f^{-1}([i])$
        for all $[i] \in \IZ/m$.
        These total orderings can be recovered from the cyclic 
        ordering if $f(\IZ/n) \subset \IZ/m$ has more then one
        element. If, on the other hand, we fix some 
        $f:\IZ/n \to \IZ/m$ with $f([x]) = [0]$ for all $x$,
        then there are $n$ different cyclic maps $f_{(i)}: n \to m$
        that induce $f$.
    }
    
    The remainder of the proof is concerned with showing that
    the pullback cocycle $(H^*\gb_k)_{|\gL^{inj}}$ agrees
    with Igusa's cocycle $c_\mcZ^k$.
    Using our definition of the sign cocycle (definition \ref{defn:sign})
    we can rewrite $c_\mcZ^k$ as
    \[
        c_{\mcZ}^k(n_0 \xrightarrow{f_1} n_1 \to \dots 
        \xrightarrow{f_{2k}} n_{2k})
        := \frac{(-1)^k k!}{(2k)!} \cdot 
        \frac{1}{\prod_{i=0}^{2k} |A_i|}
        \sum_{{(a_0, \dots, a_{2k}) \in \prod_{i=0}^{2k} A_i}
        \atop {a_i \neq a_j \text{ for all } i \neq j}}
        \sign_{2k}(a_0, \dots, a_{2k})
    \]
    where $A_l$ is the image of the composite map 
    $(f_n \circ \dots \circ f_l): \IZ/n_l \to \IZ/n_{2k}$.
    Here we think of $\IZ/n_{2k}$ as a subset $S^1$ 
    in the usual way. 
    
    We would like to show that this is equal to (a multiple of) the pullback 
    cocycle $q^*\cl{\sign}_{2k}$. Indeed, spelling out the definition
    of the map $q = s \circ N(H):N\gL^{inj} \to \mc{U}$, we see that
    in the notation of lemma \ref{lem:Lambda-F1}
    \begin{align*}
        q^*\cl{\sign}_{2k}(n_0 \xrightarrow{f_1} n_1 \to \dots 
        \xrightarrow{f_{2k}} n_{2k})
        & = \frac{1}{\prod_{i=0}^{2k} n_i}
        \cl{\sign}_{2k}(\gp^{-1}(M(n_0)^+), \dots, 
        \gp^{-1}(M(n_{2k})^+))
    \end{align*}
    Here the $M(n_i)^+$ are subsets of the $1$-manifold
    \[
        X := W(n_0) \cup_{M(n_0)} X_1 \cup_{M(n_1)} \dots 
        \cup_{M(n_{2k})} X_{2k} \cup_{M(n_{2k})} V(n_{2k})
        \stackrel{\gp}{\cong} S^1
    \]
    where $X_i:M(n_{i-1}) \to M(n_i)$ is the image of 
    the morphism $f_i:n_{i-1} \to n_i$ under $H$.
    The map from $M(n_i)^+ \cong \IZ/n_i$ to 
    $M(n_{2k})^+ \cong \IZ/n_{2k}$ sends a point 
    $x \in M(n_i)^+$ to the unique 
    $y \in M(n_{2k})^+$ such that there is a positively oriented
    arc $\gi: [0,1] \inj X$ with $\gi(0) = x$, $\gi(1) = y$,
    and $\gi^{-1}(M(n_{2k})^+) = \{y\}$.
    Since the maps $f_l:n_{l-1} \to n_l$ are all injective,
    this arc automatically satisfies $\gi^{-1}(M(n_{i})^+) = \{x\}$
    as well. (Every element of $\gi^{-1}(M(n_{i})^+)$ would be
    mapped to $y$ by the map 
    $[f_{2k} \circ \dots \circ f_{i+1}]: \IZ/n_i \to \IZ/n_{2k}$,
    but by injectivity $x$ is the only such element.)
   
    Let $x' \in M(n_j)^+$ with $i<j$ then $x$ and $x'$ 
    are neighboured in the sense of definition \ref{defn:reduced-sign}
    with respect to the tuple
    $(\gp^{-1}(M(n_0)^+), \dots, \gp^{-1}(M(n_{2k})^+))$,
    if and only if $x'$ lies on the arc $\gi:[0,1] \inj X$
    that we considered above. In other words, $x$ and $x'$
    are neighboured if and only if their images in 
    $M(n_{2k})^+$ agree. 
    In formulas this means
    \[
        \sum_{{(a_0, \dots, a_{2k}) \in \prod_{i=0}^{2k} A_i}
        \atop {a_i \neq a_j \text{ for all } i \neq j}}
        \sign_{2k}(a_0, \dots, a_{2k})
        = 
        \sum_{{(a_0, \dots, a_{2k}) \in \prod_{i=0}^{2k}
        \gp^{-1}(M(n_i)^+)}
        \atop \text{ no neighbours } } 
        \sign_{2k}(a_0, \dots, a_{2k}) 
    \]
    where we let $A_i$ denote the image of the map 
    $M(n_i)^+ \to M(n_{2k})^+ \inj S^1$ described above. 
    
    After multiplying with the correct coefficient this is implies
    \[
        c_\mcZ^k(n_0\xrightarrow{f_1}\dots\xrightarrow{f_{2k}}n_{2k})
        =  q^*\cl{\sign}_{2k}(n_0 \xrightarrow{f_1} n_1 \to \dots 
        \xrightarrow{f_{2k}} n_{2k})
    \]
    just claimed.
\end{proof}

\begin{rem}
    The proposition proves that
    $\gb_k$ in fact represents the integral
    cohomology class $(c_1)^k \in H^{2k}(B\mcF_1;\IZ)$,
    at least up to a sign. This implies that there is 
    an integral cocycle that is rationally cohomologous to $\gb_k$.
    The reason we prefer to work with the rational cocycle 
    $\gb_k$ instead is that it has the useful property of 
    being \emph{conjugation invariant}: it satisfies for 
    all $i$ and isomorphisms $\gs:n_i \cong n_i$ that 
    \[
        \gb(n_0\xrightarrow{f_1}\dots\xrightarrow{f_{2k}}n_{2k})
        = \gb(n_0\xrightarrow{f_1}\dots
        n_{i-1} \xrightarrow{\gs \circ f_i} n_i 
        \xrightarrow{f_i \circ \gs^{-1}} n_{i+1} \dots
        \xrightarrow{f_{2k}}n_{2k}).
    \]
    This property is extremely convenient when extending $\gb$
    along the zig-zag $N(F_1) \leftarrow N(F_1^\gd) \to N(\mcF_1)$,
    and one should not expect to be able find an integral cocycle with this property.
    Indeed, as noted in remark \ref{rem:U-useful} any cocycle 
    on $N(F_1)$ has to be come from a cocycle on $\mc{U}$ 
    and hence is necessarily conjugation invariant.
\end{rem}

\begin{rem}
    It seems likely that the two cocycles $\sign_{2k}$ 
    and $\cl{\sign}_{2k}$ are in fact cohomologous for all $k \ge 1$.
    If this were true we could use $\sign_{2k}$ in what follows,
    yielding a description of the cocycles representing 
    $\dlgk_k$ on $h\Cob_1$ without having to introduce the notion
    of a neighbour.
    It is however hard to `by hand' guess a $(2k+1)$-cochain
    on $\mc{U}$ whose boundary is $\sign_{2k} - \cl{\sign}_{2k}$.
    Alternatively, one could try to find a $2k$-cycle on $\gL$
    on which both $\sign_{2k}$ and $\cl{\sign}_{2k}$ evaluate
    to the same non-zero number.
\end{rem}

\subsection{Cocycles on the cobordism category}

\begin{defn}\label{defn:kappa}
    For all $k \ge 0$ define a $(2k+2)$-cochain on the 
    simplicial space $\Cut_1$ by the formula
    \[
        \gc_k(M_0 \xrightarrow{W_1}
        M_1 \xrightarrow{W_2} \dots
        \xrightarrow{W_{2k+2}} M_{2k+2})
        := 
        \frac{(-1)^k k!}{(2k)!}
        \sum_{[\gi:S^1 \inj W]}
        \cl{\op{sign}}_{2k}(\gi^{-1}(M_1^+), \dots, 
        \gi^{-1}(M_{2k+2}^+)).
    \]
    Here we write $W$ for the composition 
    $W_1 \cup_{M_1} \dots \cup_{M_{2k+1}} W_{2k+2}$,
    and the sum runs over isotopy classes of oriented
    embeddings $\gi:S^1 \inj W$ such that\footnote{
        In principle we could also omit the condition that 
        $\gi(S^1) \cap M_i \neq \emptyset$, and just use 
        the convention that $\cl{\sign}_{2k}(A_1,\dots,A_{2k+1}) = 0$
        whenever any of the $A_i$ is empty.
    }
    $\gi(S^1)$ intersects $M_i$ for all $1 \le i \le 2k+1$.

    The map $\gc_k:(\Cut_1)_k \to \IQ$ is continuous 
    and hence also well-defined on $\pi_0(\Cut_1)_k$.
    We may therefore pull it back along the compatible maps
    \[
        \begin{tikzcd}
            N(h\Cob_1) \ar[r]\ar[rrrrd, bend right=6] 
            & N(h\Cobred_1) \ar[rrrd, bend right=4] 
            & N(\gd\Cobred_1) \ar[l] \ar[r] \ar[rrd, bend right=2]
            & N(\Cobred_1) \ar[r] \ar[rd] & \Cut_1 \ar[d]\\
            &&&&\pi_0\Cut_1 
        \end{tikzcd}
    \]
    to obtain cocycles on $\Cobred_1$, $h\Cobred_1$, and $h\Cob_1$,
    which we will also denote by $\gc_k$.
\end{defn}
\begin{figure}[ht]
    \centering
    \def\svgwidth{.9\linewidth}
    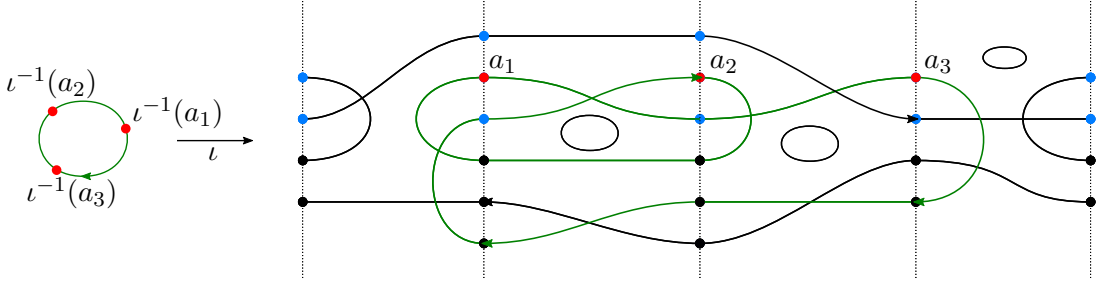
    \caption{One of the summands in the definition of the $4$-cocycle $\gc_1$.
    The figure depicts a choice of oriented embedding 
    $\gi:S^1 \inj W = W_1 \cup W_2 \cup W_3 \cup W_4$
    and positively oriented points $a_i \in M_i^+ \cap \gi(S^1)$.
    In the case shown $\sign(\gi^{-1}(a_1),\gi^{-1}(a_2),\gi^{-1}(a_2)) = -1$,
    because the three points are not in cyclic order on $S^1$.
    In fact, the triple $(a_1,a_2,a_3)$ depicted 
    is the only one with no neighbours, and hence 
    $\cl{sign}_2(\gi^{-1}(M_1^+),\gi^{-1}(M_2^+),\gi^{-1}(M_3^+)) = \frac{-1}{4}$.
    Moreover, the embedding shown is the only one that is allowed,
    and hence $\gc_1$ evaluates on this $4$-simplex as 
    $\frac{(-1)^11!}{2!}\frac{-1}{4} = \frac{1}{4}$.}
    \label{fig:kappa-1}
\end{figure}
When $k=0$ the reduced sign $\cl{\sign}_0(\gi^{-1}(M_1^+))$ 
is always equal to $1$ and therefore $\dlgk_0$ simply counts 
the number of isotopy classes of oriented embeddings 
$\gi:S^1 \inj W_0 \cup_{M_1} W_1$.
Hence $\gc_0$ is equal, as a cocycle on $h\Cobred_1$,
to the $2$-cocycle $\ga$ that we constructed in section \ref{sec:2-cocycle}.

We will now prove our second main theorem, which states 
that the cocycles $\gc_k$ indeed represent the $\dlgk$-classes
on $B(h\Cob_1)$.
\begin{thm}\label{thm:cocycles}
    For $k\ge 0$ the $(2k+2)$-cochains $\gc_k$ 
    defined above are in fact cocycles and,
    the cohomology class $[-\gc_k]$ is,
    possibly up to a sign $(-1)^k$,
    the generator $\dlgk_k$ in 
    \[
        H^*(\|\Cut_1\|; \IQ) \cong \IQ[\dlgk_0, \dlgk_1, \dlgk_2, \dots].
    \]
    As a consequence the same formula also yields 
    well defined cocycles on $B(h\Cobred_1)$ and $B(h\Cob_1)$
    where they still represent the classes $\dlgk_k$ in:
    \[
        H^*(B(h\Cobred_1)_0; \IQ) \cong 
        \IQ[\dlgk_0, \dlgk_1, \dlgk_2, \dots], 
        \qand
        H^*(B(h\Cob_1)_0; \IQ) \cong 
        \IQ[\dlgk_1, \dlgk_2, \dots]. 
    \]
\end{thm}

Before we prove this theorem, we first recall a simple
lemma about primitive cocycles on $H$-spaces.
\begin{lem}\label{lem:primitive-cocycle}
    Let $X$ be a simplicial set and $\mu: X \times X \to X$ a simplicial map 
    that induces an associative H-space structure on $\|X\|$.
    Then every cocycle $\gb \in C^n(X)$ satisifying 
    $\gb(s_i x) = 0$ and $\gb(\mu(x,y)) = \gb(x) + \gb(y)$ 
    for all $x,y \in X_n$ and $i=0,\dots, n$
    represents a primitive element of $H^*(\|X\|)$.
\end{lem}
\begin{proof}
    This is a direct consequence of the definition of the coproduct
    as the composite of the diagonal $\gD^*:C^*(X) \to C^*(X \times X)$
    with the dual of the Eilenberg-Zilber map 
    $C_*(X) \ot C_*(X) \to C_*(X \times X)$.
\end{proof}

\begin{proof}[Proof of theorem \ref{thm:cocycles}]
    In lemma \ref{lem:sign-cocyc} we saw that $\cl{\sign}_{2k}$
    is a cocycle on $\mc{U}$ and using this one can check 
    the cochains $\gc_k$ are cocycles on $\Cut_1$.
    
    We begin by showing that $[\gc_k]$ is a rational multiple 
    of $\dlgk_k$. To do so it will suffice to show that it is
    primitive with respect to Hopf algebra structure.
    By construction the cocycles $\gc_k$ are additive under disjoint
    union and hence satisfy the conditions of lemma 
    \ref{lem:primitive-cocycle} with respect to the multiplication
    $\amalg: \Cut_1 \times \Cut_1 \to \Cut_1$.
    Therefore the classes $[\gc_k] \in H^{2k+2}(\Cut_1)$ 
    are indeed primitive.
    By lemma \ref{lem:free-on-con} the infinite loop space $\|\Cut_1\|$ is 
    freely generated by $\|\Cut_1^{con}\|$ and therefore the pullback 
    along the inclusion induces an isomorphism
    \[
        \mrm{Prim}(H^*(\|\Cut_1\|; \IQ)) 
        \cong H^*(\|\Cut_1^{con}\|; \IQ)
        = \IQ\gle{\dlgk_1, \dlgk_2, \dots}
    \]
    In order to prove the theorem it will hence suffice to show
    that the restriction of $[\gc_k]$ to 
    $\Cut_1^{con} \subset \Cut_1$ represents the class $-\dlgk_k$.
    
    In lemma \ref{lem:identify-cofib-seq} and 
    section \ref{subsec:fact-vs-Connes}
    we established weak equivalences
    \[
        \|\Cut_1^{con}\| \simeq \gS^2 B(F_1)
        \qand
        B(F_1) \simeq B(F_1^\gd) \simeq B(\mcF_1) \simeq B(\gL)
        \simeq \CP^\infty.
    \] 
    Each of the spaces on the right admits a map to the simplicial
    set $\mc{U}$ from definition \ref{defn:U} and by proposition 
    \ref{prop:sign=Chern} the Chern class $(c_1)^k$ is
    represented by the pullback of the cocycle
    $\frac{(\pm1)^k k!}{(2k)!}\cl{\sign}_{2k}$.
    In fact, using the simplicial double suspension 
    of definition \ref{defn:simplicial-cone}, we have a simplicial map 
    \[
        q: \Cut_1^{con} \to \gS^2 \mc{U}, \quad 
        (M_0 \xrightarrow{W_1}
        M_1 \xrightarrow{W_2} \dots
        \xrightarrow{W_n} M_n)
        \mapsto
        \left((\gp^{-1}(M_{b-a}^+), \dots, 
        \gp^{-1}(M_b^+)), a, b\right)
    \]
    where $\gp:S^1 \cong W_0 \cup_{M_1} 
    \dots \cup_{M_{2k+1}} W_{2k}$
    is any orientation preserving diffeomorphism
    and $0\le a \le b\le n$ are maximal such that 
    $\gp(S^1)$ intersects $M_{b-a}$ and $M_a$ non-trivially.
    The geometric realisation of $q$ is the double suspension 
    of the map $B(F_1) \to \|\mc{U}\|$.
    
    The pullback of (the double suspension of)
    $\frac{(\pm1)^k k!}{(2k)!}\cl{\sign}_{2k}$ from 
    $\gS^2 \mc{U}$ to $\Cut_1^{con}$ is exactly $(\gc_k)_{|\Cut_1^{con}}$.
    Therefore $\gc_k$ indeed represents the kappa class
    $\dlgk_k \in H^{2k+2}(\Cut_1)$, possibly up to a sign.
    
    We can't really hope to keep track of all the 
    signs that were introduced by the equivalences we used.
    Instead, we observe that the relative sign of the $\dlgk_k$
    is always the same (up to the $(-1)^k$ ambiguity from proposition
    \ref{prop:sign=Chern}).
    It therefore suffices to check the sign for one of the $\dlgk_k$.
    In lemma \ref{lem:ga-is-kappa0} we showed $-[\ga] = \dlgk_0$,
    and since the $2$-cocycles $\ga$ and $\gc_0$ are equal
    on $h\Cobred_1$ this implies  $-[\gc_0] = \dlgk_0$.
    In summary, we have $-[\gc_k] = (\pm1)^k \dlgk$ for a
    global choice of $\pm1$.
\end{proof}

\printbibliography[heading=bibintoc]

\end{document}